\documentclass[11pt]{amsart}

\usepackage{cite,amsthm,amsfonts,amsmath,amscd,amssymb,mathrsfs,enumitem,eucal}
\usepackage{accents, graphicx}
\usepackage[usenames]{color}
\usepackage[colorlinks=true, linkcolor=webgreen, citecolor=webgreen, urlcolor=webbrown]{hyperref}
\usepackage[initials]{amsrefs}
\usepackage{todonotes,comment}

\definecolor{webgreen}{rgb}{0,.4,0}
\definecolor{webbrown}{rgb}{.4,0,0}

\newtheorem{Thm}{Theorem}[section]
\newtheorem{thm}{Theorem}
\newtheorem{Def}[Thm]{Definition}

\newtheorem{Lm}[Thm]{Lemma}
\newtheorem{Prop}[Thm]{Proposition}
\newtheorem{Cor}[Thm]{Corollary}
\newtheorem{Cnj}[Thm]{Conjecture}

\newtheorem*{convention}{Convention}

\theoremstyle{definition}
\newtheorem*{ack}{Acknowledgements}

\theoremstyle{remark}

\newtheorem{Rem}[Thm]{Remark}
\newtheorem{Exp}[Thm]{Example}

\numberwithin{equation}{section}

\newcommand{\DDot}[1]{\ddot{\rm #1}}
\newcommand{\DDDot}[1]{\dddot{\rm #1}}
\newcommand{\fH}[1]{\text{\rm \H{#1}}}
\newcommand{\eala}[1]{\mathsf{#1}}
\newcommand{\tquest}{??}

\def\<{\langle}
\def\>{\rangle}

\DeclareMathOperator{\GL}{GL}
\DeclareMathOperator{\SL}{SL}

\DeclareMathOperator{\Aut}{Aut}
\DeclareMathOperator{\Out}{Out}

\DeclareMathOperator{\Hom}{Hom}

\DeclareMathOperator{\stab}{stab}

\DeclareMathOperator{\Real}{Re}
\DeclareMathOperator{\Imag}{Im}

\DeclareMathOperator{\ord}{ord}
\DeclareMathOperator{\rank}{rank}

\def\Re{\mathbb R}
\def\Rat{\mathbb Q}
\def\F{\mathbb F}
\def\Co{\mathbb C}

\def\Z{\mathbb Z}

\def\a{\alpha}

\def\b{\beta}

\def\d{\delta}

\def\th{\theta}  
\def\thp{\tilde{\th}}  
\def\Thetap{\widetilde{\Theta}} 
\def\varThetap{\widetilde{\varTheta}} 
\def\ph{\varphi} 
\def\php{\tilde{\ph}} 
\def\Phip{\widetilde{\Phi}} 
\def\varPhip{\widetilde{\varPhi}} 
\def\eps{\varepsilon}
\def\l{\lambda}
\def\L{\Lambda}
\def\isP{\rotatebox[origin=c]{180}{$\Psi$}}
\def\varisP{\rotatebox[origin=c]{180}{$\varPsi$}}

\def\A{\mathcal A} 
\def\Atilde{\widetilde{\A}} 
\def\Htilde{\widetilde{\mathcal H}} 
\def\B{\mathbf B} 
\def\Cox{\mathbf C} 
\def\C{\mathcal C} 
\def\T{\mathcal T} 
\def\bH{\mathbf H} 
\def\O{\mathcal O} 
\def\Q{\mathcal Q} 
\def\P{\mathcal P} 

\def\Gaff{\mathfrak g} 
\def\G{\mathring{\mathfrak g}} 
\def\Haff{\mathfrak h} 
\def\Hring{\mathring {\mathfrak h}} 
\def\mfa{\mathfrak{a}} 
\def\mfb{\mathfrak{b}} 
\def\mfc{\mathfrak{c}} 
\def\mfe{\mathfrak{e}} 
\def\mfs{\mathfrak{s}} 
\def\mfr{\mathfrak{r}} 
\def\mfu{\mathfrak{u}} 
\def\mfv{\mathfrak{v}} 

\def\Aring{\mathring A} 
\def\W{\mathring W} 
\def\Wdaff{\widetilde{W}} 
\def\w{\mathring w} 
\def\Rring{\mathring R} 
\def\Pring{\mathring P} 
\def\Qring{\mathring Q} 

\def\wGamma{\widetilde{\Gamma}_1} 
\def\wXi{\widetilde{\Xi}_1} 
\def\wUpsilon{\widetilde{\Upsilon}_1} 

\def\Ti{\varTheta_0} 
\def\Tiii{\varPhi_0} 

\begin{document}
\title[]
{Double affine Hecke algebras and congruence groups}
\author[]{Bogdan Ion}
\address{Department of Mathematics, University of Pittsburgh, Pittsburgh, PA 15260}
\email{bion@pitt.edu}
\author[]{Siddhartha Sahi}
\address{Department of Mathematics, Rutgers University, New Brunswick, NJ 08903}
\email{sahi@math.rutgers.edu}
\thanks{Work partially supported by Simons Foundation grants 420882 (B.I.) and 509766 (S.S.) and by the National Science Foundation grant DMS-162350 (S.S.).}
\date{October 10, 2017}%
\begin{abstract}
The most general construction of double affine Artin groups (DAAG) and Hecke algebras (DAHA) associates such objects to pairs of compatible reductive group data. We show that DAAG/DAHA \emph{always} admit a faithful action by automorphisms of a finite index subgroup of the Artin group of type $A_{2}$, which descends to a faithful outer action of a congruence subgroup of $\SL(2,\mathbb{Z})$ or $\mathrm{PSL}(2,\mathbb{Z})$. This was previously known only in some special cases and, to the best of our knowledge, not even conjectured to hold in full generality.

It turns out that the structural intricacies of DAAG/DAHA are captured by the underlying \emph{semisimple} data and, to a large extent, even by \emph{adjoint} data; we prove our main result by reduction to the adjoint
case. Adjoint DAAG/DAHA correspond in a natural way to affine Lie algebras, or more precisely to their affinized Weyl groups, which are the semi-direct products $W\ltimes Q^{\vee}$ of the Weyl group $W$ with the coroot lattice
$Q^{\vee}$. They were defined topologically by van der Lek, and independently, algebraically, by Cherednik. We now describe our results for the adjoint case in greater detail.

We first give a new Coxeter-type presentation for adjoint DAAG as quotients of the Coxeter braid groups associated to certain crystallographic diagrams that we call double affine Coxeter diagrams. As a consequence we show that the rank two Artin groups of type $A_{2},B_{2},G_{2}$ act by automorphisms on the adjoint DAAG/DAHA associated to affine Lie algebras of twist number $r=1,2,3$, respectively. This extends a fundamental result of Cherednik for $r=1$.

We show further that the above rank two Artin group action descends to an outer action of the congruence subgroup $\Gamma_{1}(r)$. In particular, $\Gamma_{1}( r) $ acts naturally on the set of isomorphism classes of representations of an adjoint DAAG/DAHA of twist number $r$, giving rise to a projective representation of $\Gamma_{1}( r) $ on the space of a $\Gamma_{1}( r) $-stable representation. We also provide a classification of the involutions of Kazhdan-Lusztig type that appear in the context of these actions. 
\end{abstract}

\maketitle
\newpage
\tableofcontents

\section{Introduction}

\subsection{}\label{sec: 1.1}  Over the past two decades, double affine Hecke algebras have had a significant impact on an impressive number of mathematical fields, from mathematical physics, special functions and combinatorics, to topology, geometry, representation theory and harmonic analysis. They are objects of Lie theoretic nature, the most general construction being attached to a pair of compatible reductive group data. However, as we explain in \S\ref{sec: reductive-DABG}, \S\ref{sec: reductiveHecke}, they are isomorphic to the almost-direct product of the objects attached to the semisimple and the central part of the data. This reduces the study of their structure to the objects attached to semisimple data.

As part of the double reductive group data one must specify a pair $R_1$, $R_2$ of finite root systems with the same underlying Coxeter diagram. There are two possibilities: $R_1$ and $R_2$ are of the same Dynkin type $X_n$ or of dual (but not equal) Dynkin type, say $X_n$ and $X_n^\vee$, respectively. We call the former double reductive data \emph{untwisted} double data of type $X_n$ (or simply double reductive data of type $X_n^u$) and the latter \emph{twisted} double data of type $X_n$ (or simply double reductive data of type $X_n^t$). There are finitely many double semisimple data for fixed $R_1$ and $R_2$ and this finite set can be partially ordered in a natural fashion. Mirroring the terminology for semisimple Lie groups, we call the largest double semisimple data \emph{simply connected} and the smallest \emph{adjoint}.
 
One of the most profound results in the theory is the emergence of the modular group as a group of outer automorphisms for the double affine Hecke algebras associated to \emph{untwisted simply connected double data} \cite{CheMac}*{Theorem 4.3} or \emph{untwisted adjoint double data} \cite{ISTri}*{\S4.2} (the two results overlap in type $E_8, F_4$ and $G_2$)\footnote{Our use of the words \emph{twisted} and \emph{untwisted} is the same as that in \cites{CheMac} and dual to the one on \cites{ISTri}.}. There is an overtone in this construction: a central extension of the modular group, more precisely, the braid group of type $A_2$, also known as the braid group on three strands, acts as a group of automorphisms and this action descents to the outer action of the modular group. The importance of this phenomenon can be hardly overstated although its implications have not yet been fully explored. To only briefly allude to its role let us mention that Cherednik's difference Fourier transform is given by the action of the element $\left[\begin{smallmatrix}0 & 1\\ -1 & 0\end{smallmatrix} \right]$ of the modular group, and this has played the key role in his solution of Macdonald's evaluation-duality conjecture \cite{CheMac} as well as in subsequent developments in the harmonic analysis in the context of double affine Hecke algebras (e.g. \cite{CheDou-3}). Other important applications are concerned with Rogers-Ramanujan type identities attached to affine root systems \cite{CFRog} and the topology of torus knots \cites{CheJon, GNRef}.

One of our main results is that \emph{all} double affine Hecke algebras have a \emph{congruence group} of outer automorphisms that in fact arises from  a central extension of the congruence group acting as automorphisms. Aside from the two situations mentioned above, this result is new. We emphasize that even among untwisted semisimple data there are many examples of data that are neither simply connected nor adjoint. Furthermore, \emph{all} double affine Hecke algebras arise naturally as Hecke algebras associated to (extended) affine Kac-Moody groups over local fields \cite{BGR}*{\S 7.9}.

Nevertheless, as it turns out, the crucial case that needs to be considered is that of adjoint data. We focus our discussion here and for most of the paper on the objects attached to adjoint data which we simply refer to as double affine Hecke algebras and, in essence, to the underlying groups which we refer to as double affine Artin groups. These objects correspond in a natural way to (and canonically arise from) \emph{affine Lie algebras} such that untwisted adjoint double data correspond to untwisted affine Lie algebras and twisted adjoint double data correspond to twisted affine Lie algebras. The congruence groups and their central extensions associated to each double affine Artin group can be more efficiently indicated if we use the labeling of double affine Artin groups by affine Lie algebras. The labelling convention that we adopt in this paper is identical to the one in \cites{CheDou-2, CheDou, CheMac} and dual to the one in \cites{CheDou-1, MacAff, HaiChe, IonNon, IonInv, ISTri, IonSta, StoKoo-2}.  

We prove that, if $\Gaff$ is an affine Lie algebra listed in table Aff $r$  in \cite{KacInf}*{pg. 54-55} ($r=1,2,3$) then the corresponding double affine Artin group and Hecke algebra admits an action as automorphisms of the Coxeter braid group of type $A_2$ (for $r=1$), $B_2$ (for $r=2$), or $G_2$ (for $r=3$). These actions give rise to actions as outer automorphisms for the congruence groups $\Gamma_1(r)$, where  $\Gamma_1(1)=\rm{SL}(2,\Z)$  and, for $r=2,3$,
$$\Gamma_{1}(r) =\left\{ \begin{bmatrix}a&b\\ c&d \end{bmatrix}\in \SL(2,\Z)~\Bigg\vert~  \begin{bmatrix}a&b\\ c&d \end{bmatrix}= \begin{bmatrix}1& * \\ 0&1 \end{bmatrix} (\text{mod}~r)\right\}.$$
The fact that the Coxeter braid groups of type $A_2$, $B_2$, and $G_2$ are central extensions of the congruence groups $\Gamma_1(r)$, for $r=1,2,3$, respectively,  seems to be new for the last two cases. The results of \cite{ISTri}*{\S4.2} correspond here to the case $r=1$. As we will explain later on, one can construct natural projective representations of these congruence groups that arise from simple $\Gamma$-stable representations of double affine Artin groups. We also provide the classification of the involutions of Kazhdan-Lusztig type, which we call basic involutions, that appear in the context of the congruence group actions. 

The objects attached to simply connected semisimple data are the so-called \emph{extended} double affine Artin groups and Hecke algebras, which are in fact the objects that are most often considered in the literature.  We prove that the above braid group and congruence group actions are also valid for the extended double affine Artin groups and Hecke algebras. The results of \cite{CheMac}*{Theorem 4.3} correspond here to the case $r=1$. Finally, we show that the braid group and congruence group actions are valid for the objects attached to arbitrary reductive data, after potentially restricting to a finite index subgroup. 

The crucial role in our proof is played by a new Coxeter-type presentation for double affine Artin groups, as quotients of the Coxeter braid groups associated to certain crystallographic, non-positive-semi-definite diagrams that we call double affine Coxeter diagrams. Since the double affine Weyl groups are not Coxeter groups (see e.g. Appendix \ref{sec: nonCoxeter}) these quotients are, of course, proper but we are able to identify the relations in the kernel explicitly. As  an immediate consequence we reveal that there are redundancies \emph{and} omissions on the list of double affine Artin groups considered in the literature (e.g. in \cite{MacAff}). Furthermore, we obtain a complete list of double affine Hecke algebras, identifying in particular the generic Hecke algebras  associated to the affine root systems $A_{2n}^{(2)}$, $(BC_n, C_n)$, and $(C_n^\vee, BC_n)$ which were not considered before in the literature. Those associated to the affine root systems $(B_n, B_n^\vee)$ and  $(C_2, C_2^\vee)$ have only recently been considered in \cite{StoKoo-2}. We also indicate a conjectural presentation of the extended double affine Artin groups that is akin to our Coxeter-type presentation of the double affine Artin groups. 

In \cite{ISEALA} we have established a connection between $\eala{2}$-extended affine Lie algebras ($\eala{2}$-EALAs) and double affine Artin groups (and therefore double affine Hecke algebras). More precisely, the Artin groups associated to $\eala{2}$-EALAs are (up to isomorphism) precisely the double affine Artin groups. The congruence group  (but not the braid group) actions on double affine Artin groups canonically arise in the context of Artin groups associated to $\eala{2}$-EALAs.

In what follows we describe our results in more detail.

  
\subsection{Main results}

We emphasize that aside from the theory of affine Lie algebras, affine root systems also appear as relative root systems in the theory of reductive $p$-adic groups. A classification of all these root systems is due to Macdonald \cite{MacAff-1}. It turns out that the reduced affine root systems that appear in this context are precisely the root systems of affine Lie algebras, but there are also five infinite families of nonreduced affine root systems. The full list can be consulted in \cite{MacAff-1} and \cite{MacAff}*{\S1.3}. As we point out in \S\ref{sec: nonreduced}, the definition of a double affine Artin group depends only on the set of non-multipliable roots in the affine root system in question and this is always a reduced affine root system. Therefore, as far as double affine Artin groups are concerned  it is enough to consider reduced affine root systems. The double affine Hecke algebras that correspond to affine root systems that have the same set of non-multipliable roots are quotients of the group algebra of the same double affine Artin group modulo ideals generated by quadratic relations. The distinguishing feature is the specialization of the parameters that appear in the quadratic relations, which reflect the orbit structure of the affine root system under the action of the affine Weyl group.

Our first main result gives a Coxeter-type presentation for double affine Artin groups which is probably of independent interest. As we already mentioned, the double affine Artin groups are not Coxeter braid groups. However, we are able to give a presentation for double affine Artin groups as quotients of the Coxeter braid groups  associated to the Coxeter diagrams that we will describe in what follows. The double affine Weyl groups will admit, of course, a corresponding presentation as quotients of the Coxeter groups associated to the same diagrams.

For the purpose of this paper, we will use the term (crystallographic) Coxeter diagram to refer to a graph with at most four edges between any pair of nodes and with some nodes marked with a dot  as in Figure \ref{marked}. The marked nodes will be called affine nodes, whereas the unmarked nodes will be called finite nodes. 

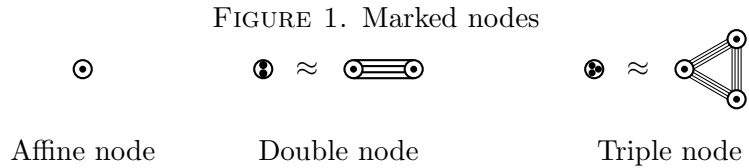
\begin{figure}[ht]
\centering
\caption{Marked nodes}\label{marked}
\begin{tikzpicture}[scale=.4]
	\draw[thick] (0,0) circle [radius=3mm];
	\draw[thick, fill] (0,0) circle [radius=.8mm];
	\draw (0,-2) node[anchor=north]  {Affine node};
	\draw[xshift=6cm, thick] (0,0) circle [radius=3mm];
        \draw[xshift=6cm, thick,fill] (90:1.4mm) circle [radius=1mm];
        \draw[xshift=6cm, thick,fill] (-90:1.4mm) circle [radius=1mm];
        \draw[xshift=9cm, thick] (0,0) circle [radius=3mm];
	\draw[xshift=9cm, thick] (0: 2) circle [radius=3mm];
	\draw[xshift=9cm, thick, fill] (0,0) circle [radius=.8mm];
	\draw[xshift=9cm, thick, fill] (0: 2) circle [radius=.8mm];
	\draw[xshift=9cm, thick] (90: .3) -- +(0: 2);
	\draw[xshift=9cm, thick] (20: .3) -- +(0:1.43); 
	\draw[xshift=9cm, thick] (-20: .3) -- +(0:1.43); 
	\draw[xshift=9cm, thick] (-90: .3) -- +(0: 2); 
	\draw (8.5,-2) node[anchor=north]  {Double node};
	\draw (6.75,0) node[anchor=west]  {$\approx$};
	\draw[xshift=17cm, thick] (0,0) circle [radius=3mm];
	\draw[xshift=17cm, thick,fill] (0:1.4mm) circle [radius=.8mm];
        \draw[xshift=17cm, thick,fill] (120:1.4mm) circle [radius=.8mm];
        \draw[xshift=17cm, thick,fill] (-120:1.4mm) circle [radius=.8mm];
	\draw[xshift=20cm, thick] (0,0) circle [radius=3mm];
	\draw[xshift=20cm, thick] (30:2) circle [radius=3mm];
	\draw[xshift=20cm, thick] (-30:2) circle [radius=3mm]; 
	\draw[xshift=20cm, thick, fill] (0,0) circle [radius=.8mm];
	\draw[xshift=20cm, thick, fill] (30:2) circle [radius=.8mm];
	\draw[xshift=20cm, thick, fill] (-30:2) circle [radius=.8mm]; 
	\draw[xshift=20cm] (60: .3) -- +(30:1.49); 
	\draw[xshift=20cm] (40: .3) -- +(30:1.42); 
	\draw[xshift=20cm] (20: .3) -- +(30:1.42); 
	\draw[xshift=20cm] (0: .3) -- +(30:1.49);
	\draw[xshift=20cm] (0: .3) -- +(-30:1.49);
	\draw[xshift=20cm] (-20: .3) -- +(-30:1.42); 
	\draw[xshift=20cm] (-40: .3) -- +(-30:1.42); 
	\draw[xshift=20cm] (-60: .3) -- +(-30:1.49); 
	\draw[xshift=20cm] (-30:2)++(120: .3) -- ++(90:1.49); 
	\draw[xshift=20cm] (-30:2)++(100: .3) -- ++(90:1.42); 
	\draw[xshift=20cm] (-30:2)++(80: .3) -- ++(90:1.42); 
	\draw[xshift=20cm] (-30:2)++(60: .3) -- ++(90:1.49);  
	\draw (19.5,-2) node[anchor=north]  {Triple node};
	\draw (17.75,0) node[anchor=west]  {$\approx$};
\end{tikzpicture}
\end{figure}

We consider a particular set of Coxeter diagrams that we call double affine Coxeter diagrams. Each of these diagrams has either two or three affine nodes. When the two or three affine nodes of a diagram  are pairwise connected with four edges, and each connects to the finite part of diagram in the same fashion, we \emph{abbreviate} the diagram by replacing the two or three affine nodes by  a \emph{double node} and, respectively, a \emph{triple node}, which connects to the finite part of the diagram in the same manner. A double or  triple node is depicted as a node marked with two or three dots as in Figure \ref{marked}. In Figure \ref{doublenodediag} and Figure \ref{triplenodediag} we include examples of diagrams that use this abbreviation.

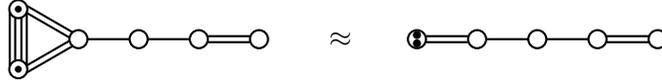
\begin{figure}[ht]
\centering
\caption{Double node abbreviation}\label{doublenodediag}
\bigskip
\begin{tikzpicture}[scale=.4]
   \foreach \x in {1,...,4}
   	\draw[xshift=\x cm,thick] (13.25cm+\x cm,0) circle [radius=3mm];
    	\draw[xshift=13.25cm, thick] (0,0) circle [radius=3mm];
    	\draw[xshift=13.25cm, thick,fill] (90:1.4mm) circle [radius=1mm];
        \draw[xshift=13.25cm, thick,fill] (-90:1.4mm) circle [radius=1mm];
    \foreach \y in {1.15, 2.15}
    	\draw[xshift=13.25cm+\y cm,thick] (\y cm,0) -- +(1.4 cm,0);
    	\draw[xshift=13.25cm, thick] (6.3 cm, .1 cm) -- +(1.4 cm,0);
    	\draw[xshift=13.25cm, thick] (6.3 cm, -.1 cm) -- +(1.4 cm,0);
    	\draw[xshift=13.25cm, thick] (.3 cm, 1mm) -- +(1.4 cm,0);
    	\draw[xshift=13.25cm, thick] (.3 cm, -1mm) -- +(1.4 cm,0);
	
	\draw (10,0) node[anchor=west]  {$\approx$};
  
	\draw[thick] (-90: 1) circle [radius=3mm];
	\draw[thick] (90: 1) circle [radius=3mm];
	\draw[thick, fill] (-90: 1) circle [radius=.8mm];
	\draw[thick, fill] (90: 1) circle [radius=.8mm];
	\draw[thick] (-90: 1)++(20: .3) -- +(90: 1.8);
	\draw[thick] (-90: 1)++(70: .3) -- +(90:1.43); 
	\draw[thick] (-90: 1)++(110: .3)-- +(90:1.43); 
	\draw[thick] (-90: 1)++(160: .3) -- +(90: 1.8); 
	   \foreach \x in {1,...,4}
   	\draw[xshift=\x cm,thick] (\x cm,0) circle [radius=3mm];
    \foreach \y in {1.15, 2.15}
    	\draw[xshift=\y cm,thick] (\y cm,0) -- +(1.4 cm,0);
    	\draw[thick] (6.3 cm, .1 cm) -- +(1.4 cm,0);
    	\draw[thick] (6.3 cm, -.1 cm) -- +(1.4 cm,0);
	\draw[xshift=2cm, thick] (130: .3) -- +(150:1.77); 
	\draw[xshift=2cm, thick] (170: .3) -- +(150:1.66); 
	\draw[xshift=2cm, thick] (190: .3) -- +(210:1.66); 
	\draw[xshift=2cm, thick] (230: .3) -- +(210: 1.77);
\end{tikzpicture}

\end{figure}

\begin{figure}[ht]
\centering
\caption{Triple node abbreviation}\label{triplenodediag}
\bigskip
\begin{tikzpicture}[scale=.4]
	 \foreach \x in {-2.45}
	\draw[xshift=\x cm, thick, white] (\x cm,0) circle [radius=3mm];
	\draw[thick] (0,0) circle [radius=3mm];
	\draw[thick] (150:2) circle [radius=3mm];
	\draw[thick] (-150:2) circle [radius=3mm]; 
	\draw[thick, fill] (0,0) circle [radius=.8mm];
	\draw[thick, fill] (150:2) circle [radius=.8mm];
	\draw[thick, fill] (-150:2) circle [radius=.8mm]; 
	\draw (120: .3) -- +(150:1.49); 
	\draw (140: .3) -- +(150:1.42); 
	\draw (160: .3) -- +(150:1.42); 
	\draw (180: .3) -- +(150:1.49);
	\draw (180: .3) -- +(-150:1.49);
	\draw (-160: .3) -- +(-150:1.42); 
	\draw (-140: .3) -- +(-150:1.42); 
	\draw (-120: .3) -- +(-150:1.49); 
	\draw (-150:2)++(120: .3) -- ++(90:1.49); 
	\draw (-150:2)++(100: .3) -- ++(90:1.42); 
	\draw (-150:2)++(80: .3) -- ++(90:1.42); 
	\draw (-150:2)++(60: .3) -- ++(90:1.49);  
 \foreach \x in {1,...,4}
    	\draw[thick,xshift=\x cm] (\x cm,0) circle [radius=3mm];
        \foreach \y in {1,3}
    	\draw[thick,xshift=\y cm] (\y cm,0) ++(.3 cm, 0) -- +(14 mm,0);
	\draw (0,0) ++(.3 cm, 0) -- +(14 mm,0);
    	\draw[thick] (4.3 cm, 1mm) -- +(1.4 cm,0);
    	\draw[thick] (4.3 cm, -1mm) -- +(1.4 cm,0);
	\draw (-150:2)++(-15: .3) -- ++(17.4: 3.32); 
	\draw (150:2)++(15: .3) -- ++(-17.4: 3.32); 
	
\draw (10,0) node[anchor=west]  {$\approx$};

\foreach \x in {1,...,4}
   	\draw[thick,xshift=13.25cm+ \x cm] (\x cm,0) circle [radius=3mm];
    	\draw[xshift=13.25cm,thick] (0,0) circle [radius=3mm];
        \draw[xshift=13.25cm, thick,fill] (0:1.4mm) circle [radius=.8mm];
        \draw[xshift=13.25cm, thick,fill] (120:1.4mm) circle [radius=.8mm];
        \draw[xshift=13.25cm, thick,fill] (-120:1.4mm) circle [radius=.8mm];
  	\foreach \x in {5,...,7}
    	\draw[white,xshift=13.25cm+\x cm] (\x cm,0) circle [radius=3mm];
        \foreach \y in {0,1,3}
    	\draw[thick,xshift=13.25cm + \y cm] (\y cm,0) ++(.3 cm, 0) -- +(14 mm,0);
    	\draw[thick, xshift=13.25cm] (4.3 cm, 1mm) -- +(1.4 cm,0);
    	\draw[thick, xshift=13.25cm] (4.3 cm, -1mm) -- +(1.4 cm,0);
\end{tikzpicture}

\end{figure}
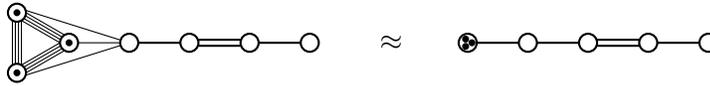

There are three kinds of double affine Coxeter diagrams, listed in Figure \ref{dddot-diagrams}, Figure \ref{ddot-diagrams}, and Figure \ref{fH-diagrams}, which we refer to as \emph{triple dot}, \emph{double dot} and \emph{umlaut} diagrams, respectively.  The umlaut diagrams have two affine nodes, while the double dot and triple dot  diagrams have a double and triple node respectively.  By convention, $\DDDot{C}_1$ will refer to the diagram $\DDDot{A}_1$ and $\DDot{C}_1$ will refer to the diagram $\DDot{A}_1$. 

We emphasize that in any cycle of a double Coxeter diagram there is an even number of double edges and an even number of triple edges. Therefore, the diagrams satisfy the hypothesis of \cite{HumRef}*{Proposition 6.6} and the associated Coxeter groups are crystallographic with respect to their reflection representation.

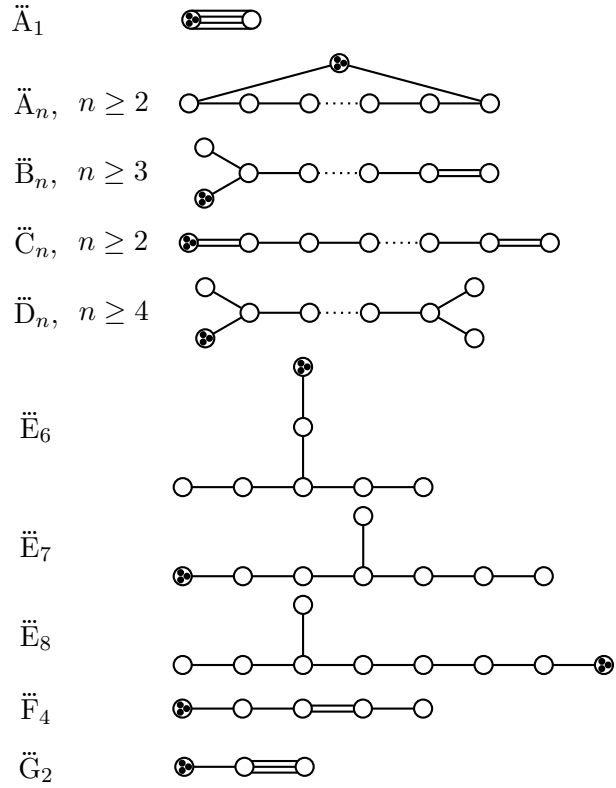
\begin{figure}[ht]
\caption{Triple dot diagrams}\label{dddot-diagrams}
\begin{center}
  \begin{tikzpicture}[scale=.4]
    \draw (-4.5,0) node[anchor=east]  {$\DDDot{A}_1$};
    \draw[xshift=2 cm,thick] (0,0) circle [radius=3mm];
    \draw[thick] (0,0) circle [radius=3mm];
    	\draw[thick,fill] (0:1.4mm) circle [radius=.8mm];
        \draw[thick,fill] (120:1.4mm) circle [radius=.8mm];
        \draw[thick,fill] (-120:1.4mm) circle [radius=.8mm];
     \foreach \x in {2,...,7}
    \draw[white,xshift=\x cm] (\x cm,0) circle [radius=3mm];
    \draw[thick] (90: .3) -- +(0: 2);
    \draw[thick] (-90: .3) -- +(0: 2); 
    \draw[thick] (20: .3) -- +(0:1.43); 
    \draw[thick] (-20: .3) -- +(0:1.43); 
  \end{tikzpicture}
\end{center}

\begin{center}
  \begin{tikzpicture}[scale=.4]
    \draw (-4,0) node[anchor=east]  {$\DDDot{A}_n,$};
    \draw (-1,0) node[anchor=east]  {$n\geq 2$};
    \foreach \x in {0,...,5}
    \draw[xshift=\x cm,thick] (\x cm,0) circle [radius=3mm];
     \foreach \x in {6,7}
    \draw[white,xshift=\x cm] (\x cm,0) circle [radius=3mm];
    \draw[xshift=0 cm,thick] (15: 51.75 mm) circle [radius=3mm];
    	\draw[shift=(15: 51.75 mm), thick,fill] (0:1.4mm) circle [radius=.8mm];
        \draw[shift=(15: 51.75 mm), thick,fill] (120:1.4mm) circle [radius=.8mm];
        \draw[shift=(15: 51.75 mm), thick,fill] (-120:1.4mm) circle [radius=.8mm];
    \draw[dotted,thick] (4.3 cm,0) -- +(1.4 cm,0);
    \foreach \y in {0.15,1.15,3.15, 4.15}
    \draw[xshift=\y cm,thick] (\y cm,0) -- +(1.4 cm,0);
    \draw[xshift=0 cm,thick] (15: 3 mm) -- (15: 48.75 mm);
    \draw[xshift=10 cm,thick] (165: 3 mm) -- (165: 48.75 mm);
  \end{tikzpicture}
\end{center}

\begin{center}
  \begin{tikzpicture}[scale=.4]
    \draw (-4,0) node[anchor=east]  {$\DDDot{B}_n,$};
    \draw (-1,0) node[anchor=east]  {$n\geq 3$};
    \draw[xshift=2 cm,thick] (150: 17 mm) circle [radius=3mm];
    \draw[xshift=2 cm,thick] (-150: 17 mm) circle [radius=3mm];
    	\draw[xshift=2 cm, thick,fill] (-150: 17 mm)+(0:1.4mm) circle [radius=.8mm];
        \draw[xshift=2 cm, thick,fill] (-150: 17 mm)+(120:1.4mm) circle [radius=.8mm];
        \draw[xshift=2 cm, thick,fill] (-150: 17 mm)+(-120:1.4mm) circle [radius=.8mm];
    \foreach \x in {1,...,5}
    \draw[xshift=\x cm,thick] (\x cm,0) circle [radius=3mm];
         \foreach \x in {6,7}
    \draw[white,xshift=\x cm] (\x cm,0) circle [radius=3mm];
    \foreach \y in {1.15,3.15}
    \draw[xshift=\y cm,thick] (\y cm,0) -- +(1.4 cm,0);
    \draw[xshift=2 cm,thick] (150: 3mm) -- +(150: 1.1 cm);
    \draw[xshift=2 cm,thick] (-150: 3mm) -- +(-150: 1.1 cm);
    \draw[dotted,thick] (4.3 cm,0) -- +(1.4 cm,0);
    \draw[thick] (8.3 cm, .1 cm) -- +(1.4 cm,0);
    \draw[thick] (8.3 cm, -.1 cm) -- +(1.4 cm,0);
  \end{tikzpicture}
\end{center}

\begin{center}
  \begin{tikzpicture}[scale=.4]
    \draw (-4,0) node[anchor=east]  {$\DDDot{C}_n,$};
    \draw (-1,0) node[anchor=east]  {$n\geq 2$};
    \foreach \x in {1,...,6}
    \draw[xshift=\x cm,thick] (\x cm,0) circle [radius=3mm];
         \foreach \x in {7}
    \draw[white,xshift=\x cm] (\x cm,0) circle [radius=3mm];
    \draw[thick] (0,0) circle [radius=3mm];
    	\draw[thick,fill] (0:1.4mm) circle [radius=.8mm];
        \draw[thick,fill] (120:1.4mm) circle [radius=.8mm];
        \draw[thick,fill] (-120:1.4mm) circle [radius=.8mm];
    \foreach \y in {1.15, 2.15, 4.15}
    \draw[xshift=\y cm,thick] (\y cm,0) -- +(1.4 cm,0);
    \draw[dotted,thick] (6.3 cm,0) -- +(1.4 cm,0);
    \draw[thick] (10.3 cm, .1 cm) -- +(1.4 cm,0);
    \draw[thick] (10.3 cm, -.1 cm) -- +(1.4 cm,0);
    \draw[thick] (.3 cm, 1mm) -- +(1.4 cm,0);
    \draw[thick] (.3 cm, -1mm) -- +(1.4 cm,0);
  \end{tikzpicture}
\end{center}

\begin{center}
  \begin{tikzpicture}[scale=.4]
    \draw (-4,0) node[anchor=east]  {$\DDDot{D}_n,$};
    \draw (-1,0) node[anchor=east]  {$n\geq 4$};
    \draw[xshift=2 cm,thick] (150: 17 mm) circle [radius=3mm];
    \draw[xshift=2 cm,thick] (-150: 17 mm) circle [radius=3mm];
    	\draw[xshift=2 cm,thick,fill] (-150: 17 mm)+(0:1.4mm) circle [radius=.8mm];
        \draw[xshift=2 cm,thick,fill] (-150: 17 mm)+(120:1.4mm) circle [radius=.8mm];
        \draw[xshift=2 cm,thick,fill] (-150: 17 mm)+(-120:1.4mm) circle [radius=.8mm];
    \foreach \x in {1,...,4}
    \draw[xshift=\x cm,thick] (\x cm,0) circle [radius=3mm];
    \draw[xshift=8 cm,thick] (30: 17 mm) circle [radius=3mm];
    \draw[xshift=8 cm,thick] (-30: 17 mm) circle [radius=3mm];
         \foreach \x in {6,7}
    \draw[white,xshift=\x cm] (\x cm,0) circle [radius=3mm];
    \draw[dotted,thick] (4.3 cm,0) -- +(1.4 cm,0);
    \draw[xshift=2 cm,thick] (150: 3mm) -- +(150: 1.1 cm);
    \draw[xshift=2 cm,thick] (-150: 3mm) -- +(-150: 1.1 cm);
    \foreach \y in {1.15,3.15}
    \draw[xshift=\y cm,thick] (\y cm,0) -- +(1.4 cm,0);
    \draw[xshift=8 cm,thick] (30: 3 mm) -- (30: 14 mm);
    \draw[xshift=8 cm,thick] (-30: 3 mm) -- (-30: 14 mm);
  \end{tikzpicture}
\end{center}

\begin{center}
  \begin{tikzpicture}[scale=.4]
    \draw (-4,2) node[anchor=east]  {$\DDDot{E}_6$};
    \foreach \x in {0,...,4}
    \draw[thick,xshift=\x cm] (\x cm,0) circle [radius=3mm];
    \draw[thick,xshift=4 cm, yshift=2 cm] (0,0) circle [radius=3mm];
    \draw[thick,xshift=4 cm, yshift=4 cm] (0,0) circle [radius=3mm];
        	\draw[xshift=4cm, yshift=4cm, thick,fill] (0:1.4mm) circle [radius=.8mm];
        \draw[xshift=4cm, yshift=4cm, thick,fill] (120:1.4mm) circle [radius=.8mm];
        \draw[xshift=4cm, yshift=4cm, thick,fill] (-120:1.4mm) circle [radius=.8mm];
         \foreach \x in {7}
    \draw[white,xshift=\x cm] (\x cm,0) circle [radius=3mm];
    \foreach \y in {0,...,3}
    \draw[thick,xshift=\y cm] (\y cm,0) ++(.3 cm, 0) -- +(14 mm,0);
    \foreach \t in {0,1}
    \draw[thick,xshift=4 cm, yshift=\t cm] (0, \t cm) ++(0, .3 cm) -- +(0,14 mm);
  \end{tikzpicture}
\end{center}

\begin{center}
  \begin{tikzpicture}[scale=.4]
    \draw (-4,1) node[anchor=east]  {$\DDDot{E}_7$};
    \foreach \x in {1,...,6}
    \draw[thick,xshift=\x cm] (\x cm,0) circle [radius=3mm];
    \draw[thick] (0,0) circle [radius=3mm];
        \draw[thick,fill] (0:1.4mm) circle [radius=.8mm];
        \draw[thick,fill] (120:1.4mm) circle [radius=.8mm];
        \draw[thick,fill] (-120:1.4mm) circle [radius=.8mm];
         \foreach \x in {7}
    \draw[white,xshift=\x cm] (\x cm,0) circle [radius=3mm];
    \foreach \y in {0,...,5}
    \draw[thick,xshift=\y cm] (\y cm,0) ++(.3 cm, 0) -- +(14 mm,0);
    \draw[thick] (6 cm,2 cm) circle [radius=3mm];
    \draw[thick] (6 cm, 3mm) -- +(0, 1.4 cm);
  \end{tikzpicture}
\end{center}

\begin{center}
  \begin{tikzpicture}[scale=.4]
    \draw (-4,1) node[anchor=east]  {$\DDDot{E}_8$};
    \foreach \x in {0,...,6}
    \draw[thick,xshift=\x cm] (\x cm,0) circle [radius=3mm];
    \draw[thick,xshift=14 cm] (0,0) circle [radius=3mm];
        \draw[xshift=14cm, thick,fill] (0:1.4mm) circle [radius=.8mm];
        \draw[xshift=14cm, thick,fill] (120:1.4mm) circle [radius=.8mm];
        \draw[xshift=14cm, thick,fill] (-120:1.4mm) circle [radius=.8mm];
    \foreach \y in {0,...,6}
    \draw[thick,xshift=\y cm] (\y cm,0) ++(.3 cm, 0) -- +(14 mm,0);
    \draw[thick] (4 cm,2 cm) circle [radius=3mm];
    \draw[thick] (4 cm, 3mm) -- +(0, 1.4 cm);
  \end{tikzpicture}
\end{center}

\begin{center}
  \begin{tikzpicture}[scale=.4]
    \draw (-4,0) node[anchor=east]  {$\DDDot{F}_4$};
        \foreach \x in {1,...,4}
    \draw[thick,xshift=\x cm] (\x cm,0) circle [radius=3mm];
    \draw[thick] (0,0) circle [radius=3mm];
        \draw[thick,fill] (0:1.4mm) circle [radius=.8mm];
        \draw[thick,fill] (120:1.4mm) circle [radius=.8mm];
        \draw[thick,fill] (-120:1.4mm) circle [radius=.8mm];
         \foreach \x in {5,...,7}
    \draw[white,xshift=\x cm] (\x cm,0) circle [radius=3mm];
        \foreach \y in {0,1,3}
    \draw[thick,xshift=\y cm] (\y cm,0) ++(.3 cm, 0) -- +(14 mm,0);
    \draw[thick] (4.3 cm, 1mm) -- +(1.4 cm,0);
    \draw[thick] (4.3 cm, -1mm) -- +(1.4 cm,0);
  \end{tikzpicture}
\end{center}

\begin{center}
  \begin{tikzpicture}[scale=.4]
    \draw (-4,0) node[anchor=east]  {$\DDDot{G}_2$};
            \foreach \x in {1,2}
    \draw[thick,xshift=\x cm] (\x cm,0) circle [radius=3mm];
    \draw[thick] (0,0) circle [radius=3mm];
         \draw[thick,fill] (0:1.4mm) circle [radius=.8mm];
        \draw[thick,fill] (120:1.4mm) circle [radius=.8mm];
        \draw[thick,fill] (-120:1.4mm) circle [radius=.8mm];
         \foreach \x in {4,...,7}
    \draw[white,xshift=\x cm] (\x cm,0) circle [radius=3mm];
            \foreach \y in {0,1}
    \draw[thick,xshift=\y cm] (\y cm,0) ++(.3 cm, 0) -- +(14 mm,0);
    \draw[xshift=2cm, thick] (40: 3mm) -- +(1.545 cm, 0);
    \draw[xshift=2cm, thick] (-40: 3 mm) -- +(1.545 cm, 0);
  \end{tikzpicture}
\end{center}
\end{figure}

\begin{figure}[ht]
\caption{Double dot diagrams}\label{ddot-diagrams}
\begin{center}
  \begin{tikzpicture}[scale=.4]
    \draw (-4,0) node[anchor=east]  {$\DDot{A}_1$};
    \draw[xshift=2 cm,thick] (0,0) circle [radius=3mm];
    \draw[thick] (0,0) circle [radius=3mm];
    	\draw[thick,fill] (90:1.4mm) circle [radius=1mm];
        \draw[thick,fill] (-90:1.4mm) circle [radius=1mm];
     \foreach \x in {2,...,7}
    \draw[white,xshift=\x cm] (\x cm,0) circle [radius=3mm];
    \draw[thick] (90: .3) -- +(0: 2);
    \draw[thick] (-90: .3) -- +(0: 2); 
    \draw[thick] (20: .3) -- +(0:1.43); 
    \draw[thick] (-20: .3) -- +(0:1.43); 
  \end{tikzpicture}
\end{center}

\begin{center}
  \begin{tikzpicture}[scale=.4]
    \draw (-4,0) node[anchor=east]  {$\DDot{C}_n,$};
    \draw (-1,0) node[anchor=east]  {$n\geq 2$};
    \foreach \x in {1,...,6}
    \draw[xshift=\x cm,thick] (\x cm,0) circle [radius=3mm];
         \foreach \x in {7}
    \draw[white,xshift=\x cm] (\x cm,0) circle [radius=3mm];
    \draw[thick] (0,0) circle [radius=3mm];
    	\draw[thick,fill] (90:1.4mm) circle [radius=1mm];
        \draw[thick,fill] (-90:1.4mm) circle [radius=1mm];
    \foreach \y in {1.15, 2.15, 4.15}
    \draw[xshift=\y cm,thick] (\y cm,0) -- +(1.4 cm,0);
    \draw[dotted,thick] (6.3 cm,0) -- +(1.4 cm,0);
    \draw[thick] (10.3 cm, .1 cm) -- +(1.4 cm,0);
    \draw[thick] (10.3 cm, -.1 cm) -- +(1.4 cm,0);
    \draw[thick] (.3 cm, 1mm) -- +(1.4 cm,0);
    \draw[thick] (.3 cm, -1mm) -- +(1.4 cm,0);
  \end{tikzpicture}
\end{center}
\end{figure}
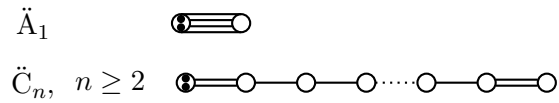

\begin{figure}[ht]
\caption{Umlaut diagrams}\label{fH-diagrams}
\begin{center}
  \begin{tikzpicture}[scale=.4]
    \draw (-5,0) node[anchor=east]  {$\fH{B}_n/\fH{C}_n,$};
    \draw (-2,0) node[anchor=east]  {$n\geq 3$};
    \draw[thick] (45: 14.15 mm) circle [radius=3mm];
    \draw[thick] (-45: 14.15 mm) circle [radius=3mm];
    \draw[thick, fill] (0,0) circle [radius=.8mm];
    \draw[thick, fill] (-45: 14.15 mm) circle [radius=.8mm];
    \foreach \x in {0,...,5}
    \draw[xshift=\x cm,thick] (\x cm,0) circle [radius=3mm];
         \foreach \x in {6,7}
    \draw[white,xshift=\x cm] (\x cm,0) circle [radius=3mm];
    \foreach \y in {1.15,3.15}
    \draw[xshift=\y cm,thick] (\y cm,0) -- +(1.4 cm,0);
    \draw[xshift=2 cm,thick] (135: 3mm) -- +(135: .815 cm);
    \draw[xshift=2 cm,thick] (-135: 3mm) -- +(-135: .815 cm);
    \draw[thick] (65: 3mm) -- +(45: .85 cm);
    \draw[thick] (25: 3mm) -- +(45: .85 cm);
    \draw[thick] (-25: 3mm) -- +(-45: .85 cm);
    \draw[thick] (-65: 3mm) -- +(-45: .85 cm);
    \draw[dotted,thick] (4.3 cm,0) -- +(1.4 cm,0);
    \draw[thick] (8.3 cm, .1 cm) -- +(1.4 cm,0);
    \draw[thick] (8.3 cm, -.1 cm) -- +(1.4 cm,0);
\end{tikzpicture}
\end{center}

\begin{center}
  \begin{tikzpicture}[scale=.4]
    \draw (-5,0) node[anchor=east]  {$\fH{B}_2/\fH{C}_2$};
    \foreach \x in {45,135,225,315}
    \draw[thick, xshift=1cm] (\x: 1.4cm) circle [radius=3mm];
    \draw[thick, xshift=1cm, fill] (-45:1.4) circle [radius=.8mm];
    \draw[thick, xshift=1cm, fill] (-135:1.4) circle [radius=.8mm];
         \foreach \x in {6,7}
    \draw[white,xshift=\x cm] (\x cm,0) circle [radius=3mm];
    \draw[xshift=1cm, shift=(135:1.4cm), thick] (20: 3mm) -- +(1.42 cm, 0);
    \draw[xshift=1cm, shift=(135:1.4cm), thick] (-20: 3 mm) -- +(1.42 cm, 0);
    \draw[xshift=1cm, shift=(135:1.4cm), thick] (-70: 3 mm) -- +(0,-1.42 cm);
    \draw[xshift=1cm, shift=(135:1.4cm), thick] (-110: 3 mm) -- +(0,-1.42 cm);
    \draw[xshift=1cm, shift=(-135:1.4cm), thick] (20: 3mm) -- +(1.42 cm, 0);
    \draw[xshift=1cm, shift=(-135:1.4cm), thick] (-20: 3 mm) -- +(1.42 cm, 0);
    \draw[xshift=1cm, shift=(45:1.4cm), thick] (-70: 3 mm) -- +(0,-1.42 cm);
    \draw[xshift=1cm, shift=(45:1.4cm), thick] (-110: 3 mm) -- +(0,-1.42 cm);
  \end{tikzpicture}
\end{center}

\begin{center}
  \begin{tikzpicture}[scale=.4]
    \draw (-6,0) node[anchor=east]  {$\fH{F}_4$};
    \foreach \x in {0,60,120,180,240,300}
    \draw[thick] (\x: 2cm) circle [radius=3mm];
    \draw[thick, fill] (-60:2) circle [radius=.8mm];
    \draw[thick, fill] (-120:2) circle [radius=.8mm];  
         \foreach \x in {6}
    \draw[white,xshift=\x cm] (\x cm,0) circle [radius=3mm];
    \draw[xshift=-2cm, thick] (60:.3) -- +(60:1.4);
    \draw[xshift=-2cm, thick] (-60:.3) -- +(-60:1.4);
    \draw[xshift=2cm, thick] (-120:.3) -- +(-120:1.4);
    \draw[xshift=2cm, thick] (120:.3) -- +(120:1.4);
    \draw[shift=(120:2cm), thick] (.3, 1mm) -- +(1.4cm,0);
    \draw[shift=(120:2cm), thick] (.3, -1mm) -- +(1.4cm,0);
    \draw[shift=(-120:2cm), thick] (.3, 1mm) -- +(1.4cm,0);
    \draw[shift=(-120:2cm), thick] (.3, -1mm) -- +(1.4cm,0);
  \end{tikzpicture}
\end{center}

\begin{center}
  \begin{tikzpicture}[scale=.4]
    \draw (-4,0) node[anchor=east]  {$\fH{G}_2$};
    \foreach \x in {45,135,225,315}
    \draw[thick, xshift=1cm] (\x: 1.4cm) circle [radius=3mm];
    \draw[thick, xshift=1cm, fill] (-45:1.4) circle [radius=.8mm];
    \draw[thick, xshift=1cm, fill] (-135:1.4) circle [radius=.8mm];
         \foreach \x in {6,7}
    \draw[white,xshift=\x cm] (\x cm,0) circle [radius=3mm];
    \draw[xshift=1cm, shift=(135:1.4cm), thick] (30: 3mm) -- +(1.46 cm, 0);
    \draw[xshift=1cm, shift=(135:1.4cm), thick] (0: 3 mm) -- +(1.38 cm, 0);
    \draw[xshift=1cm, shift=(135:1.4cm), thick] (-30: 3 mm) -- +(1.46 cm, 0);
    \draw[xshift=1cm, shift=(135:1.4cm), thick] (-90: 3 mm) -- +(0,-1.38 cm);
    \draw[xshift=1cm, shift=(-135:1.4cm), thick] (30: 3mm) -- +(1.46 cm, 0);
    \draw[xshift=1cm, shift=(-135:1.4cm), thick] (0: 3 mm) -- +(1.38 cm, 0);
    \draw[xshift=1cm, shift=(-135:1.4cm), thick] (-30: 3 mm) -- +(1.46 cm, 0);
    \draw[xshift=1cm, shift=(45:1.4cm), thick] (-90: 3 mm) -- +(0,-1.38 cm);
  \end{tikzpicture}
\end{center}
\end{figure}
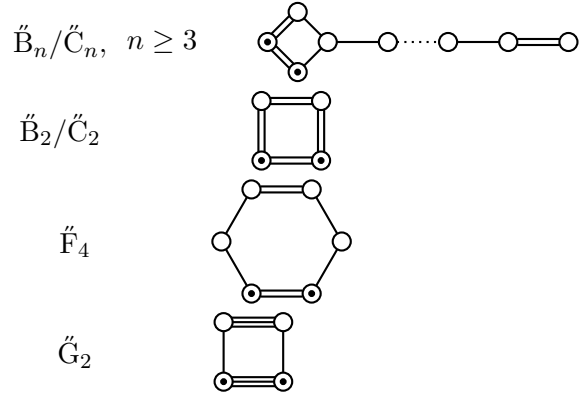

A canonical group associated to a Coxeter diagram $D$ is the Coxeter braid group $B(D)$, defined as the group generated by a set of elements, one for for each node in $D$, that satisfy the braid relations specified by the number of edges between nodes. We consider the Coxeter braid groups $B(\DDDot{X}_n)$, $B(\fH{X}_n)$, and $B(\DDot{C}_n)$. The marked nodes play, of course, no role in the definition of the corresponding Coxeter braid groups but they do play a role in the definition of some quotients, that we denote by $\B(\DDDot{X}_n)$, $\B(\fH{X}_n)$, and $\B(\DDot{C}_n)$. The precise relations in the kernel are specified in Definition \ref{def: untwisted}, \ref{def: twisted}, and \ref{def: twodot}, respectively. We only note here that, aside from $\B(\DDDot{C}_n)$ and $\B(\DDot{C}_n)$, the kernel of the canonical projection is generated by a relation imposing the centrality of one explicit element (see \eqref{centralrel-1}, \eqref{centralrel}). We refer to this as the central relation.  For  $\B(\DDDot{C}_n)$, $n\geq 2$, which in fact plays a distinguished role in the theory,  the kernel is generated by the central relation and a second relation \eqref{ellbraid} which we call the elliptic braid relation. The group $\B(\DDot{C}_n)$, $n\geq 1$ is a quotient of $\B(\DDDot{C}_n)$. 

\begin{thm}\label{state2}
The double affine Artin groups are isomorphic to the groups $\B(\DDDot{X}_n)$, $\B(\fH{X}_n)$, and $\B(\DDot{C}_n)$. The precise correspondence is the following 
\begin{equation}\label{DCvsDA}
\begin{aligned}
&\DDDot{X}_n && {X}_n^{(1)}\\
&\fH{B}_n/ \fH{C}_n && D_{n+1}^{(2)},  A_{2n-1}^{(2)}&& n\geq 2\\
&\fH{F}_4 && E_6^{(2)}\\
&\fH{G}_2 && D_4^{(3)}\\
&\DDot{C}_n && A_{2n}^{(2)} && n\geq 1.
\end{aligned}
\end{equation}
\end{thm}

For the double affine Artin groups associated to untwisted affine root systems Theorem \ref{state2} is proved in \cite{ISTri}*{Theorem 3.11}. For completeness, we include in this paper a streamlined version of the argument.

It is perhaps important to mention that a considerable part of the existing literature on double affine Hecke algebras considers extended double affine Hecke algebras. One cannot expect a topological interpretation or a Coxeter-type presentation for extended double affine Artin groups as these aspects fail even for extended affine Artin groups. Furthermore, just as the extended affine Artin groups and extended affine Hecke algebras can be recovered as semidirect products of affine Artin groups and affine Hecke algebras with a group induced from affine Dynkin diagram automorphisms, the extended double affine Artin groups and Hecke algebras can be recovered in a similar fashion.

One consequence of Theorem \ref{state2} is that it readily identifies the number of possible independent Hecke parameters for each double affine Artin group. We call the Hecke algebras that depend on the largest possible number of independent parameters generic Hecke algebras. Just as for any other Coxeter braid algebra, this number is the number of  connected components of the diagram obtained from the Coxeter diagram by erasing all multiple edges. For each double affine Coxeter diagram this number of connected components is recorded in Table \ref{table: generic-parameters}.

\begin{table}[ht]
\caption{Maximal number of Hecke parameters}\label{table: generic-parameters}
\begin{tabular}{l l l l l l l}
$\DDDot{A}_1$ && 4  &&  $\DDot{A}_1$ && 3\\
$\DDDot{A}_n, n\geq 2$ && 1 &&  $\DDot{C}_n, n\geq 2$ && 4 \\
$\DDDot{B}_n, n\geq 3$ && 2  &&  &&  \\
$\DDDot{C}_n, n\geq 2$ && 5  && $\fH{B}_n/\fH{C}_n, n\geq 3$   && 3\\
$\DDDot{D}_n, n\geq 4$ && 1   && $\fH{B}_2/\fH{C}_2$   && 4\\
$\DDDot{E}_n, n=6,7,8$ && 1  && $\fH{F}_4$  && 2\\ 
$\DDDot{F}_4$  && 2 && $\fH{G}_2$  && 2\\ 
$\DDDot{G}_2$ && 2 && && \\
\end{tabular}
\end{table}

The Hecke algebras are defined as quotients of the $\F$-group algebras of $\B(\DDDot{X}_n)$, $\B(\fH{X}_n)$, and $\B(\DDot{C}_n)$ (where $\F$ is the field of parameters) by imposing quadratic relations on the generators associated to the nodes in the double Coxeter diagrams. The generic Hecke algebras are denoted by $\bH(\DDDot{X}_n)$, $\bH(\fH{X}_n)$, and $\bH(\DDot{C}_n)$. Using this definition and Theorem \ref{state2},  we derive in Theorem \ref{thm: cherednik-presentation} the conventional presentation of double affine Hecke algebras \cites{MacAff, StoKoo-2}. The double affine Hecke algebras associated to (reduced or nonreduced) affine root systems are obtained from $\bH(\DDDot{X}_n)$, $\bH(\fH{X}_n)$, and $\bH(\DDot{C}_n)$ by setting Hecke parameters equal according to the structure of affine Weyl group orbits on the affine root system in question. Table \ref{table: parameters} lists this specialization of parameters for each affine root system. We note that the generic double affine Hecke algebras associated to the affine root systems $A_{2n}^{(2)}$, $(BC_n, C_n)$, $(C_n^\vee, BC_n)$ (and until recently \cite{StoKoo-2} $(B_n, B_n^\vee)$, $(C_2, C_2^\vee)$) were previously missing from the literature.

Let the group $\Gamma$ be  of the form $\Gamma_1(r)$. Let $\widetilde{\Gamma}$ be the inverse image of $\Gamma\subset \SL(2,\Re)$ inside its universal cover $\widetilde{\SL(2,\Re)}$. The group $\wGamma(r)$ is isomorphic to the Coxeter braid group of type $A_2$, $B_2$, or $G_2$, if $r=1,2$ or $3$, respectively. Therefore, it is generated by two elements, $\mfa$ and $\mfb$, that satisfy the $r$-braid relation. To the double affine Artin group $\B$ such that the corresponding  affine root system appears in Table Aff $r$ in \cite{KacInf}*{pg. 54-55} we associate the group $\Gamma=\Gamma_1(r)$.  Using Theorem \ref{state2} we define an explicit action of $\mfa$ and $\mfb$ on $\B$ as group automorphisms such that the following holds.

\begin{thm}\label{state3}
There is a faithful action  $\widetilde{\Gamma}\to \Aut(\B)$ that descends to a morphism $\Gamma\to \Out(\B)$.
\end{thm}
For the double affine Artin groups associated to untwisted affine root systems (that is when $r=1$) this action was constructed in \cite{ISTri}*{\S4.2}. For the \emph{extended} double affine Artin groups associated to untwisted affine root systems the action was constructed earlier by Cherednik \cite{CheMac}*{Theorem 4.3}. For  the double affine Artin groups associated to twisted affine root systems ($r=2,3$) the result is new. All these automorphisms descend to automorphisms of the Hecke algebras $\bH(\DDDot{X}_n)$, $\bH(\fH{X}_n)$, and $\bH(\DDot{C}_n)$ and we explain which subgroups descend to actions on the double affine Hecke algebras associated to the various affine root systems. The map $\Gamma\to \Out(\B)$ is precisely the outer action defined in \cite{ISEALA} in the context of the connection between double affine Artin groups and the Artin groups associated to $\eala{2}$-EALAs.

In the hierarchy of double affine Artin groups associated to (double) semisimple data for fixed root systems, the double affine Artin groups $\B=\B(\DDDot{X}_n)$ or $\B(\fH{X}_n)$ correspond to adjoint data (the smallest semisimple data).  The double affine Artin groups $\B^e=\B^e(\DDDot{X}_n)$ or $\B^e(\fH{X}_n)$ that correspond to simply connected semisimple data (the largest semisimple data) are called in the literature extended double affine Artin groups. In general, a double affine Artin group attached to arbitrary semisimple data is nothing else but an intermediate group
$$
\B\leq \B^s\leq \B^e.
$$  
A double affine Artin group associated to reductive data is isomorphic to the almost-direct product of the double affine Artin groups associated to the underlying semisimple data and central data. In this generality, we have the following version of Theorem \ref{state3}.
\begin{thm}\label{state6}
Let $\B^r$ be a double affine Artin group attached to reductive data and let $\B\leq \B^s\leq \B^e$ and $\B^c$ be the double affine Artin groups attached to the underlying semisimple data and central data, respectively. There exists a finite index subgroup $\widetilde{K}\leq \widetilde{\Gamma}$, a congruence subgroup $K\leq \Gamma$, and a faithful action  $\widetilde{K}\to \Aut(\B^r)$ that descends to a morphism $K\to \Out(\B^r)$. Both $\widetilde{K}$ and $K$ act trivially on $\B^c$.
\end{thm}
It is important to remark that both for $\B$ and $\B^e$ we have $\widetilde{K}=\widetilde{\Gamma}$.

Theorem \ref{state3} also allows us to give a description of the involutions of $\B$ that originate from the above action of $\Gamma$ and that essentially extend the Kazhdan-Lusztig involution on the finite Artin group. We call these involutions basic involutions.  

\begin{thm}\label{state4}
The elements of $\wGamma(r)$ that induce basic involutions are precisely those which, under the canonical projection $\wGamma(r)\to \Gamma_1(r)$, are mapped to elements of the form
$$
\begin{bmatrix}a & b\\ -rb & d\end{bmatrix}.
$$
\end{thm}
Each involution endows  the corresponding double affine Artin group and Hecke algebra with a $*$-structure. In turn, this leads to a $C^*$-algebra completion, a notion of unitary representation, and to a host of analytical and representation theoretical questions naturally associated to such a context.

It is important to note that if we remove the affine nodes in a double affine Coxeter diagram of type $\DDDot{X}_n$ or $\fH{X}_n$ then we obtain a finite Coxeter diagram of type $X_n$. Consequently, the groups $\B(\DDDot{X}_n)$ and $\B(\fH{X}_n)$ have $B(X_n)$, the Coxeter braid group associated to the diagram $X_n$, as a subgroup. The actions of the groups  $\widetilde{\Gamma}$  on $\B(\DDDot{X}_n)$ and $\B(\fH{X}_n)$ are always trivial on $B(X_n)$.

A (complex) representation $V$ of $\B(\DDDot{X}_n)$ or $\B(\fH{X}_n)$ is said to be $B(X_n)$-finite if for any vector $v\in V$, the vector space $B(X_n)\cdot v$ is finite dimensional. For $\hat{\pi}$ a simple representation of $B(X_n)$ we denote by $V(\hat{\pi})$ the isotypic component of $\hat{\pi}$ inside $V$, which is defined as the direct sum of all copies of $\hat{\pi}$ inside $V$ (seen as a $B(X_n)$-representation, by restriction). The representation $V$ is said to be  $B(X_n)$-admissible if it is  $B(X_n)$-finite and the isotypical component $V(\hat{\pi})$ is finite dimensional for any $\hat{\pi}$ a finite-dimensional simple representation of $B(X_n)$.

Let $V$ be a representation of $\B(\DDDot{X}_n)$ or $\B(\fH{X}_n)$ and let $\gamma\in \widetilde{\Gamma}$. The action on $V$ can be pre-composed with $\gamma$ to obtain a new representation that is denote by ${}^\gamma V$. If $\gamma$ is an inner automorphism then ${}^\gamma V$ and $V$ are isomorphic. Therefore, the action of  $\widetilde{\Gamma}$  on the set of isomorphism types of representations of $\B(\DDDot{X}_n)$ or $\B(\fH{X}_n)$ descends to an action of $\Gamma$. An orbit of $\Gamma$ acting on the set of isomorphism types of simple representations of $\B(\DDDot{X}_n)$ or $\B(\fH{X}_n)$ is called a $\Gamma$-packet. We say that a simple representation $V$ is a $\Gamma$-stable representation if  the ${}^\gamma V$ and $V$ are isomorphic for any  $\gamma\in\widetilde{\Gamma}$. In other words, $V$ is $\Gamma$-stable if and only if its $\Gamma$-packet consists of a single element.

\begin{thm}\label{state5}
For any $V$ simple $\Gamma$-stable representation of $\B(\DDDot{X}_n)$ or $\B(\fH{X}_n)$ and any $\hat{\pi}$ simple representation of $B(X_n)$, we have a projective representation
$$
{\Gamma} \to {\rm PGL}(V(\hat{\pi})).
$$
 If $V$ is $B(X_n)$-admissible then this projective representation is finite dimensional. Except for $\DDDot{X}_n$ of type $\DDDot{A}_n$, $n\geq 2$, $\DDDot{D}_{2n}$, $n\geq 2$, or $\DDDot{E}_6$, the projective representation descends to the quotient $\Gamma/\{\pm I_2\}$.
\end{thm}
\begin{proof}
For each $\gamma\in\widetilde{\Gamma}$ the representations ${}^\gamma V$ and $V$ are isomorphic. Because $V$ is simple, an isomorphism between ${}^\gamma V$ and $V$ is unique up to scaling. Therefore, we obtain a group morphism
$$
\widetilde{\Gamma} \to {\rm PGL}(V).
$$
We remark that, since the action of  $\widetilde{\Gamma}$ is trivial on $B(X_n)$, an isomorphism between  ${}^\gamma V$ and $V$ necessarily preserves the isotypic components of the two representations. Hence, for any  $\hat{\pi}$ simple representation of $B(X_n)$ we have a projective representation
$$
\widetilde{\Gamma} \to {\rm PGL}(V(\hat{\pi})).
$$
If $V$ is $B(X_n)$-admissible then, by definition, $V(\hat{\pi})$ is finite dimensional. In all situations, the generator  of the kernel of the canonical projection $\widetilde{\Gamma}\to \Gamma$ acts by an inner automorphism, more precisely by conjugation with a central element in $B(X_n)$ (see Theorem \ref{thm: auto1}, Theorem \ref{P1}, and Theorem \ref{P2}). Since a central element in $B(X_n)$ acts as a scaling on $V(\hat{\pi})$, and therefore trivially as an element of  ${\rm PGL}(V(\hat{\pi}))$ we obtain that the projective representation of $\widetilde{\Gamma}$ factors through to a projective representation of $\Gamma$. In fact, unless we are in the situations indicated in the statement, the generator  of the kernel of the canonical projection $\widetilde{\Gamma}\to \Gamma/\{\pm I_2\}$ acts by an inner automorphism given  by conjugation with a central element in $B(X_n)$. In this case, the projective representation of $\widetilde{\Gamma}$ factors through to a projective representation of $\Gamma/\{\pm I_2\}$.
\end{proof}

One source of examples of simple, finite-dimensional, $\Gamma$-stable representations are the representations of the double affine Artin groups that factor through representations of the associated Hecke algebras (with specialized parameters), more precisely Cherednik's perfect and quasi-perfect representations of double affine Hecke algebras \cite{CheDou-1}*{\S2.9.3, \S3.10.3} are of this type. Depending on the particular situation, we expect that the projective representations of $\Gamma$ described in Theorem \ref{state5} will lift to (usual) representations of a finite cover of $\Gamma$.

One particular class of examples of simple, finite-dimensional, $\Gamma$-stable representations are those that factor through a representation of the double affine Weyl groups $\Cox(\DDDot{X}_n)$ and $\Cox(\fH{X}_n)$. Such representations of $\Cox(\DDDot{X}_n)$ and $\Cox(\fH{X}_n)$ were constructed by Kac and Peterson  \cite{KacInf}*{Ch. 13} on spaces $Th_k$  of level $k$ theta functions. In this situation, $Th_k^{C(X_n)}$, which is the isotypic component of the trivial representation of the finite Weyl group $C(X_n)$ (the Coxeter group to the diagram $X_n$), can be seen as the representation ring of the level $k$ irreducible objects in the category $\O$ for the corresponding affine Lie algebra, or equivalently, as the linear span of the characters of its level $k$ highest-weight modules. Theorem \ref{state5} states that there is a projective representation of $\Gamma$ on  $Th_k^{C(X_n)}$. It can be directly verified that this projective representation lifts to a representation of $\Gamma$, recovering thus the classical result of Kac and Peterson \cite{KacInf}*{\S13.8, \S13.9}. Just as in the classical Kac-Peterson situation, it is reasonable to expect that, in the general context of Theorem \ref{state5}, functions arising from the structure of a simple $\Gamma$-stable representation (such as $B(X_n)$-type multiplicities) are modular forms.

Since, depending on one's point of view, the double affine Artin groups are labeled in various ways (double affine Coxeter diagrams, irreducible reduced affine root systems/affine Dynkin diagrams, adjoint double data) 
 we indicate in Table \ref{table: all} the precise correspondence between  labels. The groups labeled by symbols in the same row coincide. Whenever labels are missing from column (indicated by \tquest), the construction of the corresponding double affine Artin group from that point of view is missing from the literature. For completeness we have specified which are the double affine Artin groups corresponding to nonreduced irreducible affine root systems.

\begin{table}[ht]
\caption{Labels for double affine Artin groups}\label{table: all}
\begin{tabular}{ l | l l c l c}
$r$& Coxeter 						& Affine Dynkin			&  Adjoint  Data				 
&& Macdonald \\
\hline
1&$\DDDot{A}_1$ 					& $A_1^{(1)}, (BC_1, C_1),$	&$A_1^u$
&& $\left(\begin{smallmatrix} S(A_1)^\vee & S(A_1)^\vee \\ A_1 & A_1 \\ Q(A_1)^\vee & Q(A_1)^\vee \end{smallmatrix} \right) $\\
" &	"	 					& $(C_1^\vee, C_1)$						& " 
&& $ \left(\begin{smallmatrix} (C_1^\vee, C_1) & (C_1^\vee, C_1) \\ C_1 & C_1 \\ Q(C_1)^\vee & Q(C_1)^\vee \end{smallmatrix} \right) $\\ 
1&$\DDDot{A}_n, n\geq 2$ 			& $A_n^{(1)}$						&$A_n^u$		
&&  $\left(\begin{smallmatrix} S(A_n)^\vee & S(A_n)^\vee \\ A_n & A_n \\ Q(A_n)^\vee & Q(A_n)^\vee \end{smallmatrix} \right)$\\ 
1&$\DDDot{B}_n, n\geq 3$ 			& $B_n^{(1)}$						&$B_n^u$		
&& $\left(\begin{smallmatrix} S(B_n)^\vee & S(B_n)^\vee \\ B_n & B_n \\ Q(B_n)^\vee & Q(B_n)^\vee \end{smallmatrix} \right) $ \\ 
1&$\DDDot{C}_n, n\geq 2$ 			& $C_n^{(1)}, (BC_n, C_n), $			&	$C_n^u$	
&& $\left(\begin{smallmatrix} S(C_n)^\vee & S(C_n)^\vee \\ C_n & C_n \\ Q(C_n)^\vee & Q(C_n)^\vee \end{smallmatrix} \right) $ \\
" &	"						& $(C_n^\vee, C_n)$						&"	
&& $ \left(\begin{smallmatrix} (C_n^\vee, C_n) & (C_n^\vee, C_n) \\ C_n & C_n \\ Q(C_n)^\vee & Q(C_n)^\vee \end{smallmatrix} \right)  $ \\ 
1&$\DDDot{D}_n, n\geq 4$ 			& $D_n^{(1)}$						&$D_n^u$		
&& $\left(\begin{smallmatrix} S(D_n)^\vee & S(D_n)^\vee \\ D_n & D_n \\ Q(D_n)^\vee & Q(D_n)^\vee \end{smallmatrix} \right)$ \\ 
1&$\DDDot{E}_n, n=6,7,8$ 			& $E_n^{(1)}$						&$E_n^u$		
&& $\left(\begin{smallmatrix} S(E_n)^\vee & S(E_n)^\vee \\ E_n & E_n \\ Q(E_n)^\vee & Q(E_n)^\vee \end{smallmatrix} \right)$ \\ 
1&$\DDDot{F}_4$ 					& $F_4^{(1)}$						&$F_4^u$		
&& $\left(\begin{smallmatrix} S(F_4)^\vee & S(F_4)^\vee \\ F_4 & F_4 \\ Q(F_4)^\vee & Q(F_4)^\vee \end{smallmatrix} \right) $ \\ 
1&$\DDDot{G}_2$ 					& $G_2^{(1)}$						&$G_2^u$		
&&$\left(\begin{smallmatrix} S(G_2)^\vee & S(G_2)^\vee \\ G_2 & G_2 \\ Q(G_2)^\vee & Q(G_2)^\vee \end{smallmatrix} \right) $ \\ \hline
2&$\DDot{A}_1$ 				& $A_2^{(2)}, (C_1^\vee, BC_1)$			&\tquest	
&& \tquest \\ 
2&$\DDot{C}_n, n\geq 2$ 			& $A_{2n}^{(2)}, (C_n^\vee, BC_n)$		&\tquest		
&& \tquest \\ \hline
2&$\fH{B}_2/\fH{C}_2$ 					& $D_3^{(2)}, A_3^{(2)}, (C_2, C_2^\vee)$			&$B_2^t, C_2^t$		
&& $\left(\begin{smallmatrix} S(C_2) & S(C_2^\vee) \\ C_2 & C_2^\vee \\ Q(C_2) & Q(C_2)^\vee \end{smallmatrix} \right) $\\ 
2&$\fH{B}_n/\fH{C}_n, n\geq 3$ 	& $D_{n+1}^{(2)}, A_{2n-1}^{(2)}, (B_n, B_n^\vee)$	&$B_n^t, C_n^t$ 
&& $\left(\begin{smallmatrix} S(C_n) & S(C_n^\vee) \\ C_n & C_n^\vee \\ Q(C_n) & Q(C_n)^\vee \end{smallmatrix} \right) $\\  
2&$\fH{F}_4$ 					& $E_6^{(2)}$						&$F_4^t$		
&& $\left(\begin{smallmatrix} S(F_4) & S(F_4^\vee) \\ F_4 & F_4^\vee \\ Q(F_4) & Q(F_4)^\vee \end{smallmatrix} \right) $ \\ 
3&$\fH{G}_2$ 					& $D_4^{(3)}$						&$G_2^t$		
&& $\left(\begin{smallmatrix} S(G_2) & S(G_2^\vee) \\ G_2 & G_2^\vee \\ Q(G_2) & Q(G_2)^\vee \end{smallmatrix} \right) $ \\ 

\end{tabular}
\end{table}

We have also included Macdonald's parametrization in Table \ref{table: all}. Macdonald \cite{MacAff} labels the extended double affine Artin groups by data of the form 
$$\left(
\begin{matrix}
S & S^\prime \\
R & R^\prime \\
L & L^\prime
\end{matrix}
\right)
$$
where $S, S^\prime$ are irreducible affine root systems, $R, R^\prime$ are irreducible finite root systems, and $L, L^\prime$ are lattices. Only specific choices entries are allowed and they are specified in \cite{MacAff}*{(1.4.1), (1.4.2), (1.4.3)}. Macdonald constructs the extended double affine Artin groups. In Table \ref{table: all} we list the label that corresponds to the double affine Artin groups. This corresponds to letting $L, L^\prime$ be root or coroot lattices instead of weight or coweight lattices. The notation in Table \ref{table: all} for the Macdonald label uses $X_n$ and $X_n^\vee$ for a finite root system and its dual, $Q(X_n)$ and $Q(X_n)^\vee$ for the root and coroot lattice of $X_n$, and $S(X_n)$  and $S(X_n)^\vee$ for the untwisted affine root system corresponding to $X_n$  and the dual affine root system. As it can be seen from the table there are redundancies in the Macdonald's parametrization and also in the parametrization by affine root systems and these were not previously observed in the literature.

\subsection{Structure of the proofs} 
The double affine Artin groups appeared for the first time (under the name of extended Artin groups) in the work of van der Lek \cites{vdLExt, vdLHom} as the fundamental groups of the space of regular orbits of a natural action of the double affine Weyl groups. The regular points for the action consists of the complement of a certain complex hyperplane arrangement. Van der Lek described the fundamental group by generators and relations and his presentation matches precisely the subsequent definition by Cherednik \cite{CheDou-2} of the group that is ultimately fully responsible for the double affine Hecke algebra structure. The complex hyperplane arrangement itself originated as the  discriminant of a semi-universal deformation of simply-elliptic singularities \cite{LooInv} and thus the fundamental group becomes the monodromy group for this type of singularities.

From topological considerations, van der Lek identifies a finite set of relations that are sufficient to present a double affine Artin group. We refine his presentation to further reduce the number of relations in the presentation that do not typically appear in the context of Coxeter braid groups to one class of relations (there is an exception, which we identify, where an additional relation is required). Using this refined presentation we can easily establish Theorem \ref{state2}. The extra relations from the refined van der Lek presentation match the relations used in defining the quotients $\B(\DDDot{X}_n)$, $\B(\DDot{C}_n)$, and $\B(\fH{X}_n)$. The three affine generators in the definition of $\B(\DDDot{X}_n)$ explicitly appear for the first time in \cite{SahNon}  where the double affine Hecke algebra of type $(C_n^\vee, C_n)$ was considered. 

Theorem \ref{state3} is easily verified from the vantage point  of the Coxeter-type presentation. The action of $\widetilde{\Gamma}$ is always trivial on the generators associated to finite nodes and it is therefore determined by the action on the affine generators. For $\wGamma(1)$, which is the Coxeter braid group of type $A_2$, this action is precisely the canonical action of the braid group of type $A_2$ on the free group on three letters. The action of the  Coxeter braid groups of type $B_2$ and $G_2$  is quite subtle and does not seem to have been considered classically. An explicit calculation of the action of the center of $\widetilde{\Gamma}$ allows us to deduce that  the action of $\widetilde{\Gamma}$ descents to the canonical topological action of $\Gamma$ as outer automorphisms. As it is clear from the statement, it is enough to prove Theorem \ref{state6} for the double affine Artin groups attached to semisimple data and this is the subject of Theorem \ref{thm: autoall-ss}. The proof of Theorem \ref{thm: autoall-ss} proceeds by first extending the action of $\widetilde{\Gamma}$ to $\B^e$ and then deducing that there exists a finite index subgroup of $\widetilde{\Gamma}$ that stabilizes all the intermediate subgroups between $\B$ and $\B^e$. Theorem \ref{state4} is a consequence of the explicit knowledge of the action of $\widetilde{\Gamma}$.

\subsection{Structure of the paper}\label{sec: structure-paper} The expository material pertaining to double affine Artin groups is contained in \S\ref{sec: DAAGs}. The Coxeter-type presentation for double affine Artin groups (Theorem \ref{state2}) is the subject of \S\ref{sec: Coxeter}, but some preliminary technical facts that refine van der Lek's presentation are obtained in \S\ref{refinement}. The construction of the $\widetilde{\Gamma}$ action and the proof of Theorem \ref{state3} are contained in \S\ref{sec: auto}. The basic involutions that originate from the action of $\widetilde{\Gamma}$ are described in \S\ref{sec: involutions}. In \S\ref{sec: other} we discuss other automorphisms of the double affine Artin groups that are revealed by the Coxeter-type presentation. In \S\ref{sec: reductive-DABG} we study the double affine Artin groups attached to reductive data and obtain their structure as an almost-direct product of the double affine Artin groups attached to the underlying semisimple data and central data. The groups attached to simply connected semisimple data (extended double affine Artin groups) are investigated in more detail in \S\ref{sec: ext-DABG}, which also contains the proof of Theorem \ref{state6}. The results concerning double affine Hecke algebras are contained in \S\ref{sec: hecke}.

\begin{ack} Part of this work was carried during the Workshop on Hecke Algebras and Lie Theory, which was held at the University of Ottawa during May 12-15, 2016. The authors thank the National Science Foundation, the Fields Institute, and the University of Ottawa for funding this workshop. We thank  Ivan Cherednik and Jasper Stokman for their interest and their comments on the manuscript.
\end{ack}

\section{Notation and conventions}

\subsection{}
Hereafter, unless otherwise specified all vector spaces are complex vector spaces. When a complex structure on a complex vector space is specified, $\Real$ and $\Imag$ refer to the maps identifying the real and imaginary components of an element. The symbols  $\Z$, $\Re$, $\Co$ refer, as usual, to the set of integers, real numbers, and complex numbers, respectively. If $G$ is a group, $\Aut(G)$ and $\Out(G)$  refer to the groups of automorphisms and, respectively, outer automorphisms of $G$. Furthermore, if $\mathcal{S}$ is a subset of $G$ we denote by $\<\mathcal{S}\>$ the subgroup of $G$ generated by $\mathcal{S}$.

\subsection{} Let $0\leq p\leq 4$ and let $a,b$ two elements of a fixed group. We say that $a$ and $b$ satisfy the $p$-braid relation if
\begin{equation}
aba\cdots=bab\cdots,
\end{equation}
where there are $2,3,4,6$ or $\infty$ factors on each side if $p=0,1,2,3$, or greater than $4$, respectively. As usual, for $p\geq 4$, the $p$-braid relation is interpreted as the empty relation between $a$ and $b$. A concise way to specify pairwise braid relations for a set of elements is by encoding the information in a graph, called Coxeter graph or Coxeter diagram, as follows: the nodes correspond to elements and if two nodes are connected by $p$ edges ($p\leq 4$) then the corresponding two elements satisfy the $p$-braid relation. 

\subsection{}  
To a Coxeter diagram $D$ we can associate the corresponding Coxeter braid group $B(D)$ defined as the group generated by a set of elements $T_i$ , one for for each node in $D$, that satisfy the braid relations specified by the number of edges between nodes. Remark that, by the nature of the braid relations, $B(D)^{\rm op}$ (defined as $B(D)$ with the opposed multiplication) is always isomorphic to $B(D)$, a canonical isomorphism being the inverse map; a second canonical isomorphism between $B(D)$ and $B(D)^{\rm op}$ is the group morphism that acts as identity on generators. Another important group associated  to $D$ is the Coxeter group $C(D)$ defined as the quotient of $B(D)$  by the normal subgroup generated by the squares of the generators. The canonical projection
\begin{equation}
\pi_D: B(D)\to C(D)
\end{equation}
has a canonical (set-theoretic) section
\begin{equation}
t_D: C(D)\to B(D)
\end{equation}
which is defined as follows: for $w\in C(D)$, let $t_D(w)$ be the unique element of $\pi_D^{-1}(w)$ that minimizes the length of its expression as a product of the generators of $B(D)$.

\subsection{}\label{sec: cl} 
The diagrams that we will ultimately consider in this paper are extensions of the finite crystallographic Coxeter diagrams. The additional nodes, which will be clearly identified, will be called affine nodes, and we will refer to the nodes of the finite crystallographic Coxeter diagram as finite nodes. Whenever it is necessary to label the finite nodes (by positive integers)  we will label then in the conventional fashion (see e.g. \cite{KacInf}*{Table Fin}). If there is only one affine node, it will be labeled by $0$. If there are two or more affine nodes their labelling will be carefully specified.

\subsection{}  To the Coxeter diagram $D$ we can also associate a Hecke algebra. For this one needs a field $\F=\Rat(t_1^{\frac{1}{2}},\dots,t_n^{\frac{1}{2}})$ with the formal parameters $t_i$ indexed by the nodes of $D$. The Coxeter Hecke algebra $H(D)$ is  defined as the quotient of the $\F$-group algebra of $B(D)$ by the ideal generated by the set quadratic relations 
$$
T_i-T_i^{-1}=t_i^{\frac{1}{2}}-t_i^{-\frac{1}{2}}.
$$
The Hecke parameters have to coincide if the corresponding generators are conjugate inside $B(D)$, or equivalently, if the corresponding nodes can be connected in $D$ by a path consisting of simple edges. We call $1$-connected component of $D$ a path-connected component of the graph obtained from  $D$ by erasing all the multiple edges. Therefore, the set of independent Hecke parameters can be  labelled by the $1$-connected components of $D$ or by any choice of representatives  for the $1$-connected components.


\section{Affine root systems}\label{sec: DAAGs}


\subsection{}  For the most part we follow the notation in \cite{KacInf}. Let $A=(a_{ij})_{0\leq i,j\leq n}$ be an indecomposable affine Cartan matrix of rank $n$, $D(A)$ the Dynkin diagram, and  $(a_0,\dots, a_n)$  the numerical labels of $D(A)$ in Table Aff from \cite{KacInf}*{pg.54--55}. Note that we consider that the nodes $i$ and $j$ in $D(A)$ are connected by $a_{ij}a_{ji}$ laces. Unless $A=A_1^{(1)}$ this produces the same diagrams as in \cite{KacInf}.

We denote by $(a_0^\vee,\dots, a_n^\vee)$ the labels of the dual Dynkin diagram $D({}^t\! A)$ which is obtained from $D(A)$ by reversing the direction of all arrows and keeping the same enumeration of the vertices. The associated finite Cartan matrix is $\Aring = (a_{ij})_{1\leq i,j\leq n}$.   Note that  $a_0^\vee=1$ for all indecomposable affine Cartan matrices while $a_0=1$ in all cases except for $A=A_{2n}^{(2)}$ for which $a_0=2$.

\subsection{} 
Let $(\Haff, R, R^{\vee})$ and $(\Hring, \Rring, \Rring^{\vee})$ be realizations of $A$ and  $\Aring$, respectively.  The root system $R$ is said to be untwisted if  the corresponding Dynkin diagram appears in Table Aff 1 from \cite{KacInf}*{pg.54} and it is said to be twisted (of twist number $2$ or $3$) if  the corresponding Dynkin diagram appears in Table Aff 2 or Aff 3 in \cite{KacInf}*{pg.55}.

We can arrange that $\Hring\subset \Haff$ and  $\Rring\subset R$.  Let  $\{\a_i\}_{0\leq i\leq n}\subset \Haff^*$ be a basis of $R$ such that $\{\a_i\}_{1\leq i\leq n}$ is  a basis of $\Rring$ and let $\{\a_i^\vee\}_{0\leq i\leq n}$ the corresponding set of coroots. The choice of basis determines subsets of positive roots $R^+\subset R$ and $\Rring^+=R^+\cap \Rring \subset \Rring$;  with the notation $R^-:=-R^+$, $\Rring^-:=\Rring^+$ we have $R=R^+\cup R^-$ and $\Rring=\Rring^+\cup \Rring^-$. We denote by $R^{re}$ and $R^{im}$ the set of real roots and, respectively, null-roots of $R$.

\subsection{}
The positive non-divisible null-root in $R$ is
\begin{equation}
\d=a_0\a_0+\cdots+a_n\a_n.
\end{equation}
  Fix $\L_0\in\Haff^*$ such that $\L_0(\a_i^\vee)=\delta_{i,0}a_0$; $\L_0$ is unique modulo the subspace spanned by $\delta$. The vector space $\Haff^*$ can be written as
\begin{equation}
\Haff^*=\Hring^*\oplus\Co \d\oplus\Co \L_0.
\end{equation}

Let us denote by $\ph$ the highest root in $\Rring^+$.  An important role is played by the root
\begin{equation}
\th=a_1\a_1+\cdots+a_n\a_n.
\end{equation}
If $R$ is  untwisted or of type $A_{2n}^{(2)}$ we have $\th=\ph$; otherwise, $\Rring$ is not simply-laced and $\th$ is the short dominant root in $\Rring^+$.

The weight, coweight, root, and coroot lattices of $\G$  are denoted by $\Pring$, $\Pring^\vee$, $\Qring$ and $\Qring^\vee$, respectively. The root and coroot lattices of $\Gaff$ are denoted by $Q$ and $Q^\vee$, respectively. We note that we always have $Q^\vee=\Qring^\vee\oplus\Z\d$.

\subsection{} 

We will add $\Re$ as a subscript whenever we refer to the real form   of  $\Hring^*$ spanned by  the simple roots, or the real form of $\Haff^*$ spanned by  the simple roots and $\L_0$, respectively. 

The following defines the non-degenerate normalized standard bilinear  form $(\ , \ )$ on $\Haff_\Re^*$:
\begin{equation}
(\a_i,\a_j):=d_i^{-1}a_{ij}, \ 0\leq i,j\leq n\ ,\quad
(\L_0,\a_i):=\d_{i,0}a_0^{-1},\quad \text{and}\quad(\L_0,\L_0):=0,
\end{equation}
with $d_i:= a_ia_i^{{\vee}-1}$. In particular, we have
\begin{equation}
(\d,\Hring_\Re^*)=0, \quad (\d,\d)=0, \quad\text{and}\quad (\d,\L_0)=1.
\end{equation}
The corresponding isomorphism $\nu:\Haff_\Re\to\Haff_\Re^*$ sends $\a_i^\vee$ to $d_i\a_i$. This isomorphism allows us to routinely identify elements via $\nu$ and regard, for example, coroots as elements of $\Haff_\Re^*$, which is something we will do without further warning.  The form $(\ , \ )$ is extended to a symmetric form on $\Haff^*$ by $\Co$-bilinearity.

\subsection{} 

With respect to $(\ ,\ )$, the elements of $R^{re}$ have three possible lengths if $R$ is of type $A_{2n}^{(2)}$, $n\geq 2$, the one possible length if the affine Dynkin diagram is simply laced, and two possible lengths otherwise. We denote the set of short roots by $R_s$, the set of long roots by $R_\ell$, and, for $\Gaff$ of type $A_{2n}^{(2)}$, we denote the set of medium length roots by $R_m$. To avoid making the distinction later on, if there is only one root length we consider all real roots to be long. Similar notation and conventions apply to $\Rring$.

\subsection{}

For $0\leq i\leq n$, let $e_i=\max\{a_0^{-1}, d_i\}$.  Let us remark that the numbers $e_i$ depend only on the length of the corresponding simple root. The integer 
$$r:={\max_{0\leq i\leq n}\{d_i^{-1}\}}$$ 
is an important invariant. By definition, 
$$
\max_{\a\in R}(\a,\a)=2r.
$$
It is easy to see that $r\in\{1,2,3\}$ and that $A$ is listed in Table Aff $r$ in \cite{KacInf}*{pg. 54--55}. In other words, $r$ is the twist number of the root system $R$.
\subsection{}

Given $\a\in R^{re}$  and $x\in \Haff^*$ let
\begin{equation}
s_\a(x):=x-(x,\a^\vee)\a\ .
\end{equation}
The affine Weyl group $W$ is the subgroup of ${\GL}(\Haff^*)$ generated by all $s_\a$ (the simple reflections $s_i=s_{\a_i}$, $0\le i\leq n$ are enough). The finite Weyl group $\W$ is the subgroup generated by $s_1,\dots,s_n$. The bilinear form on $\Haff_\Re^*$ is equivariant with respect to the affine Weyl group action. Both the finite and the affine Weyl group are Coxeter groups and they can be abstractly defined the Coxeter groups associated $D(\Aring)$ and $D(A)$, respectively. For $x\in \Haff^*$ we denote by $W(x)$ and $\W(x)$ the orbit of $x$ under the action of $W$ and $\W$, respectively. 

\subsection{}
For each $w$ in $W$ let $\ell(w)$ be the length of a reduced decomposition of $w$ in terms of the simple reflections $s_i$, $0\leq i\leq n$. We call $\ell: W\to \Z_{\geq 0}$ the length function of $W$ and, for any $w\in W$, we refer to $\ell(w)$ as the length of $w$. The longest element of $\W$ is denoted by $w_\circ$.

\subsection{}
Let $M\subset\Hring_\Re^*$ be the lattice generated by $\W(a_0^{-1}\th)$; note that $a_0^{-1}\th=\nu(\th^\vee)$.  The lattice $M$ can be described more explicitly as  the lattice generated by $\{A_i=e_i\a_i\}_{1\leq i\leq n}$. It turns  out that each $A_i$ is either $\a_i$ or $\a_i^\vee$; we denote by $A_i^\vee$ the element $2A_i/(A_i,A_i)$, which equals $\a_i^\vee$ or $\a_i$, respectively. The lattice $M$ is equal to $\nu(\Qring^\vee)$ if $r=1$ and equal to $\Qring$ if $r=2,3$.

The affine Weyl group contains the finite Weyl group and a normal abelian subgroup isomorphic to $M$. We will denote the latter by $\l(M)$ and its elements by $\l_\mu$, $\mu\in M$. The action by conjugation of $\W$ on $\l(M)$ and the usual action of $\W$ on $M\subset \Hring_\Re^*$ are related by
\begin{equation}
\w\l_\mu \w^{-1}=\l_{\w(\mu)}.
\end{equation}
This allows $W$ to be presented as the semidirect product $\W\ltimes M$. Thus, $W$ is isomorphic to $\W\ltimes \Qring^\vee$ if $R$ is untwisted, and isomorphic to $\W\ltimes \Qring$ if $R$ is twisted.

\subsection{}

The double affine Weyl group $\Wdaff$ is defined to be the semidirect product $W\ltimes Q^\vee$ of the affine Weyl group and the coroot lattice $Q^\vee$. 
We use $\tau(Q^\vee)$ to refer to $Q^\vee$ as a subgroup of $\Wdaff$ and we denote its elements by  $\tau_\b$, $\b\in Q^\vee$. The action by conjugation of $W$ on $\tau(Q^\vee)$ and the usual action of $W$ on $Q^\vee\subset \Haff_\Re^*$ are related by
\begin{equation}
w\tau_\b w^{-1}=\tau_{w(\b)}.
\end{equation}

\begin{Prop}\label{wdef}
The double affine Weyl group $\Wdaff$ is the group generated by the finite Weyl group $\W$, two lattices $\{\l_\mu\}_{\mu\in M}$, $\{\tau_\b\}_{\b\in \Qring^\vee}$ and an element $\tau_{\d}$ satisfying the relations:
\begin{subequations} 
 \begin{alignat}{2}
& \w\l_\mu \w^{-1}=\l_{\w(\mu)} \text{ and } \w\tau_\b \w^{-1}=\tau_{\w(\b)} ,\\ 
&\l_\mu\tau_\b=\tau_\b\l_\mu\tau^{-(\b,\mu)}_{\d},\\ 
&\tau_{\d} \text{ is central}.
 \end{alignat}
 \end{subequations}
for any $\w\in\W$, $\mu\in M$, and $\b\in \Qring^\vee$.
\end{Prop}
The double affine Weyl group has a natural affine action on $\Haff^*$ that extends the linear action of $W$ on $\Haff^*$ by  
\begin{equation}
\tau_\b(x)=x+\b, \quad \text{for all } \b\in Q^\vee, ~x\in \Haff^*.
\end{equation}
We refer to this action as the  defining action of $\Wdaff$.
It is important to remark that $\W\ltimes \tau(\Qring^\vee)$, which is a subgroup of $\Wdaff$, is also a Coxeter group isomorphic to the affine  Weyl group corresponding to $A$ if $r=1$ or to ${}^t\! A$ if $r=2,3$.
\subsection{} Following van der Lek \cite{vdLHom} we define the affine Artin group as the fundamental group of a certain topological space. We briefly recall this construction. Let
\begin{equation}
\Omega=\{ x\in \Haff^*~|~ \Imag (x , \d)> 0 \}.
\end{equation}
For $\a\in {R}^{re}$ denote by ${H}_{\a}$ the complex hyperplane $\{x\in \Haff^*~|~(x,\a)=0\}\subset \Haff^*$ and consider
\begin{equation}
Y=\Omega\setminus \bigcup_{\a\in {R}^{re}} {H}_{\a}.
\end{equation}
The affine Weyl group $W$ acts freely and properly discontinuously on $Y$. Denote by  $X$  the space of $W$-orbits on $Y$. The affine Artin group  is defined as the fundamental group of $X$ and is denoted by $\A(R)$.

The Artin group associated to the finite root system $\Rring$, called the finite Artin group  and denoted by  $\A(\Rring)$, can be defined in a similar manner, but  it can also be realized as subgroup of $\A(R)$ \cite{vdLHom}*{Ch. III, Lemma 4.1}.

\subsection{}  The finite and affine Weyl groups are Coxeter groups; in these cases, the associated Coxeter braid groups turn out to be isomorphic to the corresponding Artin groups \cite{DelImm}.

\begin{Prop} With the notation above we have
\begin{enumerate}[label={\roman*)}]
\item the finite Artin group $\A(\Rring)$ is the group generated
by elements
$$T_1,\dots,T_n$$
satisfying the same braid relations as the reflections
$s_1,\dots,s_n$; 
\item the affine Artin group $\A(R)$ is the
group generated by the elements
$$T_0,\dots,T_n$$
satisfying the same braid relations as the reflections
$s_0,\dots,s_n$.
\end{enumerate}
We refer to the above presentations as the Coxeter presentations of $\A(\Rring)$ and $\A(R)$.
\end{Prop}

\begin{Rem}\label{rem: affine=untwisted}
The Coxeter presentation makes clear that the finite and affine Artin groups depend only on the Coxeter diagram underlying the relevant Dynkin diagram. In particular, any affine Artin group is isomorphic to the affine Artin group associated to an untwisted affine root system. The Coxeter presentation makes also clear how the finite Artin group can be realized as a subgroup inside the affine Artin group. 
\end{Rem}

\subsection{} 
For $w\in W$ we denote by $T_w$ the element of $\A(R)$ defined as $T_{i_p}\cdots T_{i_1}$ if $w=s_{i_p}\cdots s_{i_1}$ is a reduced expression. Since the elements $T_i$ satisfy the same braid relations as the elements $s_i$, the element $T_w$ does not depend on the choice of reduced expression for $w$. We have $T_{s_i}=T_i$ for any  $0\leq i\leq n$.  If $u,v\in W$ such that $\ell(uv)=\ell(u)+\ell(v)$ then we have $T_{uv}=T_u T_v$.

We will use the notation $\Phi=T_{s_\ph}$ and $\Theta=T_{s_\th}$.
\subsection{}

For further use, we introduce the following lattices
\begin{equation}
\Q_Y:=\{Y_\mu;\mu\in M\}\quad \text{and} \quad\Q_X:=\{X_\b;\b\in
\Qring^\vee\}.
\end{equation}
Recall that the affine Weyl group has a second presentation, as a semidirect product. There is a corresponding description of the affine Artin group  due to van der Lek \cite{vdLHom}*{Ch. III, Theorem 5.5} and a closely related presentation independently obtained by Bernstein (unpublished) and Lusztig \cite{LusAff}. To be more precise, Bernstein  and Lusztig  give  the corresponding description of what is called in the literature the \emph{extended Hecke algebra} (the proof also works for the extended Coxeter braid group). Van der Lek's result is more subtle and the proof relies on the topological description of the affine Artin group.

\begin{Prop}\label{bernstein-presentation}
The affine Artin group $\A(R)$ is generated by the finite Artin group $\A(\Rring)$ and the lattice $\Q_Y$ such that the following relations are satisfied for all $1\leq i\leq n$ and $\mu\in M$
\begin{subequations} 
 \begin{alignat}{2}
& T_iY_\mu=Y_\mu T_i   \text{ if } (\mu, A_i^\vee)=0, \\
& T_iY_\mu T_i=Y_{s_i(\mu)} \text{ if } (\mu,A_i^\vee)=1.
\end{alignat}
\end{subequations}
\end{Prop}
\begin{Rem}
The presentation of $\A(R)$ described in Proposition \ref{bernstein-presentation} is referred to in the literature as the Bernstein presentation.
\end{Rem}
\begin{Rem}\label{braid01}
In this description $Y_\mu=T_{\l_\mu}$ for $\mu$ any anti-dominant element of $M$. For example, $$Y_{-a_0^{-1}\th}=\Theta T_0.$$ Proposition \ref{bernstein-presentation} implies that the element $\Theta^{-1}Y_{-a_0^{-1}\th}$ satisfies the $a_{0i}a_{i0}$-braid relations with the generators $T_i$, $1\leq i\leq n$. 
\end{Rem}
\begin{Rem}\label{rem: double braid}
As we already pointed out in Remark \ref{rem: affine=untwisted} we may assume that any affine Artin group is the affine Artin group associated to an untwisted root system. In this case, if $(\a_i, \th^\vee)\neq 0$  then necessarily $(\a_{i},\th^\vee)=1$ and Proposition \ref{bernstein-presentation} implies that $Y_{\th^\vee}$ and $T_{i}Y_{\th^\vee}T_{i}$ commute. In other words, in any affine Artin group, $T_{i}$ and $T_0^{-1}\Theta^{-1}$ satisfy the $2$-braid relation. Keeping in mind that any Coxeter braid group is self-anti-isomorphic through the anti-morphism that acts as identity of the generators, we infer that the same true for $T_{i}$ and $\Theta^{-1}T_0^{-1}$. On the other hand, if $(\a_i, \th^\vee)= 0$ then Proposition \ref{bernstein-presentation} implies that $T_{i}$  and $Y_{\th^\vee}$ commute or, in other words, $T_{i}$  and  $T_0^{-1}\Theta^{-1}$ commute. As before, we infer that $T_{i}$ and $\Theta^{-1}T_0^{-1}$ also commute.
\end{Rem}

\subsection{} 
In fact, van der Lek's description is even more precise; he identifies a finite set of relations which should be imposed.
\begin{Prop}
The affine Artin group $\A(R)$ is generated by the finite Artin group $\A(\Rring)$ and the lattice $\Q_Y$ such that the following
relations are satisfied for $1\leq i,j\leq n$
\begin{enumerate}[label={\alph*)}]
\item For any pair of indices $(i,j)$ such that
  $2r_{ji}=-(A_j,A_i^\vee)$,
with $r_{ji}$ a non-negative integer  we have
\begin{equation}
T_iY_{\mu_j}=Y_{\mu_j} T_i
\end{equation}
where $\mu_j=A_j+r_{ji}A_i$; note that $(\mu_j,A_i^\vee)=0$;
\item For any pair of indices $(i,j)$ such that \
$2r_{ji}-1=-(A_j,A_i^\vee)$, with $r_{ji}$ a non-negative integer we
have
\begin{equation}
T_iY_{\mu_j} T_i=Y_{s_i(\mu_j)}
\end{equation}
where $\mu_j=A_j+r_{ji}A_i$; note that $(\mu_j,A_i^\vee)=1$.
\end{enumerate}
\end{Prop}
\subsection{} To the double affine Weyl group $\Wdaff$ van der Lek associated another topological space that we will now describe. For $\a\in {R}^{re}$ and $k\in\Z$ denote by ${H}_{\a,k}$ the complex hyperplane $\{x\in \Haff^*~|~(x,\a)=k\}\subset \Haff^*$ and consider
\begin{equation}
\widetilde{Y}=\Omega\setminus \bigcup_{\a\in {R}^{re},~k\in\Z} {H}_{\a,k}.
\end{equation}
The double affine Weyl group $\Wdaff$ acts freely and properly discontinuously on $\widetilde{Y}\subset \Haff^*$. Denote by $\widetilde{X}$ the space  of $\Wdaff$-orbits on $\widetilde{Y}$. The double affine Artin group, defined as the fundamental group of $\widetilde{X}$, is denoted by  $\Atilde(R)$. The affine braid group $\A(R)$ can be realized as as subgroup of $\Atilde(R)$ \cite{vdLHom}*{Ch. III, Lemma 4.1}.

\subsection{} 

Van der Lek's results  \cite{vdLHom}*{Ch. III, Theorem 2.5} provide a Bernstein-type presentation for $\Atilde(R)$. 
\begin{Prop}\label{defcherednik}
The double affine Artin group $\Atilde(R)$ is generated by the
affine Artin group $\A(R)$, the lattice $\Q_X$, and the element
$X_{\d}$  such that the following relations are satisfied for
all $0\leq i\leq n$ and $\b\in \Qring^\vee$
\begin{subequations} 
 \begin{alignat}{2} \label{dc1}
& T_iX_\b=X_\b T_i \text{ if }  (\b,\a_i)=0,\\ \label{dc2}
& T_iX_\b T_i=X_{s_i(\b)}   \text{ if }  (\b,\a_i)=-1,\\ \label{dc3}
& X_{\d} \text{ is central}.
\end{alignat}
\end{subequations}
\end{Prop}
\begin{Rem}\label{rem: artin-to-weyl}
There is canonical group morphism $\Atilde(R)\to \Wdaff$ that sends $T_i$ to $s_i$, $0\leq i\leq n$, and $X_\b$ to $\tau_\b$, $\b\in \Qring^\vee$. The kernel of morphism is the normal subgroup of $\Atilde(R)$ generated by $T_i^2$, $0\leq i\leq n$, except for $R$ of type $A_1^{(1)}$, $C_n^{(1)}$, $n\geq 2$, and $A_{2n}^{(2)}$, $n\geq 1$. More precisely, for $R$ of type  $A_{2n}^{(2)}$, $n\geq 1$, the kernel is generated by $T_i^2$, $0\leq i\leq n$, and $(X_{\th^\vee}\Theta^{-1})^2$, and for $R$ of type $A_1^{(1)}$, $C_n^{(1)}$, $n\geq 2$, the kernel is generated by $T_i^2$, $0\leq i\leq n$, $(X_{\th^\vee}\Theta^{-1})^2$, and $(T_0^{-1}X_{\a^\vee_0})^2$. 
\end{Rem}

In \cite{vdLHom}*{Corollary 2.11} a smaller set of relations that need to be imposed in the algebraic description of $\Atilde(R)$ is identified. 
\begin{Prop}\label{first-presentation}
The double affine Artin group $\Atilde(R)$ is generated by the affine Artin group $\A(R)$, the lattice $\Q_X$, and the element $X_\d$ such that the following relations are satisfied for $0\leq i \leq n$ and $1\leq j\leq n$
\begin{enumerate}[label={\alph*)}]
\item For any pair of indices $(i,j)$ such that $2r_{ji}=-(\a_j^\vee,\a_i)$, with $r_{ji}$ a non-negative integer  we have
\begin{equation}\label{eq1}
T_iX_{\mu_j}=X_{\mu_j} T_i
\end{equation}
where $\mu_j=\a_j^\vee+r_{ji}\a_i^\vee$; note that $(\mu_j,\a_i)=0$;
\item For any pair of indices $(i,j)$ such that 
$2r_{ji}+1=-(\a_j^\vee,\a_i)$, with $r_{ji}$ a non-negative integer we
have
\begin{equation}\label{eq2}
T_iX_{\mu_j} T_i=X_{s_i(\mu_j)}
\end{equation}
where $\mu_j=\a_j^\vee+r_{ji}\a_i^\vee$; note that $(\mu_j,\a_i)=-1$;
\item $X_{\d}$ is central.
\end{enumerate}
\end{Prop}

\begin{Rem}\label{rem: lattices}
For each $0\leq i\leq n$ denote by $\eala{M}_i\subseteq Q^\vee$ the lattice generated by $\d$ and  \[\mu_j=\a_j^\vee+r_{ji}\a_i^\vee\quad \text{and}\quad s_i(\mu_j), \quad\text{for } 1\leq j\leq n,\] with $r_{ji}$ as in Proposition \ref{first-presentation}. As it can be easily seen, 
\begin{align*}
&\mu_j=s_i(\mu_j)=\a_j^\vee, &&\text{ if } \;\;\;\;(\a_j^\vee,\a_i)=0, \\
&\mu_j=\a_j^\vee \quad \text{and}\quad s_i(\mu_j)=\a_j^\vee+\a_i^\vee, &&\text{ if } -(\a_j^\vee,\a_i)=1,\\
&\mu_j=s_i(\mu_j)=\a_j^\vee+\a_i^\vee, &&\text{ if } -(\a_j^\vee,\a_i)=2,  \\
&\mu_j=\a_j^\vee+\a_i^\vee \quad \text{and}\quad  s_i(\mu_j)=\a_j^\vee+2\a_i^\vee, &&\text{ if } -(\a_j^\vee,\a_i)=3.
\end{align*}
Note that in all situations we have $\eala{M}_i+\Z\a_i^\vee=Q^\vee$. In fact, $\eala{M}_i=Q^\vee$, unless $A=A_1^{(1)}$ or $\a_i$ is a long root whose neighbors in $D(A)$ are all short.  Therefore,  the nodes for which $\eala{M}_i\subset Q^\vee$ are as follows
\begin{equation*}
\begin{aligned}
&A_1^{(1)} && i=0,1,\\
&C_n^{(1)}, ~n\geq 2 && i=0,n,\\
&A_{2n}^{(2)}, ~n\geq 1 && i=n,\\
&A_{2n-1}^{(2)}, ~n\geq 2 && i=n,\\
&D_{3}^{(2)}  && i=1.
\end{aligned}
\end{equation*}
Recall our convention on labeling in \S\ref{sec: cl}. It is important to remark that if $\eala{M}_i\subset Q^\vee$ then $\a_i$ is orthogonal on $\eala{M}_i$.
\end{Rem}

\subsection{} 

For our purposes, the following hybrid presentation will also be useful.
\begin{Prop}\label{second-presentation}
The double affine Artin group $\Atilde(R)$ is generated by the affine Artin group $\A(R)$, the lattice $\Q_X$, and the element $X_\d$ such that the following relations are satisfied.

\begin{enumerate}[label={\alph*)}]
\item For all $1\leq i\leq n$ and $\b\in \Qring^\vee$
\begin{subequations} 
 \begin{alignat}{2} \label{dc1v2}
& T_iX_\b=X_\b T_i \text{ if }  (\b,\a_i)=0,\\ \label{dc2v2}
& T_iX_\b T_i=X_{s_i(\b)}   \text{ if }  (\b,\a_i)=-1;
\end{alignat}
\end{subequations}
\item For any pair of indices of the form $(0,j)$, $1\leq j\leq n$,  such that $2r_{j0}=-(\a_j^\vee,\a_0)$, with $r_{j0}$ a non-negative integer  we have
\begin{equation}\label{eq1v2}
T_0X_{\mu_j}=X_{\mu_j} T_0
\end{equation}
where $\mu_j=\a_j^\vee+r_{j0}\a_0^\vee$; note that $(\mu_j,\a_0)=0$;
\item For any pair of indices of the form $(0,j)$, $1\leq j\leq n$, such that 
$2r_{j0}+1=-(\a_j^\vee,\a_0)$, with $r_{j0}$ a non-negative integer we
have
\begin{equation}\label{eq2v2}
T_0X_{\mu_j} T_0=X_{s_0(\mu_j)}
\end{equation}
where $\mu_j=\a_j^\vee+r_{j0}\a_0^\vee$; note that $(\mu_j,\a_0)=-1$;
\item 
\begin{equation}\label{dc3v2}
X_{\d} \text{ is central}.
\end{equation}
\end{enumerate}
\end{Prop}
\begin{Rem}\label{X-artin}
By comparing the relations \eqref{dc1v2}, \eqref{dc1v2} with Proposition \ref{bernstein-presentation} we obtain that the subgroup of $\Atilde(R)$ generated by $\A(\Rring)$ and $\Q_X$ is isomorphic to the affine Artin group corresponding to $A$ if $a_0^{-1} r=1$ or to ${}^t\! A$ if $a_0^{-1} r=2,3$. As in Remark \ref{braid01}, the affine generator of the group generated by $\A(\Rring)$ and $\Q_X$ is $X_{\ph^\vee}\Phi^{-1}$, which satisfies the $a_{0i}a_{i0}$-braid relations (if $a_0^{-1}r=1$), or the $a^\vee_{0i}a^\vee_{i0}$-braid relations (if $a_0^{-1}r=2,3$), with the generators $T_i$, $1\leq i\leq n$.
\end{Rem}

\subsection{}

One application of Proposition \ref{second-presentation} that will be useful for us is a comparison of $\Atilde(A_{2n}^{(2)}$) with $\Atilde(A_1^{(1)})$ (for $n=1$) and  $\Atilde(C_n^{(1)})$ (for $n\geq 2$). For this purpose, let us also consider a slightly larger group, which we will denote by  $\Atilde^c(A_{2n}^{(2)})$, defined as the group generated by  $\Atilde(A_{2n}^{(2)})$  and a central element $X_{\frac{1}{2}\d}$ such that $X^2_{\frac{1}{2}\d}=X_\d$. Also, to be able to state the comparison result more concisely we adopt the following convention.

\begin{convention}\label{bigdaddy}
In what follows $C_1^{(1)}$ refers to the affine root system of type $A_1^{(1)}$ and this extends also to all the objects associated affine root systems (e.g. affine/double affine Weyl group, affine/double affine Artin group).
\end{convention}

For $n\geq 1$, the affine Artin groups for the affine root systems of type $C_n^{(1)}$ and $A_{2n}^{(2)}$ are generated by generators satisfying the same set of Coxeter relations. We denote these generators by the same symbols $T_0, T_1, \dots, T_n$. We can also realize both affine root systems inside the same vector space. Let $V$ be an $n+1$--dimensional $\Re$-vector space with basis $\d, \eps_1, \dots, \eps_n$ and bilinear form $(\cdot, \cdot)$ with kernel $\Re\d$ and for which $\eps_1, \dots, \eps_n$ are orthonormal. With this notation, the simple affine roots in the usual realization of affine root system of type $C_n^{(1)}$ are $$\d-\sqrt{2}\eps_1, (\eps_1-\eps_2)/\sqrt{2}, \dots, (\eps_{n-1}-\eps_n)/\sqrt{2}, \sqrt{2}\eps_n,$$
and the finite co-root lattice is $\oplus_{i=1}^n \Z\sqrt{2}\eps_i$. The simple affine roots in the usual realization of affine root system of type $A_{2n}^{(2)}$ are $$\d/2-\eps_1, \eps_1-\eps_2, \dots, \eps_{n-1}-\eps_n, 2\eps_n,$$
and the finite co-root lattice is $\oplus_{i=1}^n \Z\eps_i$.

\begin{Prop} \label{comparison}
Let $n\geq 1$.
\begin{enumerate}[label={\roman*)}]
\item The map that sends $T_0$ to $T_0$, $T_i$ to $T_i$, $X_{\sqrt{2}\eps_i}$ to $X_{\eps_i}$ for $1\leq i\leq n$,   and $X_\d$ to $X_{\d}$, extends to a surjective group morphism
 $$\Atilde(C_n^{(1)})\to \Atilde(A_{2n}^{(2)})$$
 whose kernel is the normal subgroup generated by $X_\d(T_0^{-1} X_{-\sqrt{2}\eps_1})^2$.
\item The map that sends $T_0$ to $T_0$, $T_i$ to $T_i$, $X_{\sqrt{2}\eps_i}$ to $X_{\eps_i}$ for $1\leq i\leq n$,  and $X_\d$ to $X_{\frac{1}{2}\d}$, extends to a surjective group morphism
 $$\Atilde(C_n^{(1)})\to \Atilde^c(A_{2n}^{(2)})$$
 whose kernel is the normal subgroup generated by $(T_0^{-1} X_{\a_0^\vee})^2$.
 \end{enumerate}
\end{Prop}
\begin{proof}
Straightforward from Proposition \ref{second-presentation}.
\end{proof}

\subsection{}\label{sec: nonreduced}

There are irreducible affine root systems that are nonreduced. We refer to \cite{MacAff}*{\S1.3} for the full list of irreducible affine root systems, noting that in \cite{MacAff} the labelling of the reduced irreducible affine root systems differs from the one in \cite{KacInf}.  We will use the labelling in \cite{MacAff}*{\S1.3} for the nonreduced irreducible affine root systems. The construction of the topological space $\widetilde{X}$, and consequently of the double affine Artin group, make sense for any irreducible affine root system $R$, reduced or nonreduced. However, it is clear form the definition that $\widetilde{X}$ only depends on the set 
$$
R_{nm}:=\{\a\in R~|~2\a\not \in R\}
$$
of non-multipliable roots in $R$, which is always an irreducible reduced affine root system (in the sense of Macdonald). Therefore, an double affine Artin group associated to a nonreduced irreducible affine root system is isomorphic to a double affine Artin group associated to a reduced irreducible affine root system. The precise correspondence is specified in Table \ref{table: nonreduced}.
\begin{table}[ht]
\caption{Irreducible nonreduced affine root systems}
\label{table: nonreduced}
\begin{tabular}{ l  l}
$R$ & $R_{nm}$ \\ \hline
$(BC_n, C_n), n\geq 1$& $C_n^{(1)}$\\
$(C_n^\vee, BC_n), n\geq 1$ & $A_{2n}^{(2)}$\\
$(B_n, B_n^\vee), n\geq 3$ & $A_{2n-1}^{(2)}$\\
$(C_n^\vee, C_n), n\geq 1$& $C_n^{(1)}$\\
$(C_2, C_2^\vee)$  & $A_{3}^{(2)}$
\end{tabular}
\end{table}


\section{A refinement of van der Lek's presentation}\label{refinement}

\subsection{}  As a result of Proposition \ref{comparison} we are able to remove the double affine Artin group  $\Atilde(A_{2n}^{(2)})$ from the considerations that follow as it can be studied through its relationship with $\Atilde(C_n^{(1)})$. Hereafter, unless otherwise stated, we assume that $A$ is an indecomposable affine Cartan matrix different from $A_{2n}^{(2)}$.

\subsection{}

Let us start by further analyzing the presentation of the double affine Artin group given in Proposition \ref{second-presentation}. We focus on the relations \eqref{eq1v2} and \eqref{eq2v2}.

\begin{description}
\item[Type $(0,j)_0$] These are relations associated to $1\leq j\leq n$ such that $$2r_{j0}=-(\a_j^\vee,\a_0)=(\a_j^\vee,\th),$$ with $r_{j0}$ a non-negative integer. For such a $j$, if we set $\mu_j=\a_j^\vee+r_{j0}\a_0^\vee$, the following relation holds
\begin{equation}
T_0X_{\mu_j}=X_{\mu_j} T_0.
\end{equation}
Since $X_{\d}$ is central, we can replace $\mu_j$ in the above relation by $\a_j^\vee-r_{j0}\th^\vee$. The only possible even non-zero value for the scalar product $(\a_j^\vee,\th)$ is $2$ and this can happen only if $A=C_n^{(1)}$, $n\geq 1$. Therefore, the relations of this type are
\begin{subequations} 
 \begin{alignat}{2} \label{eq3}
& T_0X_{\a_j^\vee} = X_{\a_j^\vee} T_0&& \text{if}\ \   (\a_j^\vee,\th)=0,\\ \label{eq4-2}
& T_0X_{\a_j^\vee-\th^\vee} = X_{\a_j^\vee-\th^\vee} T_0 \quad && \text{if}\ \   (\a_j^\vee,\th)=2.
\end{alignat}
\end{subequations}
Note that the relation \eqref{eq4-2} is present only for $A=C_n^{(1)}$, $n\geq 2$.

\item[Type $(0,j)_1$]  These are relations associated to $1\leq j\leq n$ such that $$2r_{j0}+1=-(\a_j^\vee,\a_0)=(\a_j^\vee,\th),$$ with $r_{j0}$ a non-negative integer. For such a $j$, if we set $\mu_j=\a_j^\vee+r_{j0}\a_0^\vee$, the following relation holds
\begin{equation}
T_0X_{\mu_j} T_0=X_{s_i(\mu_j)}.
\end{equation}
The only odd value the scalar product $(\a_j^\vee,\th)$ could take is $1$ and this happens only for $A \neq C_n^{(1)}$, $n\geq 1$. Therefore, the only relations of this type are
\begin{equation}\label{eq4}
T_0X_{\a_j^\vee} T_0 = X_{\a_j^\vee+\a_0^\vee}\quad  \text{if}\ \   (\a_j^\vee,\th)=1.
\end{equation}
Note that the relations \eqref{eq4} are not present if $A=C_n^{(1)}$, $n\geq 1$.
\end{description}

\subsection{}
The following result is useful in reducing the number of necessary relations in the above presentation. 
\begin{Lm}\label{propagation}
Let $\b$ and $\gamma$ be simple roots  whose nodes in the Dynkin diagram are connected, but none of them is connected to the node of $\a_0$. Assume that $\b$ is longer that $\gamma$ or they have the same length. Then, if $T_0$ commutes with $X_{\b^\vee}$ then $T_0$ commutes with $X_{\gamma^\vee}$.
\end{Lm}
\begin{proof}
The hypothesis implies that $ (\b^\vee,\gamma)=-1 $ and $s_\gamma(\b^\vee)=\b^\vee+\gamma^\vee$. From equation Definition \ref{defcherednik}b) we know that
$$
T_{s_\gamma} X_{\b^\vee} T_{s_\gamma} =X_{\gamma^\vee+\b^\vee}
$$
or, equivalently, $X_{\gamma^\vee}=T_{s_\gamma} X_{\b^\vee} T_{s_\gamma} X_{-\b^\vee}$. From this expression it is clear that if  $T_0$ commutes with $X_{\b^\vee}$ then $T_0$ commutes with $X_{\gamma^\vee}$.
\end{proof}

\subsection{}

Some of the properties of affine root systems call for separate treatment of twisted and untwisted affine root systems. The properties that distinguish between untwisted and twisted roots systems that are most relevant for our purposes are the following. For untwisted root systems $\a_0$ and $\th$ have the longest possible length among the real roots and $M=\Qring^\vee$. Twisted root systems necessarily have real roots of two different lengths, $\a_0$ and $\th$ are short roots and $M=\Qring$.

We will denote by $i_\th$ the node in the finite Dynkin diagram that is connected to the affine node.  If there are two such nodes (which is the case for $S(A_n^{(1)})$, $n\geq 2$) we choose one of them. We denote by $\ell_0$ the number of laces by which the nodes corresponding to $0$ and $i_\th$ are connected. Note that $\ell_0$ is always 1, except for $A=  C_{n}^{(1)}, D_{n+1}^{(2)}$, $n\geq 2$  for which it takes the value 2, and for $A= A_{1}^{(1)}$  for which it takes the value 4. For twisted affine root systems $\theta$ is not the highest root; in this case we denote by $i_\ph$ the unique node in the finite Dynkin diagram for which the corresponding simple root is not orthogonal on $\ph$. Remark that  $(\ph^\vee,\a_{i_\ph})=1$.

\subsection{}\label{untwisted}
In this section we assume that $R$ is an irreducible untwisted affine root system.

\begin{Prop}\label{reduction2}
In Proposition \ref{second-presentation} the relations   \eqref{eq1v2} and \eqref{eq2v2} can be replaced by  the following relation
\begin{subequations} 
 \begin{alignat}{2}\label{eq5-1}
& T_0X_{\a_{i_\th}^\vee} T_0 = X_{\a_{i_\th}^\vee+\a_0^\vee} && \text{if}\ \   \ell_0=1,\\ \label{eq5-2}
& T_0X_{\a_{i_\th}^\vee-\th^\vee} = X_{\a_{i_\th}^\vee-\th^\vee} T_0 \quad && \text{if}\ \   \ell_0=2.
\end{alignat}
\end{subequations}
\end{Prop}
\begin{proof}
Assume that  $\ell_0=1$. Remark first that the relation \eqref{eq5-1} must be satisfied by Definition \ref{defcherednik}b). We will show that the relations \eqref{eq3} follow from the relation \eqref{eq5-1} and the relations in Proposition \ref{second-presentation}a),d). By using Lemma \ref{propagation} we see that the  type $(0,j)_0$ relations \eqref{eq3} are implied by the knowledge of the braid relations and of the commutation of $T_0$ with $X_{\b^\vee}$, where $\b$ is any simple root neighbor of $\a_{i_\th}$. The commutation of $T_0$ and $X_{\b^\vee}$ holds indeed: since $\a_{i_\th}$ is a long root we have $(\beta,\a_{i_\th}^\vee)=-1$ and
$$
X_{\b^\vee}=T_{s_\b} X_{\a_{i_\th}^\vee} T_{s_\b} X_{-\a_{i_\th}^\vee}.
$$
Now,
\begin{align*}
T_0 X_{\b^\vee} T_0^{-1} &= T_0 T_{s_\b} X_{\a_{i_\th}^\vee} T_{s_\b} X_{-\a_{i_\th}^\vee}T_0^{-1} & & \\
&= T_{s_\b} T_0 X_{\a_{i_\th}^\vee} T_0 T_{s_\b} T_0^{-1} X_{-\a_{i_\th}^\vee}T_0^{-1} & & \text{by the braid relations for $T_0$ and $T_{s_\b}$}  \\
&= T_{s_\b} X_{\a_{i_\th}^\vee+\a_0^\vee}  T_{s_\b} X_{-\a_{i_\th}^\vee-\a_0^\vee} & & \text{by \eqref{eq5-1}} \\
&= T_{s_\b} X_{\a_{i_\th}^\vee}  T_{s_\b} X_{-\a_{i_\th}^\vee} & & \text{by \eqref{dc1v2} and \eqref{dc3v2}} \\
&= X_{\b^\vee}. & &
\end{align*}
For $A=A_n^{(1)}$, $n\geq 2$, there are two type $(0,j)_1$ relations: \eqref{eq5-1}  and another one, associated to the second neighbor (let us call it $\a^\prime$) of the affine simple root in the Dynkin diagram. A straightforward computation, which exploits the fact that  $T_0$ commutes with $X_{\a_{i_\th}^\vee+\a^{\prime\vee} +\a_0^\vee}$, (fact which is a consequence of the commuting relations proved above) will show that \eqref{eq4} for $\a^\prime$ holds. For completeness, let us explain the details:
\begin{align*}
T_0 X_{\a^{\prime\vee}} T_0 &= T_0 X_{-\a_{i_\th}^\vee-\a_0^\vee} X_{\a_{i_\th}^\vee+\a^{\prime\vee}+\a_0^\vee} T_0 & &\\
&= T_0 X_{-\a_{i_\th}^\vee-\a_0^\vee} T_0 X_{\a_{i_\th}^\vee+\a^{\prime\vee}+\a_0^\vee} & & \text{by \eqref{eq3}}  \\
&= X_{-\a_{i_\th}^\vee}  X_{\a_{i_\th}^\vee+\a^{\prime\vee}+\a_0^\vee} & & \text{by \eqref{eq5-1}} \\
&=  X_{\a^{\prime\vee}+\a_0^\vee}. & &
\end{align*}
The proof of our result in the case $\ell_0=1$ is now completed. The case $\ell_0=2$ is treated completely similarly.
\end{proof}

\subsection{} \label{twisted}

In this section we assume that $R$ is an irreducible twisted affine root system. Recall that the twisted affine root systems are not simply-laced, $\th\neq \ph$, and $\Qring\subset \Qring^\vee$. In this situation there are two other positive roots that play a distinguished role: the long root $\php=-s_{\th}(\ph)$ and the short root $\thp=-s_{\ph}(\th)$. Remark that $\php^\vee=\th^\vee-\ph^\vee$ and $\thp=\ph-\th$. Set $\Phip=T_{s_{\php}}$, $\Thetap=T_{s_{\thp}}$, $\Psi=T_{s_\ph s_\th}^{-1}$, and  $\isP=T_{s_\th s_\ph}^{-1}$.

\begin{Prop}\label{reduction-2}
In Proposition \ref{second-presentation} the relations   \eqref{eq1v2} and \eqref{eq2v2} can be replaced by  the following relation
\begin{equation}\label{eq5}
T_0X_{\ph^\vee}T_0=X_{\ph^\vee+\a_0^\vee}.
\end{equation}
\end{Prop}
\begin{proof}
First, remark that the relation \eqref{eq5} must be satisfied in $\Atilde(R)$. Indeed, using basic facts about root systems, which can be found for example in \cite{BouLie}*{VI, \S 1.3}, one obtains that $(\ph^\vee,\th)=1$ and therefore $(\ph^\vee,\a_0)=-1$. Hence, by Definition \ref{defcherednik}{b)},  we obtain that \eqref{eq5} is satisfied .

Let us explain how the type $(0,j)_1$ relation \eqref{eq4} follows from \eqref{eq5} and relations in Proposition \ref{second-presentation}a),d). Indeed,  if $\ell_0=1$ or, equivalently,  if $(\th^\vee,\a_{i_\th})=1$, we have
\begin{align*}
T_0X_{\a_{i_\th}^\vee} T_0&=T_0X_{\a_{i_\th}^\vee-\th^\vee+\ph^\vee}X_{\th^\vee-\ph^\vee} T_0& &\\
&=T_0T_{i_\th} X_{-\th^\vee+\ph^\vee} T_{i_\th} X_{\th^\vee-\ph^\vee} T_0 & &\text{by \eqref{dc2v2}} \\
&=T_0T_{i_\th} T_0 T_0^{-1}X_{-\th^\vee+\ph^\vee}T_0^{-1} T_0T_{i_\th} T_0^{-1} T_0 X_{\th^\vee-\ph^\vee} T_0& &\\
&=T_0T_{i_\th} T_0 X_{-\delta+\ph^\vee} T_0 T_{i_\th} T_0^{-1} X_{\delta-\ph^\vee}& &\text{by \eqref{eq5} and \eqref{dc3v2}}\\ 
&=T_{s_\a} T_0 T_{i_\th} X_{-\delta+\ph^\vee} T_{i_\th}^{-1} T_0 T_{i_\th} X_{\delta-\ph^\vee}& &\text{by the $T_0$, $T_{i_\th}$ braid relations}\\ 
&=T_{i_\th} T_0  X_{-\delta+\ph^\vee} T_0 T_{i_\th} X_{\delta-\ph^\vee}& &\text{by  \eqref{dc1v2} and \eqref{dc3v2}}\\ 
&=T_{i_\th}  X_{-\th^\vee+\ph^\vee} T_{i_\th} X_{\delta-\ph^\vee}& &\text{by \eqref{eq5}}\\ 
&= X_{\a_{i_\th}^\vee+\a_0^\vee}& &\text{by \eqref{dc2v2}}\\ 
\end{align*}
If $\ell_0=2$  or, equivalently, if $(\th^\vee,\a_{i_\th})=2$, then $$s_{i_\th} s_\th(-\ph^\vee)=s_{i_\th}(\th^\vee-\ph^\vee)=\th^\vee-\ph^\vee-2\a_{i_\th}^\vee$$ is a negative co-root since $(\th-\ph^\vee-2\a_{i_\th}^\vee,\th)=-1$. Taking into account that $\a$ is a simple root, this implies that $\a_{i_\th}^\vee=\th^\vee-\ph^\vee$. In this case, \eqref{eq4} is precisely \eqref{eq5}.

Let us now explain how the type $(0,j)_0$ relations \eqref{eq3} follow from \eqref{eq5} and relations in Proposition \ref{second-presentation}a),d). 

We first establish the commutation relation for the simple root $\a_{i_\ph}$. Indeed,
\begin{align*}
T_0X_{\a_{i_\ph}^\vee} T_0^{-1}&=T_0X_{\a_{i_\ph}^\vee+\ph^\vee}X_{-\ph^\vee} T_0^{-1}& &\\
&=T_0T_{i_\ph}^{-1} X_{\ph^\vee} T_{i_\ph}^{-1} X_{-\ph^\vee} T_0^{-1} & &\text{by \eqref{dc2v2}} \\
&=T_0T_{i_\ph}^{-1}T_0^{-1} T_0 X_{\ph^\vee}T_0 T_0^{-1} T_{i_\ph}^{-1} T_0 T_0^{-1 }X_{-\ph^\vee} T_0^{-1} & &\\
&=T_0T_{i_\ph}^{-1}T_0^{-1}  X_{\ph^\vee+\a_0^\vee} T_0^{-1} T_{i_\ph}^{-1} T_0 X_{-\ph^\vee-\a_0^\vee} & &\text{by \eqref{eq5}}\\ 
&=T_{i_\ph}^{-1}  X_{\ph^\vee+\a_0^\vee} T_{i_\ph}^{-1} X_{-\ph^\vee-\a_0^\vee} & &\text{by the $T_0$, $T_{i_\ph}$ braid relations }\\ 
&=T_{i_\ph}^{-1}  X_{\ph^\vee-\th^\vee} T_{i_\ph}^{-1} X_{-\ph^\vee+\th^\vee}& &\text{by \eqref{dc3v2}}\\ 
&= X_{\a_{i_\ph}^\vee}& &\text{by  \eqref{dc2v2}}\\ 
\end{align*}
If $(\th^\vee,\a_{i_\th})=1$, we obtain in a similar way a commutation relation between $T_0$ and $X_{\th^\vee-\ph^\vee-\a_{i_\th}^\vee}$:
\begin{align*}
T_0X_{\th^\vee-\ph^\vee-\a_{i_\th}^\vee} T_0^{-1}&=T_0T_{i_\th}^{-1} X_{\th^\vee-\ph^\vee} T_{i_\th}^{-1} T_0^{-1} & &\text{by \eqref{dc2v2}} \\
&=T_0T_{i_\th}^{-1}T_0^{-1}T_0 X_{\th^\vee-\ph^\vee} T_0T_0^{-1}T_{i_\th}^{-1} T_0^{-1} & &\\
&=T_0T_{i_\th}^{-1}T_0^{-1}X_{\delta-\ph^\vee}T_0^{-1}T_{i_\th}^{-1} T_0^{-1}& &\text{by \eqref{eq5}}\\ 
&=T_{i_\th}^{-1}T_0^{-1}T_{i_\th} X_{\delta-\ph^\vee}T_{i_\th}^{-1}T_0^{-1}T_{i_\th}^{-1} & &\text{by the $T_0$, $T_{i_\th}$ braid relations}\\ 
&=T_{i_\th}^{-1}T_0^{-1} X_{\delta-\ph^\vee}T_0^{-1}T_{i_\th}^{-1} & &\text{by  \eqref{dc1v2} and \eqref{dc3v2}}\\ 
&=T_{i_\th}^{-1}X_{\th^\vee-\ph^\vee}T_{i_\th}^{-1} & &\text{by \eqref{eq5}}\\ 
&= X_{\th^\vee-\ph^\vee-\a_{i_\th}^\vee}& &\text{by \eqref{dc2v2}}\\ 
\end{align*}
This commutation relation implies a relation of type $(0,j)_0$. Indeed, since $(\th^\vee,\a_{i_\th})=1$, $\th^\vee-\ph^\vee-\a_{i_\th}^\vee$ is a the co-root of a long root, being equal to $s_{i_\th} s_\th(-\ph^\vee)$. Furthermore, it is orthogonal on $\th$. Therefore, the reflection $s_{\th^\vee-\ph^\vee-\a_{i_\th}^\vee}$ is conjugate (in the stabilizer of $\th$ in $\W$) to a simple reflection $s_{j_0}$, with respect to a long simple root $\a_{j_0}$. We can write 
$$
s_{j_p}\cdots s_{j_1}(\a_{j_0}^\vee)=\th^\vee-\ph^\vee-\a_{i_\th}^\vee
$$
such that $s_{j_k}(\th)=\th$ (or, equivalently the nodes $0$ and $j_k$ are not connected in the Dynkin diagram)  and, because $\a_{j_0}$ is long,  $(s_{j_{k-1}}\cdots s_{j_1}(\a_{j_0}^\vee), \a_{j_k})=-1$ for all $1\leq k\leq p$. By  \eqref{dc2v2} we obtain
$$
X_{\a_{j_0}^\vee}=T_{j_1}^{-1}\cdots T_{j_p}^{-1}X_{\th^\vee-\ph^\vee-\a^\vee}T_{j_p}^{-1}\cdots T_{j_1}^{-1}
$$
The fact that $T_0$ commutes with $X_{\th^\vee-\ph^\vee-\a_{i_\th}^\vee}$ and $T_{j_k}$, for all $1\leq k\leq p$, implies that $T_0$ and $X_{\a_{j_0}^\vee}$ commute.

By examining the twisted Dynkin diagrams we can see that Lemma \ref{propagation} produces all the relations of type $(0,j)_0$ starting from the commutation relation between $T_0$ and $X_{\a_{i_\ph}^\vee}$ except for $A^{(2)}_{2n-1}$, $n\geq 3$, for which we also need to use the commutation relation between $T_0$ and $X_{\a_{j_0}^\vee}$, where, in this case, $\a_{j_0}$ is the unique long simple root.
\end{proof} 
\section{The Coxeter presentation}\label{sec: Coxeter}

\subsection{} As we explain in Appendix \ref{sec: nonCoxeter}, double affine Weyl groups are not Coxeter groups, the most basic obstruction being the fact that they have infinite center. As we also point out in Appendix \ref{sec: nonCoxeter}, taking a quotient by a proper subgroup of the center, or removing the center altogether do not produce Coxeter groups either. Therefore, the other possibility for relating double affine Weyl groups to Coxeter groups is to resolve the center. We show that this is indeed possible and the corresponding relationship carries over to the double affine Artin groups.

\subsection{}
The goal of this section is to obtain a Coxeter-type presentation for double affine Artin groups and consequently for double affine  Weyl groups. This presentation involves realizing the double affine Artin groups as quotients of the Coxeter braid  groups associated to the double affine Coxeter diagrams  listed in Figure \ref{dddot-diagrams}, Figure \ref{ddot-diagrams}, and Figure \ref{fH-diagrams}. To be consistent with Convention \ref{bigdaddy}, $\DDDot{C}_1$ will refer to the diagram $\DDDot{A}_1$, and $\DDot{C}_1$ will refer to the diagram $\DDot{A}_1$. Remark that if we replace a double or triple node with a regular node, then the diagrams of type $\DDDot{X}_n$, $\DDot{X}_n$ are precisely the affine Coxeter diagrams of type $\widetilde{X}_n$ (in the notation of \cite{BouLie}*{Ch. VI, \S 4.3}). In fact, for any  double affine Coxeter diagram $\DDDot{X}_n$,  $\fH{X}_n$, or $\DDot{X}_n$, if we erase all the affine nodes we obtain the finite Coxeter diagram $X_n$ and if we erase all but one affine node we obtain an affine Coxeter diagram.

\subsection{} Let us fix notation for the generators of the Coxeter braid group associated to a double affine Coxeter diagram. The generators associated to the finite nodes will be denoted by $\T_1,\dots,\T_n$.

To a fixed affine node of $\DDDot{X}_n$, $\DDot{X}_n$, or $\fH{X}_n$ we can associate an element of $B(X_n)$ as follows. Temporarily denote by $\dot{X}_n$ the affine Coxeter diagram obtained from the double Coxeter diagram by erasing the other affine nodes. The vector space supporting the reflection representation of $C(X_n)$ can be naturally realized as a subspace of the reflection representation of $C(\dot{X}_n)$. In the reflection representation of $C(\dot{X}_n)$, the action of the reflection associated to the affine node of $\dot{X}_n$ coincides, when restricted to the reflection representation of $C(X_n)$, to the action of a unique reflection in $C(X_n)$. Hence, to the affine node we can associate a unique element of $C({X}_n)$ and, by applying $t_{X_n}$, a unique element of $B(X_n)$. The affine nodes that produce the same affine Coxeter diagram $\dot{X}_n$ give rise to the same element of $B(X_n)$. 

We fix notation for the affine node that is suggestive of this correspondence. More precisely, for $\DDDot{X}_n$, the generators corresponding to the affine nodes will be denoted by $\varTheta_{01}$, $\varTheta_{02}$, $\varTheta_{03}$ and the associated element of $B(X_n)$ will be denoted by $\varTheta$.  For $\DDot{X}_n$ and $\fH{X}_n$, the affine nodes will be denoted by $\Ti$ and $\Tiii$ and the associated elements of $B(X_n)$ will be denoted by $\varTheta$ and $\varPhi$, respectively. Note that for $\DDot{C}_n$ we have $\varTheta=\varPhi$.

As a general rule, we denote by $\T_{i_\th}$ the generator corresponding to the finite node connected to the affine node corresponding to $\varTheta_0, \varTheta_{01}, \varTheta_{02}$, or $\varTheta_{03}$, and we denote by $\ell_0$ the number of laces between these two nodes.  If there are two such finite roots (which is the case for $\DDDot{A}_n$, $n\geq 2$) we choose one of them. Note that $\ell_0$ is always 1, except for $\DDDot{C}_{n}$, $\DDot{C}_n$, $n\geq 2$,  $\fH{B}_{2}$, and possibly  $\fH{B}_{n}/\fH{C}_n$, $n\geq 3$  for which it takes the value 2, and  for $\DDDot{A}_{1}$,  $\DDot{A}_{1}$  for which it takes the value 4. Also, we denote by $\T_{i_\ph}$ the generator corresponding to the finite node connected to the affine node corresponding to $\varPhi_0$.

For the double affine Coxeter diagrams of type $\fH{X}_n$, there are two conjugacy classes of reflections in $C(X_n)$ and $\pi_{X_n}(\varTheta)$ and $\pi_{X_n}(\varPhi)$ are the reflections of maximal length in each conjugacy class. In particular, $\varTheta$ and $\varPhi$ coincide (up to a permutation) with $\Theta$ and $\Phi$ in Section \ref{app: finite}. With the notation,
\begin{equation}\label{eq: tildedef}
\varThetap=t_{X_n}\pi_{X_n}(\varPhi\varTheta\varPhi),\quad   \varPhip=t_{X_n}\pi_{X_n}(\varTheta\varPhi\varTheta),
\end{equation}
we have, according to Lemma \ref{Psi}iii),
\begin{equation}
\varThetap\varTheta=\varPhi\varPhip\quad \text{and}\quad \varPhip\varPhi=\varTheta\varThetap.
\end{equation}
Remark that the two relations are permuted if $\varTheta$ and $\varPhi$ are interchanged. We denote
\begin{equation}\label{eq: Phipdef}
\varPsi= \varPhip\varTheta^{-1}=\varPhi^{-1}\varThetap,\quad  \varisP=\varThetap\varPhi^{-1}=\varTheta^{-1}\varPhip.
\end{equation}

To emphasize the distinguished role played by these elements lets us record the following relations in  $B(\fH{X}_n)$, each being in fact a relation in an affine Coxeter braid group.

\begin{Prop}\label{B2-rels} If $X_n$ is double laced, in $B(\fH{X}_n)$ the following hold,
\begin{enumerate}[label={\roman*)}]
\item $\Tiii$ and $\varPhip$ satisfy the $0$-braid relation;
\item $\Tiii$ and $\varThetap$ satisfy the $2$-braid relation;
\item $\Ti$ and $\varThetap$ satisfy the $0$-braid relation;
\item $\Ti$ and $\varPhip$ satisfy the $2$-braid relation.
\end{enumerate}
Therefore, $\Ti$, $\Tiii$,  $\varThetap$, and $\varPhip$ satisfy the braid relations specified by the Coxeter diagram in Figure \ref{fig: labelledB2}.
\begin{figure}[ht]
\caption{Labelled $\fH{B}_2/\fH{C}_2$ diagram}\label{fig: labelledB2}
\begin{center}
  \begin{tikzpicture}[scale=.4]
      \draw (2.3,1.3) node[anchor=south]  {$\varPhip$};
      \draw (-0.3,1.3) node[anchor=south]  {$\varThetap$};
      \draw (2.3,-1.3) node[anchor=north]  {$\Ti$};
      \draw (-0.3,-1.3) node[anchor=north]  {$\Tiii$};
    \foreach \x in {45,135,225,315}
    \draw[thick, xshift=1cm] (\x: 1.4cm) circle [radius=3mm];
    \draw[xshift=1cm, shift=(135:1.4cm), thick] (20: 3mm) -- +(1.42 cm, 0);
    \draw[xshift=1cm, shift=(135:1.4cm), thick] (-20: 3 mm) -- +(1.42 cm, 0);
    \draw[xshift=1cm, shift=(135:1.4cm), thick] (-70: 3 mm) -- +(0,-1.42 cm);
    \draw[xshift=1cm, shift=(135:1.4cm), thick] (-110: 3 mm) -- +(0,-1.42 cm);
    \draw[xshift=1cm, shift=(-135:1.4cm), thick] (20: 3mm) -- +(1.42 cm, 0);
    \draw[xshift=1cm, shift=(-135:1.4cm), thick] (-20: 3 mm) -- +(1.42 cm, 0);
    \draw[xshift=1cm, shift=(45:1.4cm), thick] (-70: 3 mm) -- +(0,-1.42 cm);
    \draw[xshift=1cm, shift=(45:1.4cm), thick] (-110: 3 mm) -- +(0,-1.42 cm);
  \end{tikzpicture}
\end{center}
\end{figure}
\end{Prop}
\begin{proof}
We present arguments for the first two claims, the last two claims being proved similarly. The first two claims involve only the subgroup of $B(\fH{X}_n)$ generated by $\Tiii, \T_1,\dots, \T_n$, which without loss of generality we may assume that is the Coxeter braid group associated to the untwisted affine Dynkin diagram ${X}^{(1)}_n$. In particular, we adopt the notation used in this context we freely make reference to the corresponding finite root system as set up in \S \ref{app-nsl}. 

Now, the first claim follows from the fact that, $(\ph, \php)=0$. If $(\ph, \a_{i_\ph}^\vee)(\ph^\vee,\a_{i_\ph})=2$ then $\Tiii$ and $\T_{i_\ph}$ satisfy the $2$-braid relation, but also, Lemma \ref{lemma-wy-artin}iv),vi) implies that $\T_{i_\ph}=\varThetap$. Therefore, $\Tiii$ and $\varThetap$ satisfy the $2$-braid relation. 

If $(\ph, \a_{i_\ph}^\vee)(\ph^\vee,\a_{i_\ph})=1$ then  $\Tiii$ and $\T_{i_\ph}$ satisfy the $1$-braid relation. We have $(\a_{i_\ph}^\vee, \thp)=1$ and $(\ph^\vee, s_{i_\ph}(\thp))=0$. This implies that $\Tiii$ and $\T_{s_{i_\ph}s_{\thp} s_{i_\ph}}$ commute. 

Note that since $(\th, \thp)=0$, $\php$ is a root in the standard parabolic sub-system of $\Rring$ obtained by removing $\a_{i_\th}$. In fact, the only simple root in this parabolic sub-system that has positive scalar product with $\thp$ is $\a_{i_\ph}$ and $(\a_{i_\ph}, \thp^\vee)(\a_{i_\ph}^\vee, \thp)=2$. Lemma \ref{lemma-extrabraid-2} implies that $\T_{s_{i_\ph}}$ and $\T_{s_{i_\ph}s_{\thp} s_{i_\ph}}$ satisfy the $2$-braid relation. From Lemma \ref{basic-braid-rels2}ii) for $a=\Tiii$, $b=\T_{i_\ph}$, and $c=\T_{s_{i_\ph}s_{\thp} s_{i_\ph}}$ we obtain that $\Tiii$ and $\T_{i_\ph}\T_{s_{i_\ph}s_{\thp} s_{i_\ph}}\T_{i_\ph}$ satisfy the $2$-braid relation.  By Lemma \ref{lemma-extrabraid-1} we have $\varThetap=\T_{i_\ph}\T_{s_{i_\ph}s_{\thp} s_{i_\ph}}\T_{i_\ph}$. In conclusion, $\Tiii$ and $\varThetap$ satisfy the $2$-braid relation. The fact that $\varThetap$ and $\varPhip$ satisfy the $2$-braid relation is proved in Lemma \ref{lemma-explicit-2}iii).
\end{proof}

For the double affine Coxeter diagrams of type $\DDot{C}_n$, as we already noted, $\varPhi=\varTheta$. For notational consistency, we still adopt the notation set up in \eqref{eq: tildedef} and \eqref{eq: Phipdef}, but remark that in this case
\begin{equation}\label{eq: twodotcollapsing}
\varThetap=\varPhip=\varTheta=\varPhi \quad \text{and}\quad \varPsi=\varisP=1.
\end{equation}

\subsection{}\label{sec: def-dddot}
We are now ready to associate to each double affine Coxeter diagram a quotient of the corresponding Coxeter braid group. We start with the double affine Coxeter diagrams of type $\DDDot{X}_n$.

\begin{Def}\label{def: untwisted} The group $\B(\DDDot{X}_n)$ is defined as the quotient of $B(\DDDot{X}_n)$ by the normal subgroup generated by the following relations:
\begin{enumerate}[leftmargin=40pt, label={\alph*)}]
\item The element
\begin{equation}\label{centralrel-1}
\C:=\varTheta_{01}\varTheta_{02}\varTheta_{03}\varTheta
\end{equation} 
is central.
\item If $\ell_0=2$, for $(i,j)\in \{(1,2),~(1,3),~(2,3)\}$ we have
\begin{equation}\label{ellbraid}
\varTheta_{0i}\T_{i_\th}^{-1}\varTheta_{0j}\T_{i_\th}=\T_{i_\th}^{-1}\varTheta_{0j}\T_{i_\th} \varTheta_{0i}.
\end{equation}
\end{enumerate}
\end{Def}

Let us briefly comment on the minimality of the set of generators and of the set of relations in the definition of $\B(\DDDot{X}_n)$.  To be able to use standard results on affine Coxeter braid groups we remark that the affine Coxeter diagram obtained by erasing from $\DDDot{X}_n$ all but one affine node coincide with the Coxeter diagram underlying the affine Dynkin diagram of type $X_n^{(1)}$. Therefore, we will freely use any result that holds in the Coxeter braid of type $X_n^{(1)}$. We note that the elements of $B(X_n)\subset B(\DDDot{X}_n)$ denoted by  $\varTheta$ and  $\T_{i_\th}$ correspond respectively to the elements of $B(X_n)\subset B({X}^{(1)}_n)$ denoted by  $\Theta$  and $T_{i_\th}$.

\begin{Lm}
In $\B(\DDDot{X}_n)$, if  $\ell_0=1$, $\varTheta_{02}$ can be expressed in terms of the other generators.
\end{Lm}
\begin{proof}
Let us start with an immediate consequence of \eqref{centralrel-1}.
Since we can use the braid relations to write
$$
\varTheta_{02}=\T_{i_\th}\varTheta_{02}\T_{i_\th}\varTheta_{02}^{-1}\T_{i_\th}^{-1},
$$
using \eqref{centralrel-1} we obtain
$$
\varTheta_{02}=\T_{i_\th}\varTheta_{02}\C^{-1}\T_{i_\th}\varTheta_{02}^{-1}\C\T_{i_\th}^{-1},
$$
which becomes
$$
\varTheta_{02}=\T_{i_\th}\varTheta_{01}^{-1}\varTheta^{-1}\varTheta_{03}^{-1}\T_{i_\th}\varTheta_{03}\varTheta\varTheta_{01}\T_{i_\th}^{-1}.
$$
By using the braid relation between $\varTheta_{01}$ and $\T_{i_\th}$ we obtain the equality 
\begin{equation}\label{magic1}
\varTheta_{02}=\varTheta_{01}^{-1}\T_{i_\th}^{-1}\varTheta_{01}\T_{i_\th} \varTheta^{-1}\varTheta_{03}^{-1}\T_{i_\th}\varTheta_{03}\varTheta\varTheta_{01}\T_{i_\th}^{-1},
\end{equation}
which we record for later use.
\end{proof}
\begin{Lm}\label{reducel=2}
In Definition \ref{def: untwisted}, if $\ell_0=2$, only one of the relations
\eqref{ellbraid} should be imposed.
\end{Lm}
\begin{proof}
We will prove that if we impose only one relation in
\eqref{ellbraid}, say
\begin{equation}\label{c}
\varTheta_{01}\T_{i_\th}^{-1}\varTheta_{02}\T_{i_\th}=\T_{i_\th}^{-1}\varTheta_{02}\T_{i_\th}\varTheta_{01}
\end{equation}
the other two easily follow from this one and  the fact that $\C$ is central. We will illustrate this briefly.

The above relation says that $\varTheta_{01}$ and $\T_{i_\th}^{-1}\varTheta_{02}^{-1}\T_{i_\th}$ commute. Since
 $\varTheta_{01}$ also commutes with $\C$, with $\T_{i_\th}^{-1}\varTheta_{01}^{-1}\T_{i_\th}^{-1}$ (this
is just the braid relation between $\varTheta_{01}$ and $\T_{i_\th}$), and with $\T_{i_\th}\varTheta^{-1}\T_{i_\th}$ (from Lemma \ref{lemma-extrabraid-1} and Lemma \ref{lemma-simple-commuting}), it follows that
$\varTheta_{01}$ commutes with their product, which is
$$
\T_{i_\th}^{-1}\varTheta_{02}^{-1}\T_{i_\th}\T_{i_\th}^{-1}\varTheta_{01}^{-1}\T_{i_\th}^{-1}\T_{i_\th}
\varTheta^{-1}\T_{i_\th}\C=\T_{i_\th}^{-1}\varTheta_{03}\T_{i_\th}.
$$
We proved that
$$
\varTheta_{01}\T_{s_\a}^{-1}\varTheta_{03}\T_{s_\a}=\T_{s_\a}^{-1}\varTheta_{03}\T_{s_\a}\varTheta_{01}.
$$
The argument is the same if we choose to keep any other relation.
\end{proof}

\begin{Lm}\label{reducel=2bis}
In Definition \ref{def: untwisted}, if $\ell_0=2$, the braid relations involving $\varTheta_{02}$ are superfluous.
\end{Lm}
\begin{proof}
We can write
\begin{equation}
\varTheta_{02}=\C\varTheta_{01}^{-1}\varTheta^{-1}\varTheta_{03}^{-1}.
\end{equation}
If $i\neq i_\th$, then $\T_i$ and $\varTheta^{-1}\varTheta_{03}^{-1}$ commute by Remark \ref{rem: double braid} and  $\T_i$ and $\varTheta_{01}^{-1}$ commute (the braid relation between $\T_i$ and $\varTheta_{01}$). Keeping in mind \eqref{centralrel-1} we obtain that $\T_i$ and $\varTheta_{02}$ satisfy the 0-braid relation.

Checking that $\T_{i_\th}$ and $\varTheta_{02}$ satisfy the 2-braid relations is equivalent to checking that $\varTheta_{02}$ and $\C\T_{i_\th}\varTheta_{01}^{-1}\varTheta^{-1}\varTheta_{03}^{-1}\T_{i_\th}$ commute. Since
 $\varTheta_{02}$  commutes with $\C$, with $\T_{i_\th}\varTheta_{01}^{-1}\T_{i_\th}^{-1}$ and $\T_{i_\th}^{-1}\varTheta_{03}^{-1}\T_{i_\th}$ (by the relations \eqref{ellbraid}), and with $\T_{i_\th}\varTheta^{-1}\T_{i_\th}$ (from Lemma \ref{lemma-extrabraid-1} and Lemma \ref{lemma-simple-commuting}), it follows that
$\varTheta_{02}$ commutes with their product, which is precisely our claim.
\end{proof}

Let us record the following equivalent presentation for $\B(\DDDot{C}_n)$, $n\geq 1$, which is a direct consequence of Lemma \ref{reducel=2} and Lemma \ref{reducel=2bis}.
\begin{Prop}\label{prop: Cequivpres}
The group $\B(\DDDot{C}_n)$, $n\geq 1$, is the group generated by the elements $\varTheta_{01}$, $\varTheta_{03}$, $\T_i$, $1\leq i\leq n$, and $\C$, with the following relations:
\begin{enumerate}[leftmargin=40pt, label={\alph*)}]
\item The elements $\varTheta_{01}$, $\varTheta_{03}$, $\T_i$, $1\leq i\leq n$, generate the braid group $B(\DDot{C}_n)$ (with $\varTheta_{01}$, $\varTheta_{03}$ associated to the affine nodes).
\item The element $\C$ is central.
\item If $n\geq 2$,  we have
$$
\varTheta_{01}\T_{i_\th}^{-1}\varTheta_{03}\T_{i_\th}=\T_{i_\th}^{-1}\varTheta_{03}\T_{i_\th} \varTheta_{01}.
$$
\end{enumerate}
\end{Prop}

\subsection{}\label{sec: def-ddot} We will now discuss the double affine Coxeter diagrams of type $\fH{X}_n$.

\begin{Def}\label{def: twisted} The group $\B(\fH{X}_n)$ is defined as the quotient of $B(\fH{X}_n)$ by the normal subgroup generated by the following relations:

The element
\begin{equation}\label{centralrel}
\C:=\Tiii\varPhi\Ti\varPsi\Tiii\varTheta\Ti
\end{equation} 
is central.
\end{Def}

\begin{Rem}\label{rem: labeling}
A priori, for each diagram of type $\fH{X}_n$ one can associate two groups, depending on labeling of  the affine nodes by $\varPhi_0$ and $\varTheta_0$. Equivalently, once a  labeling of  the affine nodes by $\varPhi_0$ and $\varTheta_0$ is fixed one can associate two quotients of $B(\fH{X}_n)$ as in Definition \ref{def: twisted}, one corresponding to
\begin{equation}
\C:=\Tiii\varPhi\Ti\varPsi\Tiii\varTheta\Ti,
\end{equation} 
and the other correspoding to 
\begin{equation}
\C^\prime:=  \Ti \varTheta \Tiii   \varisP \Ti   \varPhi   \Tiii.
\end{equation}
We temporarily denote these groups by $\B(\fH{X}_n)$ and respectively $\B(\fH{X}_n)^\prime$.

The double Coxeter diagrams of type $\fH{B}_2/\fH{C}_2$, $\fH{F}_4$, and $\fH{G}_2$ admit an automorphism that exchanges the two affine nodes and this diagram automorphism induces an isomorphism between $\B(\fH{X}_n)$ and $\B(\fH{X}_n)^\prime$. For the double Coxeter diagram of type $\fH{B}_n/\fH{C}_n$, $n\geq 3$, such a diagram automorphism does not exist since the two affine nodes have different degree. In this case, the involution of $B(\fH{B}_n/\fH{C}_n)$ that inverts all the generators sends $\C^{-1}$ to $\C^\prime$ and therefore induces an isomorphism between $\B(\fH{B}_n/\fH{C}_n)$ and $\B(\fH{B}_n/\fH{C}_n)^\prime$. In conclusion, the two different labellings of the affine nodes produce isomorphic groups. 

This fact mirrors the isomorphism between the double affine Artin groups corresponding to the affine root systems of type $D_{n+1}^{(2)}$ and $A_{2n-1}^{(2)}$, $n\geq 3$.  Just as before, despite the fact that we introduce some redundancy, for the benefit on notational uniformity, we will consider both labellings of the affine nodes for the double Coxeter diagram of type $\fH{B}_n/\fH{C}_n$, $n\geq 3$, which we denote as in Figure \ref{fig: BClabels}.

\begin{figure}[ht]
\caption{Affine node labels for $\fH{B}_n/\fH{C}_n$}\label{fig: BClabels}
\begin{center}
  \begin{tikzpicture}[scale=.4]
    \draw (-6,0) node[anchor=east]  {$\fH{B}_n,$};
    \draw (-3,0) node[anchor=east]  {$n\geq 3$};
    \draw[thick] (45: 14.15 mm) circle [radius=3mm];
    \draw[thick] (-45: 14.15 mm) circle [radius=3mm];
    \draw (-0.25,0) node[anchor=east]  {$\varTheta_0$};
    \draw (1,-2.75) node[anchor=south]  {$\varPhi_0$};
    \draw[thick, fill] (0,0) circle [radius=.8mm];
    \draw[thick, fill] (-45: 14.15 mm) circle [radius=.8mm];
    \foreach \x in {0,...,5}
    \draw[xshift=\x cm,thick] (\x cm,0) circle [radius=3mm];
         \foreach \x in {6,7}
    \draw[white,xshift=\x cm] (\x cm,0) circle [radius=3mm];
    \foreach \y in {1.15,3.15}
    \draw[xshift=\y cm,thick] (\y cm,0) -- +(1.4 cm,0);
    \draw[xshift=2 cm,thick] (135: 3mm) -- +(135: .815 cm);
    \draw[xshift=2 cm,thick] (-135: 3mm) -- +(-135: .815 cm);
    \draw[thick] (65: 3mm) -- +(45: .85 cm);
    \draw[thick] (25: 3mm) -- +(45: .85 cm);
    \draw[thick] (-25: 3mm) -- +(-45: .85 cm);
    \draw[thick] (-65: 3mm) -- +(-45: .85 cm);
    \draw[dotted,thick] (4.3 cm,0) -- +(1.4 cm,0);
    \draw[thick] (8.3 cm, .1 cm) -- +(1.4 cm,0);
    \draw[thick] (8.3 cm, -.1 cm) -- +(1.4 cm,0);
\end{tikzpicture}
\end{center}

\begin{center}
  \begin{tikzpicture}[scale=.4]
    \draw (-6,0) node[anchor=east]  {$\fH{C}_n,$};
    \draw (-3,0) node[anchor=east]  {$n\geq 3$};
    \draw[thick] (45: 14.15 mm) circle [radius=3mm];
    \draw[thick] (-45: 14.15 mm) circle [radius=3mm];
    \draw[thick, fill] (0,0) circle [radius=.8mm];
    \draw[thick, fill] (-45: 14.15 mm) circle [radius=.8mm];
        \draw (-0.25,0) node[anchor=east]  {$\varPhi_0$};
    \draw (1,-2.75) node[anchor=south]  {$\varTheta_0$};
    \foreach \x in {0,...,5}
    \draw[xshift=\x cm,thick] (\x cm,0) circle [radius=3mm];
         \foreach \x in {6,7}
    \draw[white,xshift=\x cm] (\x cm,0) circle [radius=3mm];
    \foreach \y in {1.15,3.15}
    \draw[xshift=\y cm,thick] (\y cm,0) -- +(1.4 cm,0);
    \draw[xshift=2 cm,thick] (135: 3mm) -- +(135: .815 cm);
    \draw[xshift=2 cm,thick] (-135: 3mm) -- +(-135: .815 cm);
    \draw[thick] (65: 3mm) -- +(45: .85 cm);
    \draw[thick] (25: 3mm) -- +(45: .85 cm);
    \draw[thick] (-25: 3mm) -- +(-45: .85 cm);
    \draw[thick] (-65: 3mm) -- +(-45: .85 cm);
    \draw[dotted,thick] (4.3 cm,0) -- +(1.4 cm,0);
    \draw[thick] (8.3 cm, .1 cm) -- +(1.4 cm,0);
    \draw[thick] (8.3 cm, -.1 cm) -- +(1.4 cm,0);
\end{tikzpicture}
\end{center}
\end{figure}
\end{Rem}


\subsection{}\label{sec: def-dddot*} We will now discuss the double affine Coxeter diagrams of type $\DDot{C}_n$ which share features with both the diagrams of type $\fH{C}_n$ (because there are only two affine nodes) and those type $\DDDot{C}_n$ (because the affine nodes attach in the same way to the finite Coxeter diagram). This will reflect accordingly in the definition of $\B(\DDot{C}_n)$.

\begin{Def}\label{def: twodot} The group $\B(\DDot{C}_n)$ is defined as the quotient of $B(\DDot{C}_n)$ by the normal subgroup generated by the following relations:
\begin{enumerate}[leftmargin=40pt, label={\alph*)}]
\item The element
\begin{equation}\label{centralrelbis}
\C=(\Tiii\varTheta\Ti)^2
\end{equation} 
is central.
\item If $\ell_0=2$,  we have
\begin{equation}\label{ellbraidbis}
\Ti\T_{i_\th}^{-1}\Tiii\T_{i_\th}=\T_{i_\th}^{-1}\Tiii\T_{i_\th} \Ti.
\end{equation}
\end{enumerate}
We also define the group $\B(\DDot{C}_n)^c$ as the group generated by $\B(\DDot{C}_n)$ and a central element $\C^{1/2}$ such that $(\C^{1/2})^2=\C$.
\end{Def}

\begin{Rem} The relation \eqref{ellbraidbis} is the same as the relation \eqref{ellbraid}  for the diagrams of type $\DDDot{C}_n$. We can see that  the relation \eqref{centralrelbis} is also the same as the relation \eqref{centralrel} for the the diagrams of type $\fH{X}_n$. Indeed, taking into account \eqref{eq: twodotcollapsing}, we can write
\begin{equation}\label{centralrel-2bis}
(\Tiii\varTheta\Ti)^2=\Tiii\varPhi\Ti\varPsi\Tiii\varTheta\Ti,
\end{equation}
which is precisely the relevant element that appears in \eqref{centralrel}.
\end{Rem}

\begin{Rem}
As in the case of diagrams of type $\fH{X}_n$, a priori, Definition \ref{def: twodot} associates two possible groups to a diagram of type $\DDot{C}_n$, depending of the labeling of the affine nodes by $\Ti$ and $\Tiii$. The discussion in Remark \ref{rem: labeling} carries over to this case to show that the two different labelings of affine nodes produce isomorphic groups.
\end{Rem}

Proposition \ref{prop: Cequivpres} allows us to directly relate the groups $\B(\DDDot{C}_n)$, $\B(\DDot{C}_n)$, and $\B(\DDot{C}_n)^c$.

\begin{Prop} \label{comparisonbis}
Let $n\geq 1$.
\begin{enumerate}[label={\roman*)}]
\item The map that sends $\varTheta_{01}$ to $\Ti$, $\varTheta_{03}$ to $\Tiii$, $\T_i$ to $\T_i$, $1\leq i\leq n$,  and $\C$ to $\C$, extends to a surjective group morphism
 $$\B(\DDDot{C}_n)\to \B(\DDot{C}_n)$$
 whose kernel is the normal subgroup generated by $\C^{-1}\varTheta_{02}^2$.
\item The map that sends $\varTheta_{01}$ to $\Ti$, $\varTheta_{03}$ to $\Tiii$, $\T_i$ to $\T_i$, $1\leq i\leq n$,  and $\C$ to $\C^{1/2}$, extends to a surjective group morphism
 $$\B(\DDDot{C}_n)\to \B(\DDot{C}_n)^c$$
 whose kernel is the normal subgroup generated by $\varTheta_{02}^2$.
 \end{enumerate}
\end{Prop}
\begin{proof}
Straightforward from Proposition \ref{prop: Cequivpres} and Definition \ref{def: twodot}.
\end{proof}


\subsection{}\label{sec: labeling}  We will show that the groups defined in \S \ref{sec: def-dddot} and \S\ref{sec: def-ddot} are (up to isomorphism) double affine Artin groups. The precise correspondence is the one specified in \eqref{DCvsDA} with the caveat that the correspondence  pairs $\fH{B}_n$  and  $D_{n+1}^{(2)}$, $n\geq 2$, and $\fH{C}_n$ and  $A_{2n-1}^{(2)}$,  $n\geq 2$. We will use the symbol $\iota$ to refer to this correspondence (in either direction). Remark that through this correspondence the Coxeter diagrams consisting of finite nodes coincide and therefore can be labeled in the same fashion so that the two Coxeter braid groups can be readily identified. Furthermore, we arrange that the node labelled $\varTheta_0, \varTheta_{01}, \varTheta_{02}$, or $\varTheta_{03}$ in a double affine Coxeter diagram is precisely the affine node in the corresponding affine Dynkin diagram. In this fashion, the elements denoted by  $\varTheta$, $\varPhi$, $\T_{i_\th}$, and $\T_{i_\ph}$ correspond respectively to the elements denoted by  $\Theta$, $\Phi$, $T_{i_\th}$, and $T_{i_\ph}$. We note that the labelling of the affine nodes in Figure \ref{fig: BClabels} respects this convention. Henceforth, we adopt the labelling described above.
 
\begin{Rem}
If $X_n$ is double-laced then the central element $\C$ can be written as
\begin{equation}\label{centralrel-2}
\C=(\Tiii\varThetap\varPhip\Ti)^2.
\end{equation}
Indeed, 
\begin{align*}
\Tiii\varPhi\Ti\varPsi\Tiii\varTheta\Ti&= \Tiii\varPhi\Ti\varThetap^{-1}\varPhip^{-1}\Tiii\varTheta\Ti && \text{by Lemma \ref{lemma-explicit-2}iv)}\\
&= \Tiii\varPhi\varThetap^{-1}\Ti\Tiii\varPhip^{-1}\varTheta\Ti &&  \text{by Lemma \ref{B2-rels}i),iii)}\\
&= \Tiii\varThetap\varPhip\Ti \Tiii\varThetap\varPhip\Ti &&  \text{by Lemma \ref{lemma-explicit-2}i),ii)}
\end{align*}
\end{Rem}
\begin{Rem}
For $\fH{G}_2$ the central element $\C$ can be written as
\begin{equation}\label{centralrel-3}
\C=(\Tiii\T_{i_\ph}\T_{i_\th}\T_{i_\ph}\T_{i_\th}\Ti)^2.
\end{equation}
Indeed, 
\begin{align*}
\Tiii\varPhi\Ti\varPsi\Tiii\varTheta\Ti&= \Tiii\varPhi\Ti\T_{i_\ph}^{-1}\T_{i_\th}^{-1}\Tiii\varTheta\Ti && \text{by Lemma \ref{lemma-wy-artin}vii)}\\
&= \Tiii\varPhi \T_{i_\ph}^{-1} \Ti\Tiii  \T_{i_\th}^{-1} \varTheta\Ti && \text{by the braid relations}\\
&= \Tiii(\T_{i_\ph}\T_{i_\th})^2\Ti\Tiii(\T_{i_\ph}\T_{i_\th})^2\Ti&& 
\end{align*}
\end{Rem}

\begin{Prop}\label{superfluous}
In Definition \ref{def: twisted}, the braid relation between $\varPhi_0$ and $\varTheta_0$ is superfluous.
\end{Prop}
\begin{proof} If $X_n$ is double-laced, then
\begin{align*}
\Tiii \Ti \Tiii \Ti &=\C \Tiii \cdot  \C^{-1}\Ti \Tiii  \cdot \Ti && \text{by \eqref{centralrel}}\\
&= \C \Tiii \varPhip^{-1} \varThetap^{-1} \Tiii^{-1} \Ti^{-1} \varPhip^{-1} \varThetap^{-1} \Ti && \text{by \eqref{centralrel-2}}\\
&= \C \Tiii \varPhip^{-1} \varThetap^{-1} \Tiii^{-1}\cdot \varThetap^{-1}\varThetap \cdot \varPhip \varPhip^{-1}\cdot  \Ti^{-1} \varPhip^{-1} \varThetap^{-1} \Ti &&\\
&= \C \varPhip^{-1} \varThetap^{-1} \Tiii^{-1} \varThetap^{-1}\cdot  \Tiii \varThetap \varPhip \Ti \cdot  \varPhip^{-1}  \Ti^{-1} \varPhip^{-1} \varThetap^{-1} && \text{by Lemma \ref{B2-rels}}\\
&= \C \varPhip^{-1} \varThetap^{-1} \Tiii^{-1} \varThetap^{-1}\cdot  \C \Ti^{-1} \varPhip^{-1}\varThetap^{-1}\Tiii^{-1} \cdot  \varPhip^{-1}  \Ti^{-1} \varPhip^{-1} \varThetap^{-1} && \text{by \eqref{centralrel-2}}\\
&= \C \varPhip^{-1} \varThetap^{-1} \Tiii^{-1} \Ti^{-1} \varThetap^{-1} \varPhip^{-1}  \cdot  \C  \varThetap^{-1}\varPhip^{-1}  \Tiii^{-1}     \Ti^{-1} \varPhip^{-1} \varThetap^{-1} && \text{by Lemma \ref{B2-rels}i)iii)}\\
&= \Ti\Tiii \Ti\Tiii &&  \text{by \eqref{centralrel-2}}
\end{align*}
Therefore, $\Ti$ and $\Tiii$ satisfy the $2$-braid relation.

For $\fH{G}_2$ we have
\begin{align*}
\T_{i_\ph} \T_{i_\th} \Ti \Tiii \varPhi \cdot \T_{i_\th} \cdot \varPhi^{-1} \Tiii^{-1}\Ti^{-1} \T_{i_\ph}^{-1} \T_{i_\th}^{-1}&= \T_{i_\ph}\T_{i_\th} \Ti \cdot \T_{i_\th} \cdot \Ti^{-1}  \T_{i_\th}^{-1}\T_{i_\ph}^{-1}\\
&= \T_{i_\ph} \Ti \T_{i_\ph}^{-1}\\
&=\Ti.
\end{align*}
and
\begin{align*}
\T_{i_\ph} \T_{i_\th} \Ti \Tiii \varPhi \cdot \T_{i_\ph} \cdot \varPhi^{-1} \Tiii^{-1}\Ti^{-1} \T_{i_\ph}^{-1} \T_{i_\th}^{-1}&=  \C\T_{i_\th}^{-1} \T_{i_\ph}^{-1}\Tiii^{-1}\Ti^{-1} \cdot \T_{i_\ph} \cdot \Ti \Tiii \T_{i_\ph}\T_{i_\th}\C^{-1}\\
&= \T_{i_\th}^{-1} \T_{i_\ph}^{-1}\Tiii^{-1} \cdot \T_{i_\ph} \cdot \Tiii \T_{i_\ph}\T_{i_\th}\\
&= \T_{i_\th}^{-1}  \Tiii \T_{i_\th}\\
&=\Tiii.
\end{align*}
Since $\T_{i_\th}$ and $\T_{i_\ph}$ satisfy the $3$-braid relation we obtain that $\Ti$ and $\Tiii$ satisfy the $3$-braid relation.
\end{proof}
\subsection{}  One immediate consequence of the above discussion is that it shows that $\B(\DDDot{A}_1)$,  $\B(\DDot{A}_1)$,  and $\B(\fH{B}_2)$ embed in $\B(\DDDot{X}_n)$, $\B(\DDot{C}_n)$, and $\B(\fH{X}_n)$ (for $X_n\neq G_2$), respectively. The appropriate statement for $\fH{G}_2$ is, of course, trivial.

\begin{Prop}\label{thm: rank 1 embedding}
There exists canonical embeddings that preserve the affine generators of $\B(\DDDot{A}_1)$  into $\B(\DDDot{X}_n)$,  of   $\B(\DDot{A}_1)$ into $\B(\DDot{C}_n)$, and of  $\B(\fH{B}_2)$ into $\B(\fH{X}_n)$ for $X_n$ double laced.
\end{Prop}
\begin{proof}
For $\DDDot{X}_n$ the generator associated to the finite node of $\DDDot{A}_1$ is mapped to $\varTheta$. The same holds for $\DDot{C}_n$ and  $\DDot{A}_1$. For $\fH{X}_n$ the generators associated to the finite nodes of $\fH{B}_2$ are mapped to $\varThetap$ and $\varPhip$.  Our claim is straightforward from Definition \ref{def: untwisted},  Definition \ref{def: twodot} and \eqref{centralrel-2bis}, and Definition \ref{def: twisted}, Proposition \ref{B2-rels}, and \eqref{centralrel-2}.
\end{proof}

\subsection{}  We are now in position to show that the groups $\B(\DDDot{X}_n)$ and $\Atilde(X_n^{{(1)}})$ are isomorphic. Let us define the candidates for isomorphism between the two groups. 

Let
$$
\phi: \B(\DDDot{X}_n) \to \Atilde(X_n^{{(1)}})
$$
be defined as the extension to a group morphism of the map
$$
 \phi(\varTheta_{01})=T_0,\quad \phi(\varTheta_{02})=T_0^{-1}X_{\a_0^\vee},\quad \phi(\varTheta_{03})=X_{\th^\vee}\Theta^{-1}, \quad \phi(\T_i)=T_i,\quad 1\leq i\leq n.
$$
\begin{Prop}\label{lemma1}
The map $\phi: \B(\DDDot{X}_n) \to \Atilde(X_n^{{(1)}})$ is well defined.
\end{Prop}
\begin{proof}
 From Remark \ref{X-artin} we know that $X_{\th^\vee}\Theta^{-1}$ and $T_i$, $1\leq i\leq n$ satisfy the desired braid relations. Let us verify that $T_0^{-1}X_{\a_0^\vee}$ and $T_i$, $1\leq i\leq n$ satisfy the $a_{0i}a_{i0}$-braid relations.
 
The claim is obvious if the $0$ node and the $i$ node are not connected in the Dynkin diagram. Therefore, we only need to verify the $\ell_0$-braid relation for $T_0^{-1}X_{\a_0^\vee}$ and $T_{i_\th}$. If $A=A_n^{(1)}$, $n \geq 2$, then the argument we present works also for the second simple root whose node is connected to the $0$ node. Remark that if $\ell_0=4$ there is nothing to prove and the same is true for $\ell_0=2$ by Lemma \ref{reducel=2bis}. 

If  $\ell_0=1$ then,
\begin{align*}
T_0^{-1}X_{\a_0^\vee}T_{i_\th} T_0^{-1}X_{\a_0^\vee}&= T_0^{-1}T_{i_\th}^{-1}X_{\a_0^\vee+\a_{i_\th}^\vee}T_0^{-1}X_{\a_0} & & \text{by \eqref{dc2}}\\
&= T_0^{-1}T_{i_\th}^{-1}T_0 X_{\a_{i_\th}^\vee+\a_0^\vee} & & \text{by \eqref{dc2}}\\
&= T_{i_\th} T_0^{-1}T_{i_\th}^{-1}X_{\a_{i_\th}^\vee+\a_0^\vee} & & \text{by the braid relation for $T_0$ and $T_{i_\th}$}\\
&= T_{i_\th} T_0^{-1}X_{\a_0^\vee}T_{i_\th}, & & \text{by \eqref{dc2}}
\end{align*}
which is the $1$-braid relation for $T_0^{-1}X_{\a_0^\vee}$ and $T_{i_\th}$.

 
 Also,
\begin{align*}
\phi(\C)&=\phi(\varTheta_{01}\varTheta_{02}\varTheta_{03}\varTheta)\\
&=T_0T_0^{-1}X_{\a_0^\vee}X_{\th^\vee}\Phi^{-1}\Phi\\
&=X_{\d}
\end{align*}
which is central in the double affine Artin group. The only thing which needs explanation is the fact that the image of relation \eqref{ellbraid} holds if $\ell_0=2$. In fact, as it follows from Lemma \ref{reducel=2}, we need to do this only for one relation, say
\begin{equation}
\varTheta_{01}\T_{i_\th}^{-1}\varTheta_{02}\T_{i_\th}=\T_{i_\th}^{-1}\varTheta_{02}\T_{i_\th}\varTheta_{01}.
\end{equation}

Indeed,
\begin{align*}
\phi(\varTheta_{01}\T_{i_\th}^{-1}\varTheta_{02}\T_{i_\th})&= T_{0}T_{i_\th}^{-1}T_{0}^{-1}X_{\a_0^\vee}T_{i_\th} & & \\
&= T_{0}T_{i_\th}^{-1}T_{0}^{-1}T_{i_\th}^{-1}X_{\a_0^\vee+\a_{i_\th}^\vee} & & \text{by \eqref{dc2v2} and \eqref{dc3v2}}\\
&= T_{i_\th}^{-1}T_{0}^{-1}T_{i_\th}^{-1}T_{0}X_{\a_0^\vee+\a_{i_\th}^\vee} & & \text{by the braid relation between $T_0$ and $T_{i_\th}$} \\
&= T_{i_\th}^{-1}T_{0}^{-1}T_{i_\th}^{-1}X_{\a_0+\a_{i_\th}}T_{0} & &\text{by \eqref{eq1v2}} \\
&= T_{i_\th}^{-1}T_{0}^{-1}X_{\a_0}T_{i_\th} T_{0} & & \text{by \eqref{dc2v2} and \eqref{dc3v2}} \\
&= \phi(\T_{i_\th}^{-1}\varTheta_{02}\T_{i_\th}\varTheta_{01}). & &
\end{align*}
The proof is completed.
\end{proof}
Let 
$$
\psi:  \Atilde(X_n^{{(1)}})\to \B(\DDDot{X}_n)
$$
be defined as the extension to a group morphism of the map
\begin{subequations}
\begin{alignat}{3}
&\psi(T_{0})=\varTheta_{01}, &\quad &\psi(T_i)=\T_i, \quad1\leq i\leq n\\
&\psi(X_{\th^\vee})=\varTheta_{03}\varTheta, &&\phi(X_\d)=\C.
\end{alignat}
\end{subequations}

\begin{Prop}\label{lemma2}
The map $\psi:  \Atilde(X_n^{{(1)}})\to \B(\DDDot{X}_n)$ is well defined.
\end{Prop}
\begin{proof}
As it clear from Remark \ref{X-artin}, the elements for which we defined $\psi$  generate $\Atilde(X_n^{{(1)}})$. The only fact that requires justification is that the image of the relations \eqref{eq5-1} and \eqref{eq5-2} holds in $\B(\DDDot{X}_n)$.

Let us consider first the case $\ell_0=1$. We have to prove that
$$
\varTheta_{01}\psi(X_{\a_{i_\th}^\vee})\varTheta_{01}=\psi(X_{\a_0^\vee+a_{i_\th}^\vee}).
$$
Since $T_{i_\th} X_{-\th^\vee}T_{i_\th}=X_{\a_{i_\th}^\vee-\th^\vee}$ we know that
$$
X_{\a_{i_\th}^\vee}= T_{i_\th} X_{-\th^\vee}T_{i_\th} X_{\th^\vee} .
$$
Therefore we want that
$$
\varTheta_{01} \T_{i_\th} \varTheta^{-1} \varTheta_{03}^{-1} \T_{i_\th}\varTheta_{03}\varTheta\varTheta_{01}= \C \T_{i_\th} \varTheta^{-1} \varTheta_{03}^{-1}\T_{i_\th}
$$
or equivalently
$$
\varTheta\T_{i_\th}^{-1} \varTheta_{01} \T_{i_\th} \varTheta^{-1} \varTheta_{03}^{-1}\T_{i_\th} \varTheta_{03}\varTheta\varTheta_{01} \T_{i_\th}^{-1}\varTheta_{03}=\C.
$$
This immediately follows by replacing the formula \eqref{magic1} for $\varTheta_{02}$ into the equation \eqref{centralrel-1}.

In the case $\ell_0=2$ we have to prove that
$$
\varTheta_{01} \psi(X_{\a_{i_\th}^\vee-\th^\vee})= \psi(X_{\a_{i_\th}^\vee-\th^\vee})\varTheta_{01}.
$$
As before,
$$
X_{\a_{i_\th}^\vee-\th^\vee}= T_{i_\th} X_{-\th^\vee}T_{i_\th} ,
$$
hence our statement is proved as follows
\begin{align*}
\varTheta_{01} \psi(X_{\a_{i_\th}^\vee-\th^\vee})&=\varTheta_{01} \T_{i_\th} \varTheta^{-1}\varTheta_{03}^{-1} \T_{i_\th} & & \\
&=\C^{-1}\varTheta_{01} \T_{i_\th}\varTheta_{01}\varTheta_{02}\T_{i_\th} & & \text{by \eqref{centralrel-1}} \\
&=\C^{-1}\varTheta_{01} \T_{i_\th}\varTheta_{01}\T_{i_\th}\T_{i_\th}^{-1}\varTheta_{02}\T_{i_\th} & & \\
&=\C^{-1}\T_{i_\th}\varTheta_{01}\T_{i_\th}\varTheta_{01}\T_{i_\th}^{-1}\varTheta_{02}\T_{i_\th} & & \text{by the braid relation between $\varTheta_{01}$ and $\T_{i_\th}$}\\
&=\C^{-1}\T_{i_\th}\varTheta_{01}\T_{i_\th}\T_{i_\th}^{-1}\varTheta_{02}\T_{i_\th}\varTheta_{01} & & \text{by \eqref{ellbraid}}\\
&=\C^{-1}\T_{i_\th}\varTheta_{01}\varTheta_{02}\T_{i_\th}\varTheta_{01} & & \\
&=\T_{i_\th}\varTheta^{-1}\varTheta_{03}^{-1} \T_{i_\th}\varTheta_{01} & & \text{by \eqref{centralrel-1}}\\
&=\psi(X_{\a_{i_\th}^\vee-\th^\vee})\varTheta_{01}. & &
\end{align*}
The proof  is now complete.
\end{proof}

\begin{Thm}\label{main}
The groups $\Atilde(X_n^{{(1)}})$ and $\B(\DDDot{X}_n)$ are isomorphic.
\end{Thm}
\begin{proof}
The morphisms constructed in Proposition \ref{lemma1} and Proposition \ref{lemma2} are inverse for each other, as it can be easily checked  on generators.
\end{proof}


\subsection{}
From Proposition \ref{comparison} we also obtain a similar description for $\Atilde(A_{2n}^{(2)})$ and $\Atilde^c(A_{2n}^{(2)})$.

\begin{Thm}\label{mainA2n2}
The double affine Artin group $\Atilde(A_{2n}^{(2)})$ is isomorphic to  $\B(\DDot{C}_n)$ and the group  $\Atilde^c(A_{2n}^{(2)})$ is isomorphic to $\B(\DDot{C}_n)^c$.
\end{Thm}
\begin{proof}
Straightforward from Proposition \ref{comparison}, Proposition \ref{comparisonbis}, and Theorem \ref{main}.
\end{proof}
For completeness, let us specify below the inverse isomorphisms between $\Atilde^c(A_{2n}^{(2)})$ and $\B(\DDot{C}_n)^c$. The relevant isomorphisms between $\Atilde(A_{2n}^{(2)})$ and $\B(\DDot{C}_n)$ is simply obtained by restriction.
Let
$$
\phi: \B(\DDot{C}_n)^c \to \Atilde^c(A_{2n}^{{(2)}})
$$
be defined as the extension to a group morphism of the map
$$
 \phi(\Ti)=T_0,\quad \phi(\Tiii)=X_{\th^\vee}\Theta^{-1}, \quad \phi(\T_i)=T_i,\quad 1\leq i\leq n, \quad \phi(\C^{1/2})=X_{\frac{1}{2}\d}.
$$
Its inverse is
$$
\psi:  \Atilde^c(A_{2n}^{{(2)}})\to \B(\DDot{C}_n)^c
$$
defined as the extension to a group morphism of the map
\begin{subequations}
\begin{alignat}{3}
&\psi(T_{0})=\Ti, &\quad &\psi(T_i)=\T_i, \quad1\leq i\leq n\\
&\psi(X_{\th^\vee})=\Tiii\varTheta, &&\phi(X_{\frac{1}{2}\d})=\C^{1/2}.
\end{alignat}
\end{subequations}


\subsection{} We will show that the groups $\B(\fH{X}_n)$ and $\Atilde(\fH{X}_n^\iota)$ are also isomorphic. Let us define the candidates for isomorphism between the two groups.

Let
$$
\phi: \B(\fH{X}_n)\to \Atilde(\fH{X}_n^\iota)
$$
be defined as the extension to a group morphism of the map
\begin{equation}
\phi(\Ti)=T_0, \quad \phi(\Tiii)=X_{\ph^\vee}\Phi^{-1}, \quad \phi(\T_i)=T_i, \quad1\leq i\leq n.
\end{equation}
\begin{Prop}\label{prop1}
The map $\phi: \B(\fH{X}_n)\to \Atilde(\fH{X}_n^\iota)$ is well defined.
\end{Prop}
\begin{proof}  As noted in Remark \ref{X-artin}, the element $\phi(\Tiii)$ satisfies the predicted braid relations with $T_i$, $1\leq i\leq n$. Furthermore, from Proposition \ref{superfluous} we know that we can exclude from the definition $\B(\fH{X}_n)$ the braid relations between $\Ti$ and $\Tiii$. Therefore, we only need to check that the image of relation \eqref{centralrel} holds in $\Atilde(\fH{X}_n^\iota)$. More precisely, we need to verify that
\begin{equation}\label{eq6}
X_\delta=X_{\ph^\vee}\Phi^{-1}\Phi T_{0} \Psi X_{\ph^\vee}\Phi^{-1}\Theta T_{0}
\end{equation} 
or, equivalently, that
\begin{equation}\label{eq7}
X_\delta=T_{0}X_{\ph^\vee}T_{0}\Psi X_{\ph^\vee}\Phi^{-1}\Theta.
\end{equation} 
This equality follows from \eqref{eq5} and Proposition \ref{lemma-affine-commutation}i) applied to the affine Coxeter group generated by $X_{\ph^\vee}\Phi^{-1}$ and $T_i$, $1\leq i\leq n$.
\end{proof} 

Let
$$
\psi: \Atilde(\fH{X}_n^\iota)\to \B(\fH{X}_n)
$$
be defined as the extension to a group morphism of the map
\begin{subequations}
\begin{alignat}{3}
&\psi(T_{0})=\Ti, &\quad &\psi(T_i)=\T_i, \quad1\leq i\leq n\\
&\psi(X_{\ph^\vee})=\Tiii\varPhi, &&\phi(X_\d)=\C.
\end{alignat}
\end{subequations}
\begin{Prop}\label{prop2}
The map $\psi: \Atilde(\fH{X}_n^\iota)\to \B(\fH{X}_n)$ is well defined.
\end{Prop}
\begin{proof} The only fact that requires justification is that the image of relation \eqref{eq5} holds in $\B(\fH{X}_n)$. Using Proposition \ref{lemma-affine-commutation}i) applied to the affine Coxeter group generated by $X_{\ph^\vee}\Phi^{-1}$ and $T_i$, $1\leq i\leq n$,  we obtain that
$$
\psi(X_{\php^\vee})=\varPsi\Tiii\varTheta.
$$
Therefore, we need to verify that
\begin{equation}
\Ti\Tiii\varPhi\Ti=\C \varTheta^{-1}\Tiii^{-1}\varPsi^{-1}
\end{equation}
or, equivalently, that
\begin{equation}
\Tiii\varPhi\Ti\varPsi\Tiii \varTheta \Ti=\C,
\end{equation}
which is precisely \eqref{centralrel}.
\end{proof}
\begin{Thm}\label{thm1}
The groups $\Atilde(\fH{X}_n^\iota)$ and $\B(\fH{X}_n)$ are isomorphic.
\end{Thm}
\begin{proof}
The group morphisms constructed in Proposition \ref{prop1} and Proposition \ref{prop2} are inverse to each other, as it can be easily checked on generators.
\end{proof}
\subsection{} We record here Coxeter group versions of the results proved in this section.
\begin{Def}
Let $\Cox(\DDDot{X}_n)$, $\Cox(\DDot{C}_n)$, and $\Cox(\fH{X}_n)$ denote the quotient of $\B(\DDDot{X}_n)$,  $\B(\DDot{C}_n)$, and, respectively,  $\B(\fH{X}_n)$, by the normal subgroup generated by the squares of the generators.
\end{Def}
\begin{Rem}\label{rem: Cox}
Alternatively, $\Cox(\DDDot{X}_n)$, $\Cox(\DDot{C}_n)$, and $\Cox(\fH{X}_n)$ can be defined as the quotient of $C(\DDDot{X}_n)$, $C(\DDot{C}_n)$, and, respectively,  $C(\fH{X}_n)$ by the relations \eqref{centralrel-1}-\eqref{ellbraid}, \eqref{centralrelbis}-\eqref{ellbraidbis}, and, respectively, \eqref{centralrel}.
\end{Rem}
\begin{Cor}\label{cor: Cox}
The group $\Cox(\DDDot{X}_n)$ is isomorphic to the double affine Weyl group associated to $X_n^{(1)}$ and the $\Cox(\fH{X}_n)$ is isomorphic to the double affine Weyl group associated to $\fH{X}_n^\iota$. Furthermore, the double affine Weyl group associated to $A_{2n}^{(2)}$ is isomorphic to $\Cox(\DDot{C}_n)$, and  $\Cox(\DDDot{C}_n)$ is isomorphic to the trivial central extension of the double affine Weyl group associated to $A_{2n}^{(2)}$ by the group generated by an element $\tau_{\frac{1}{2}\d}$ such that $\tau_{\frac{1}{2}\d}^2=\tau_{\d}$.
\end{Cor}
\begin{proof}
Straightforward from Remark \ref{rem: artin-to-weyl}, Theorem \ref{main}, Theorem \ref{thm1}, and Theorem \ref{mainA2n2}. Alternatively, it can be verified that the proofs of Theorem \ref{main}, Theorem \ref{thm1}, and Theorem \ref{mainA2n2} carry over at the level of Weyl double affine groups.
\end{proof}

\section{Automorphisms}\label{sec: auto}

\subsection{}
As an application of the Coxeter-type presentation of the double affine Artin group we show that the rank two Coxeter braid groups (see \S\ref{sec: rank 2 braid}) faithfully act by automorphisms on the appropriate double affine Artin group and they induce actions of the appropriate congruence groups by outer automorphisms as in \cite{ISEALA}. For the double affine Artin groups associated to untwisted affine Dynkin diagrams this action was constructed by Cherednik \cite{CheMac}*{Theorem 4.3} and a construction of the action based on the Coxeter-type presentation as well as the faithfulness statement can be found in \cite{ISTri}*{\S4.2}. For  the double affine Artin groups associated to twisted affine Dynkin diagrams the results are new. 


\subsection{}  Let us describe in detail the braid groups and congruence groups that will turn out to act as automorphisms. For an integer $N\geq 2$, let $\Gamma(N)$ denote the level $N$ principal congruence subgroup of $\SL(2,\Z)$, which is defined as the kernel of the canonical morphism $\SL(2,\Z)\to \SL(2,\Z/N\Z)$. Explicitly, we have
$$
\Gamma(N)=\left\{ \begin{bmatrix}a&b\\ c&d \end{bmatrix}\in \SL(2,\Z)~\Bigg\vert~  \begin{bmatrix}a&b\\ c&d \end{bmatrix}= \begin{bmatrix}1& 0 \\ 0&1 \end{bmatrix} (\text{mod}~N)\right\}.
$$

Denote $\Gamma_1(1)=\rm{SL}(2,\Z)$  and, for $r=2,3$,
$$
\Gamma_1(r)=\left\{ \begin{bmatrix}a&b\\ c&d \end{bmatrix}\in \SL(2,\Z)~\Bigg\vert~  \begin{bmatrix}a&b\\ c&d \end{bmatrix}= \begin{bmatrix}1& * \\ 0&1 \end{bmatrix} (\text{mod}~r)\right\}.
$$
The index of $\Gamma_1(2)$ and $\Gamma_1(3)$ inside $\Gamma_1(1)$ is 3 and 8, respectively.
Another congruence group that will appear in our context is
$$
\Gamma_1(2)^\prime=\left\{ \begin{bmatrix}a&b\\ c&d \end{bmatrix}\in \SL(2,\Z)~\Bigg\vert~  a+d\equiv b+c \equiv 0~(\text{mod}~2)\right\}.
$$
Denote by $I_2$ the two-by-two identity matrix and let
$$
u_{12}=\begin{bmatrix} 1& -1\\ 0&1\end{bmatrix},\quad u_{21}=\begin{bmatrix} 1& 0\\ 1&1\end{bmatrix}, \quad\text{and}\quad e(r)=\begin{bmatrix} 0& r^{-\frac{1}{2}}\\  r^{\frac{1}{2}}&0\end{bmatrix}\in{\GL}(2,\Re), ~1\leq r\leq 3.
$$
As it can be directly verified, 
\begin{equation}
\begin{aligned}
&u_{12}u_{21}u_{12}=u_{21}u_{12}u_{21}=\begin{bmatrix} 0& -1\\1& 0\end{bmatrix},& & (u_{12}u_{21})^3=-I_2, \\ 
& (u_{12}u_{21}^2)^2=-I_2, & & (u_{12}u_{21}^3)^3=I_2,&\\
&e(r)u_{12}e(r)=u_{21}^{-r},& &e(r)^2=I_2, ~ 1\leq r\leq 3.&
\end{aligned}
\end{equation}
It is well-known that $\<u_{12}, u_{21}\>$ is precisely $\Gamma_1(1)$. Furthermore, for $r=2,3$,
$$
\<u_{12}, u_{21}^r\>\leq \Gamma_1(r),
$$
and the index of $\<u_{12}, u_{21}^r\>$ inside $\<u_{12}, u_{21}\>$ is 3 and, respectively, 8 (a set of coset representatives being, for example, the classes of $I_2$, $u_{21}$, $u_{12}u_{21}$ and, respectively, $I_2$, $u_{21}$, $u_{12}u_{21}$, $u_{12}^2u_{21}$, $u_{21}^2$, $u_{12}u_{21}^2$, $u_{12}^2u_{21}^2$). Therefore, for $r=1,2,3$, $u_{12}$ and $u_{21}^r$ satisfy the $r$-braid relation and
$$
\<u_{12}, u_{21}^r\>= \Gamma_1(r).
$$
As can be directly checked, the groups $\Gamma_1(2)$ and $\Gamma_1(2)^\prime$ are conjugate inside $\Gamma_1(1)$, more precisely
$$
u_{21}\Gamma_1(2)u_{21}^{-1}=\Gamma_1(2)^\prime,
$$
and therefore, $u_{21}u_{12}u_{21}^{-1}$ and $u_{21}^2$ satisfy the $2$-braid relation and
$$
\<u_{21}u_{12}u_{21}^{-1}, u_{21}^2\>= \Gamma_1(2)^\prime.
$$
For $1\leq r\leq 3$, let us also denote by $\Xi_1(r)$ the subgroup of $\GL(2,\Re)$ generated by $e(r)$ and $\Gamma_1(r)$, and by $\Xi_1(2)^\prime$ the subgroup of $\GL(2,\Re)$ generated by $e(1)$ and $\Gamma_1(2)^\prime$. We remark that $\Xi_1(1)=\GL(2,\Z)$ and in general $\Gamma_1(r)$, $\Gamma_1(2)^\prime$ is normal of index two in $\Xi_1(r)$, $\Xi_1(2)^\prime$, respectively.
\subsection{}\label{sec: rank 2 braid} Let us consider the Coxeter braid group of rank two (associated to the Coxeter diagrams of type $A_2$, $B_2$, and $G_2$). In order to emphasize their relationship to the congruence groups discussed above we use the following notation. We denote by $\wGamma(1)$ the braid group generated by two elements $\mfu_1, \mfu_2$ that satisfy the $1$-braid relation (the Coxeter braid group of type $A_2$), and set 
\begin{equation}
\mfc=(\mfu_1\mfu_2)^3.
\end{equation}
 The  center of $\wGamma(1)$ is the subgroup generated by $\mfc$. By Lemma \ref{basic-braid-rels2}i), for $r=2,3$, the elements $\mfu_1,\mfu_2^r$ satisfy the $r$-braid relation. In fact, it is known \cite{FMPri}*{\S3.5.2} that these braid relations are the defining relations for the groups generated by the above pairs of elements. Therefore, the group $\wGamma(2)=\<\mfu_1,\mfu_2^2\>$ is the Coxeter braid group of type $B_2$, and $\wGamma(3)=\<\mfu_1,\mfu_2^3\>$ is the Coxeter braid group of type $G_2$. The group $\wGamma(2)$ can be interpreted as the subgroup of the braid group on three strands that is fixing the third strand. Furthermore,
\begin{equation}
(\mfu_1\mfu_2^2)^2=(\mfu_1\mfu_2)^3=\mfc\quad\text{and}\quad (\mfu_1\mfu_2^3)^3=(\mfu_1\mfu_2)^6=\mfc^2,
\end{equation}
and these elements generate the center of $\wGamma(2)$ and $\wGamma(3)$, respectively. Define also $$\wGamma(2)^\prime=\mfu_{2}\wGamma(2)\mfu_{2}^{-1}=\<\mfu_{2}\mfu_{1}\mfu_{2}^{-1}, \mfu_{2}^2\>.$$ Remark that $\mfc\in \wGamma(2)^\prime$ and this element generates the center of $\wGamma(2)^\prime$. The group $\wGamma(2)^\prime$ can be interpreted as the subgroup of the braid group on three strands that is fixing the middle strand.

There exists a surjective group morphism
$$\pi: \wGamma(1) \to \Gamma_1(1)$$ defined by $\pi(\mfu_1)=u_{12}$, $\pi(\mfu_2)=u_{21}$, that restricts to corresponding surjective morphisms 
$$\pi: \wGamma(r) \to \Gamma_1(r), ~~r=2,3,\quad\text{and}\quad \pi: \wGamma(2)^\prime \to \Gamma(2)^\prime.$$
All these maps are central extensions. However, since the groups $\Gamma_1(r)$, $r=1,2,3$, are not perfect groups (and nor is, for example, $\wGamma(1)$ \cite{GLAlg}*{Theorem 2.1}) they do not have universal central extensions \cite{SteLec}*{\S7(ii),(iv)} as abstract groups. What can be said in turn, is that $\wGamma(r)$ is the preimage of $\Gamma_1(r)\subset \SL(2,\Re)$ inside its universal cover $\widetilde{\SL(2,\Re)}$.

For $1\leq r\leq 3$, we also consider the group $\wXi(r)$ defined as the semi-direct product of $\wGamma(r)$ and the cyclic group of order two generated by an element $\mfv(r)$ such that
$$\mfv(r)\mfu_{1}\mfv(r)=\mfu_{2}^{-r}.$$
Remark that $\mfv(1)$ normalizes $\wGamma(2)^\prime$ inside $\wGamma(1)$ and define $\wXi(2)^\prime$ as the subgroup generated by $\mfv(1)$ and $\wGamma(2)^\prime$. By sending $\mfv(r)$ to $e(r)$, we can extend the group morphism $\pi$ defined above to 
$$\pi: \wXi(r) \to \Xi_1(r), ~~1\leq r\leq 3,\quad\text{and}\quad \pi: \wXi(2)^\prime \to \Xi(2)^\prime,$$
which are also surjective group morphisms.


\subsection{}
We begin by describing a certain anti-involution which itself plays a major role in the theory of double affine Hecke algebras and in Macdonald theory. Its existence was announced in \cite{CheDou-2}*{Theorem 2.2, Theorem 2.4} for the double affine Artin groups associated to untwisted affine Dynkin diagrams and proofs treating all double affine Artin groups later appeared in \cite{IonInv} (topological) and\cite{MacAff}*{\S 3.5-3.7} (algebraic). The argument  for existence of this anti-involution based on the Coxeter presentation of the double affine Artin groups first appeared in \cite{ISTri} for what corresponds here to the groups associated to untwisted Dynkin diagrams. We stress that the existence of this anti-involution is an immediate consequence of the Coxeter type presentation for double affine Artin groups.

\subsection{} The action of the map we will now define is present at several levels. To avoid cumbersome notation we will denote all these maps by $\mfe$ as it will be clear from the context to which one we refer to.  At the level of double affine Coxeter diagrams in Figure \ref{dddot-diagrams}, Figure \ref{ddot-diagrams}, and Figure \ref{fH-diagrams} the involution $\mfe$ acts as identity.  At the level of labeled double affine Coxeter diagrams $\mfe$ still acts as identity except for the diagrams of type $\fH{B}_n$ and $\fH{C}_n$ (see Figure \ref{fig: BClabels}) which are interchanged.

For each pair of diagrams that correspond via $\mfe$ there is an associated involution between their nodes. For labelled diagrams of type $\DDDot{X}_n$ the map 
\begin{equation}
\mfe: \DDDot{X}_n\to \DDDot{X}_n
\end{equation}
is again the identity map except for the nodes labelled by $\varTheta_{01}$ and $\varTheta_{03}$ that are interchanged.

For labelled diagrams of type $\DDot{C}_n$ the map 
\begin{equation}
\mfe: \DDot{C}_n\to \DDot{C}_n
\end{equation}
is again the identity map except for the two affine nodes which are interchanged.

For labeled diagrams of type $\fH{X}_n$, the map 
\begin{equation}
\mfe: \fH{X}_n\to \fH{X}_n^\mfe
\end{equation}
is the unique diagram isomorphism that interchanges the two labelled affine nodes. Note that for the diagrams of type $\fH{B}_2/\fH{C}_2$, $\fH{F}_4$, $\fH{G}_2$, $\mfe$ is forced to act non-trivially on the set of finite nodes. 

Being an isomorphism between Coxeter diagrams, $\mfe$ induces between the corresponding Coxeter braid groups both an isomorphism and an anti-isomorphism, in either case, its inverse being still  the map induced by $\mfe$ in the other direction. In other words, $\mfe$ at the level of Coxeter braid groups has always order two, except for $\mfe: B(\fH{B}_n)\to B(\fH{C}_n)$, $n\geq 3$ whose inverse is 
$\mfe: B(\fH{C}_n)\to B(\fH{B}_n)$. In what follows we will not be making this distinction and we will say that $\mfe$ has order two. Only the induced anti-isomorphisms descend to maps between the corresponding double affine Artin groups and whenever used in this context $\mfe$ always denotes an \emph{anti-involution}. 

\begin{Thm} The map $\mfe$ induces anti-involutions
\begin{align}
&\mfe: \B(\DDDot{X}_n)\to \B(\DDDot{X}_n),&  \mfe:& \B(\fH{X}_n)\to \B(\fH{X}^\mfe_n)&\\
&\mfe: \B(\DDot{C}_n)\to \B(\DDot{C}_n),& \mfe:& \B(\DDot{C}_n)^c\to \B(\DDot{C}_n)^c. &
\end{align}

\end{Thm}
\begin{proof}
The fact that $\mfe$ preserves the relations in Definition \ref{def: untwisted} and Definition \ref{def: twodot} is trivially verified. For $\mfe: \B(\fH{X}_n)\to \B(\fH{X}^\mfe_n)$, we have
$$
\mfe(\varTheta)=\varPhi, \quad \mfe(\varPhi)=\varTheta, \quad \mfe(\varThetap)=\varPhip, \quad \mfe(\varPhip)=\varThetap, \quad \mfe(\varPsi)=\varPsi.
$$
with the help of which the fact that $\mfe$ preserves the relations in Definition \ref{def: twisted} immediately follows.
\end{proof}
\begin{Rem}
The anti-involution $\mfe$ is compatible with the surjective morphisms in Proposition \ref{comparisonbis}. 
\end{Rem}
\begin{Rem}
In the Bernstein presentation of the double affine Artin group, the anti-involution $\mfe$ is precisely the anti-involution in \cite{MacAff}*{(3.5.1)} and its composition  with the map taking inverses is precisely the involution in \cite{IonInv}*{Theorem 2.2}.
\end{Rem}

\subsection{}

In what follows we will be particularly interested in automorphisms of double affine Artin groups that act as identity when restricted to the corresponding finite Artin groups. We denote by  $\Aut(\B(\DDDot{X}_n);X_n)$, $\Aut(\B(\DDot{C}_n);C_n)$, and $\Aut(\B(\fH{X}_n);X_n)$ the group of such automorphisms and by  $\Out(\B(\DDDot{X}_n);X_n)$, $\Out(\B(\DDot{C}_n);C_n)$, and $\Out(\B(\fH{X}_n);X_n)$ the corresponding group of outer automorphisms. 

Let us consider first the case of a Coxeter diagram of type $\DDDot{X}_n$. We define two maps on the set of generators of $\B(\DDDot{X}_n)$ by letting them fix the generators corresponding to the finite nodes and by
 \begin{align*}
\mfa(\varTheta_{01})&=\varTheta_{02}, & \mfa(\varTheta_{02})&=\varTheta_{02}^{-1}\varTheta_{01}\varTheta_{02}, &\mfa(\varTheta_{03})&=\varTheta_{03}, \\ 
\mfb(\varTheta_{01})&=\varTheta_{01}, & \mfb(\varTheta_{02})&=\varTheta_{03}, & \mfb(\varTheta_{03})&=\varTheta_{03}^{-1}\varTheta_{02}\varTheta_{03}. 
\end{align*}
\begin{Thm}\label{thm: auto1}
The above maps extend to elements of $\Aut(\B(\DDDot{X}_n);X_n)$. Furthermore, 
\begin{enumerate}[label={\roman*)}]
\item $\mfa^{-1}=\mfe\mfb\mfe, \quad\mfb^{-1}=\mfe\mfa\mfe,\quad \text{and}\quad \mfa\mfb\mfa=\mfb\mfa\mfb$;
\item $(\mfa\mfb)^3$ acts on the affine generators by conjugation with $\T_{w_\circ}$;
\item If $X_n$ is of type $B_n$, $n\geq 3$, $C_n$, $n\geq 1$, $D_{2n+1}$, $n\geq 2$, $E_7$, $E_8$, $F_4$, $G_2$, then $(\mfa\mfb)^3$ acts on $\B(\DDDot{X}_n)$ by conjugation by $\T_{w_\circ}$;
\item If $X_n$ is of type $A_n$, $n\geq 2$, $D_{2n}$, $n\geq 2$, $E_6$, then $(\mfa\mfb)^6$ acts on $\B(\DDDot{X}_n)$ by conjugation by $\T^2_{w_\circ}$.
\end{enumerate}
\end{Thm}
\begin{proof}
We need to verify that the two maps preserve the relations in Definition \ref{def: untwisted}. We will perform the verifications for $\mfa$, the verifications for $\mfb$ are entirely similar.

The relation \eqref{centralrel-1} is clearly preserved by $\mfa$. Let us check that $\mfa(\varTheta_{02})$ satisfies the braid relations with the generators corresponding to the simple nodes. For a node $i\neq i_\th$ the braid relations follow from the fact that $\varTheta_{01}$ and $\Ti$ commute, and $\varTheta_{02}$ and $\T_i$ commute. For $i_\th$, we have to argue that $\T_{i_\th}$ and $\varTheta_{02}^{-1}\varTheta_{01}\varTheta_{02}$ satisfy the same braid relation as $\ell_0$-braid relation. If $\ell_0=4$ there is nothing to prove. 

If $\ell_0=1$, then by applying Lemma \ref{basic-braid-rels1} for $a=\T_{i_\th}$, $b=\varTheta_{01}$, and $c=\varTheta_{02}$ we see that we can equivalently show that $T_{i_\th}$ and $\varTheta_{01}\varTheta_{02}$ satisfy the $2$-braid relation, or, after taking \eqref{centralrel-1} into account,  that $T_{i_\th}$ and $\Theta^{-1}\varTheta_{03}^{-1}$ satisfy the $2$-braid relation. But this is precisely Remark \ref{rem: double braid} for the affine Coxeter group generated by $\varTheta_{03}$, $\T_i$, $1\leq i\leq n$.

If $\ell_0=2$, then $\varTheta_{01}$ commutes with $\T_{i_\th}\varTheta_{01}\T_{i_\th}$ (by braid relations) and with $\T_{i_\th}^{-1}\varTheta_{02}\T_{i_\th}$. Therefore,  $\varTheta_{01}$ commutes with $\T_{i_\th}\varTheta_{01}\varTheta_{02}\T_{i_\th}$. Similarly, $\varTheta_{02}$ commutes with $\T_{i_\th}\varTheta_{01}\varTheta_{02}\T_{i_\th}$. Now, using these facts we obtain
\begin{align*}
\T_{i_\th} \varTheta_{02}^{-1}\varTheta_{01}\varTheta_{02}\T_{i_\th} \varTheta_{02}^{-1}\varTheta_{01}\varTheta_{02}&=
\T_{i_\th} \varTheta_{02}^{-1}\T_{i_\th}^{-1}\T_{i_\th} \varTheta_{01}\varTheta_{02}\T_{i_\th} \varTheta_{02}^{-1}\varTheta_{01}\varTheta_{02}\\
&= \T_{i_\th} \varTheta_{02}^{-1}\T_{i_\th}^{-1}\varTheta_{02}^{-1}\varTheta_{01}\varTheta_{02}\T_{i_\th}\varTheta_{01}\varTheta_{02}\T_{i_\th}\\
&= \varTheta_{02}^{-1}\T_{i_\th}^{-1}\varTheta_{02}^{-1}\T_{i_\th} \varTheta_{01}\varTheta_{02}\T_{i_\th}\varTheta_{01}\varTheta_{02}\T_{i_\th}\\
&= \varTheta_{02}^{-1}\varTheta_{01}\varTheta_{02}\T_{i_\th}\varTheta_{02}^{-1}\varTheta_{01}\varTheta_{02}\T_{i_\th},
\end{align*}
which is exactly our claim.

If $\ell_0=2$ we also need to verify that the image of the relations \eqref{ellbraid} through $\mfa$ are preserved. As explained in Proposition \ref{reducel=2} only one of the relations \eqref{ellbraid} is necessary in the definition. In our case it is clear that the image of the relation \eqref{ellbraid} for the pair $(1,3)$ is mapped to the relation \eqref{ellbraid} for the pair $(2,3)$. In conclusion, $\mfa$ extends indeed to an endomorphism of $\B(\DDDot{X}_n)$.

The fact that $\mfa$ and $\mfb$ are indeed isomorphisms can be seen by verifying that $\mfe\mfa\mfe$ is the inverse of $\mfb$ and that $\mfe\mfb\mfe$ is the inverse of $\mfa$. As for the  equality $\mfa\mfb\mfa=\mfb\mfa\mfb$,  it can be directly verified that
\begin{equation}\label{eq21}
\begin{aligned}
\mfa\mfb\mfa(\varTheta_{01})&=\varTheta_{03}=\mfb\mfa\mfb(\varTheta_{01})\\
\mfa\mfb\mfa(\varTheta_{02})&=\varTheta_{03}^{-1}\varTheta_{02}\varTheta_{03}=\mfb\mfa\mfb(\varTheta_{02})\\
\mfa\mfb\mfa(\varTheta_{03})&=\varTheta_{03}^{-1}\varTheta_{02}^{-1}\varTheta_{01}\varTheta_{02}\varTheta_{03}=\mfb\mfa\mfb(\varTheta_{03}).
\end{aligned}
\end{equation}

By using \eqref{eq21} we see that $(\mfa\mfb\mfa)^2$ acts on the affine generators by conjugation with $$(\varTheta_{01}\varTheta_{02}\varTheta_{03})^{-1}.$$ After taking into account \eqref{centralrel-1} this is the same as conjugation with $\varTheta$. But since in $C(X_n)$ we have $w_\circ=s_\th x$, $x\in \stab{\th}$, $\ell(w_\circ)=\ell(s_\th)+\ell(x)$, we obtain that $(\mfa\mfb\mfa)^2$ acts on the affine generators by conjugation with $\T_{w_\circ}$. If $\T_{w_\circ}$ is central in $B(X_n)$ then we can say that $(\mfa\mfb\mfa)^2$ acts on $\B(\DDDot{X}_n)$ by conjugation by $\T_{w_\circ}$. Otherwise, $\T^2_{w_\circ}$ is central in $B(X_n)$ and therefore $(\mfa\mfb\mfa)^4$ acts on $\B(\DDDot{X}_n)$ by conjugation by $\T^2_{w_\circ}$.
\end{proof}

\begin{Thm}\label{thm: auto1main}
The map sending $\mfu_1$, $\mfu_2$ to $\mfa$, $\mfb$, respectively, defines an injective group morphism  
$$\wGamma(1)\to \Aut(\B(\DDDot{X}_n); X_n).$$
The induced morphism
$$\Gamma_1(1)\to \Out(\B(\DDDot{X}_n); X_n)$$
is injective if $X_n$ is of type $A_n$, $n\geq 2$, $D_{2n}$, $n\geq 2$, $E_6$, and has kernel $\pm I_2$ otherwise. 
\end{Thm}
\begin{proof}
The fact that the maps define indeed group morphisms as specified  follows straight from Theorem \ref{thm: auto1}. Remark that the action of $\wGamma(1)$ descends to an action 
$$\Gamma_1(1)\to \Out(\Cox(\DDDot{X}_n)).$$
Through the isomorphism in the proof of Theorem \ref{main}, we can transfer this action to an action on the double affine Weyl group associated to the affine root system of type $X_n^{(1)}$.   A straightforward verification shows that $\mfa$ and $\mfb$ descend to the canonical action from \cite{ISEALA} of $u_{12}$ and $u_{21}$ on $\Cox(\DDDot{X}_n)$. The fact that the canonical action is faithful if $X_n$ is of type $A_n$, $n\geq 2$, $D_{2n}$, $n\geq 2$, $E_6$, and has kernel $\pm I_2$ otherwise implies that the same is true for the morphism $\Gamma_1(1)\to \Out(\B(\DDDot{X}_n); X_n)$.
From the commutative diagram
$$
\begin{CD}
\wGamma(1)     @>>>  \Aut(\B(\DDDot{X}_n);X_n)\\
@VVV        @VVV\\
\Gamma_1(1)    @>>>  \Out(\Cox(\DDDot{X_n}))
\end{CD}
$$
we obtain that the kernel of $\wGamma(1)\to \Aut(\B(\DDDot{X}_n); X_n)$ has to be a subgroup of the group generated by $\mfc$. But Theorem \ref{thm: auto1}ii) assures that this group acts faithfully on $\B(\DDDot{X}_n)$ implying that the action of $\wGamma(1)$ is faithful.
\end{proof}
The corresponding result for $\B(\DDot{C}_n)$ and  $\B(\DDot{C}_n)^{c}$ is the following.
\begin{Thm}\label{thm: auto2n2}
The action of $\wGamma(1)$ on $\B(\DDDot{C}_n)$ descends to an injective group morphisms
$$\wGamma(2)^\prime\to \Aut(\B(\DDot{C}_n); C_n),\quad\text{and}\quad \wGamma(2)^\prime\to \Aut(\B(\DDot{C}_n)^{c}; C_n).$$ 
The induced morphisms
$$\Gamma_1(2)^\prime\to \Out(\B(\DDot{C}_n); C_n),\quad\text{and}\quad \Gamma_1(2)^\prime\to \Out(\B(\DDot{C}_n)^{c}; C_n)$$
have kernel $\pm I_2$. 
\end{Thm}
\begin{proof}
Straightforward from Proposition \ref{comparisonbis}, Theorem \ref{thm: auto1main} and the fact that $\mfb\mfa\mfb^{-1}$ and $\mfb^2$ preserve  $\C$ and the normal subgroup generated by $\varTheta_{02}^2$.
\end{proof}

Let us record the explicit action of the generators of $\wGamma(2)^\prime$:
 \begin{align*}
\mfb\mfa\mfb^{-1}(\varTheta_0)&=\varPhi_0,& \mfb\mfa\mfb^{-1}(\varPhi_0)&=\varPhi_0^{-1}\varTheta_{0}\varPhi_0, \\ 
\mfb^2(\varTheta_0)&=\varTheta_0, & \mfb^2(\varPhi_0)&=\varTheta\varTheta_0\varPhi_0\varTheta_0^{-1}\varTheta^{-1}. 
\end{align*}

\subsection{}
Let us now consider the case of a Coxeter diagram of type $\fH{X}_n$.  We define two maps on the set of generators of $\B(\fH{X}_n)$ by letting them fix the generators corresponding to the finite nodes and by
 \begin{align*}
\mfa(\varTheta_0)&=\varPhi_0\varPhi\varTheta_0\varPhi^{-1}\varPhi_0^{-1},& \mfa(\varPhi_0)&=\varPhi_0, \\ 
\mfb(\varTheta_0)&=\varTheta_0, & \mfb(\varPhi_0)&=\varTheta\varTheta_0\varPhi_0\varTheta_0^{-1}\varTheta^{-1}. 
\end{align*}

\begin{Thm}\label{thm: auto2}
The above maps extend to elements of $\Aut(\B(\fH{X}_n);X_n)$. Moreover, $$\mfa^{-1}=\mfe\mfb\mfe\quad  \text{and}\quad \mfb^{-1}=\mfe\mfa\mfe.$$
\end{Thm}
\begin{proof}
We need to verify that the two maps preserve the relations in Definition \ref{def: twisted}. By Proposition \ref{superfluous} the braid relation between $\Ti$ and $\Tiii$ is a consequence of the other relations and does not need to be verified. For each of the two maps the verifications eventually boil down to relations in an affine Coxeter group, which is different for each map. The relevant affine Coxeter group is the one associated to an affine Coxeter diagram obtained as follows: for $\mfa$ remove the short affine node from $\fH{X}_n$, for  $\mfb$ remove the long affine node from $\fH{X}_n$. We will perform the verifications for $\mfb$. This has the advantage that we can use the results of Section \ref{affine-rels} without any notation adjustment. The verifications for $\mfa$ are entirely similar.

Let us first check that $\mfb(\Tiii)$ satisfies the braid relations with the generators corresponding to the simple nodes. For a node $i\neq i_\th$ the braid relations follow from the fact that $\T_i$ and $\Ti$ commute, and that $\T_i$ and $\varTheta$ commute (by Lemma \ref{lemma-simple-commuting-artin} for $\gamma=\th$). For $i_\th$, we have to argue that $\T_{i_\th}$ commutes with $\mfb(\Tiii)$ or, equivalently, that $\T_{i_\th}$ commutes with $(\Tiii\mfb(\Tiii))^{-1}$. The latter element, by \eqref{centralrel}, equals $\C^{-1}\varTheta\Ti\varPhi\Ti\Psi$ and the claim follows from \eqref{centralrel} and Lemma \ref{lemma-affine-commutation}ii).

Let us now check that the image of the relation \eqref{centralrel} is satisfied. Indeed,
\begin{align*}
\C&=\Tiii\varPhi\Ti\varPsi\Tiii\varTheta\Ti &&\\
&= \varTheta\Ti \Tiii\varPhi\Ti\varPsi\Tiii &&\\
&= \varTheta\Ti \Tiii\cdot \varPhi\Ti\varPsi\cdot \Tiii \cdot \Ti^{-1}\varTheta^{-1}\varTheta\Ti&&\\
&= \varTheta\Ti \Tiii \Ti^{-1}\varTheta^{-1} \cdot  \varPhi\Ti\varPsi \cdot \varTheta\Ti  \Tiii  \Ti^{-1}\varTheta^{-1}\cdot \varTheta\Ti && \text{by Lemma \ref{lemma-affine-commutation}ii)}\\
&=\mfb(\Tiii)\varPhi\mfb(\Ti)\varPsi\mfb({\Tiii})\varTheta\mfb(\Ti).
\end{align*}
The fact that they are indeed isomorphisms can be seen by verifying that $\mfe\mfa\mfe$ is the inverse of $\mfb$ and that $\mfe\mfb\mfe$ is the inverse of $\mfa$.
\end{proof}
\subsection{} We will now investigate the braid  relations satisfied by $\mfa$ and $\mfb$. We proceed with the case of a double laced double Coxeter diagram. 

\begin{Thm}\label{P1}
If $\fH{X}_n$ is double-laced, then $\mfa\mfb\mfa(\Ti)=\T_{w_\circ}\Ti\T_{w_\circ}^{-1}$ and $\mfb\mfa\mfb(\Tiii)=\T_{w_\circ}\Tiii\T_{w_\circ}^{-1}$. Furthermore, $(\mfa\mfb)^2=(\mfb\mfa)^2$ acts on $\B(\fH{X}_n)$ by conjugation by $\T_{w_\circ}$.
\end{Thm}
\begin{proof}
Remark that it is enough to show that $\mfa\mfb\mfa(\Ti)=\T_{w_\circ}\Ti\T_{w_\circ}^{-1}$ for any double-laced diagram $\fH{X}_n$. Indeed,
\begin{align*}
\mfb\mfa\mfb(\Tiii)&=\mfe\cdot\mfe\mfb\mfe\cdot\mfe\mfa\mfe\cdot\mfe\mfb\mfe(\Ti)\\
&=\mfe\mfa^{-1}\mfb^{-1}\mfa^{-1}(\Ti)\\
&=\mfe(\T_{w_\circ}^{-1}\Ti\T_{w_\circ})\\
&=\T_{w_\circ}\Tiii\T_{w_\circ}^{-1}.
\end{align*}
In preparation for computing $\mfa\mfb\mfa(\Ti)$ let us note that since $\mfa$ fixes $\C$ we obtain from \eqref{centralrel} that 
\begin{equation}\label{eq14}
\mfa(\Ti\Tiii\varPhi\Ti)=\Ti\Tiii\varPhi\Ti.
\end{equation}
Also, 
\begin{equation}\label{eq13}
\begin{aligned}
\Ti^{-1}\varTheta^{-1}\varPhi\Ti\varPhi^{-1}\varTheta\Ti&=\Ti^{-1}\varThetap\varPhip^{-1}\Ti\varPhip\varThetap^{-1}\Ti && \text{by Lemma \ref{lemma-wy-artin}i),iv)} \\
&= \varThetap\Ti^{-1}\varPhip^{-1}\Ti\varPhip\Ti\varThetap^{-1}  && \text{by Lemma \ref{B2-rels}iii)}\\
&= \varThetap\varPhip\Ti\varPhip^{-1}\varThetap^{-1} && \text{by Lemma \ref{B2-rels}iv)}\\
&= \varThetap\varPhip\varThetap\Ti\varThetap^{-1}\varPhip^{-1}\varThetap^{-1} && \text{by Lemma \ref{B2-rels}iii)}\\
&= \varPhi\Ti\varPhi^{-1} &&\text{by Lemma \ref{lemma-explicit-2}ii)}
\end{aligned}
\end{equation}
Now, 
\begin{align*}
\mfa\mfb\mfa(\Ti)&=\mfa(\varTheta\Ti\Tiii\cdot \Ti^{-1}\varTheta^{-1}\varPhi\Ti\varPhi^{-1}\varTheta\Ti\cdot \Tiii^{-1}\Ti^{-1}\varTheta^{-1})  &&  \\
&=\mfa(\varTheta\Ti\Tiii\cdot \varPhi\Ti\varPhi^{-1} \cdot \Tiii^{-1}\Ti^{-1}\varTheta^{-1}) && \text{by \eqref{eq13}}\\
&= \varTheta\Ti\Tiii\varPhi\Ti\varPhi^{-1}\Tiii^{-1}\cdot \Tiii\varPhi\Ti^{-1}\varPhi^{-1}\Tiii^{-1}\cdot \varTheta^{-1}  && \text{by \eqref{eq14}} \\
&= \varTheta\Ti\varTheta^{-1} &&  \\
&= \varTheta\varThetap\Ti\varThetap^{-1}\varTheta^{-1} &&  \text{by Lemma \ref{B2-rels}iii)}  \\
&= \T_{w_\circ}\Ti\T_{w_\circ}^{-1} &&  \text{by \eqref{eq8}}
\end{align*}
which is precisely our claim.

Furthermore, we have
\begin{align*}
(\mfa\mfb)^2(\Ti)&=\mfa\mfb\mfa(\Ti)\\
&= \T_{w_\circ}\Ti\T_{w_\circ}^{-1}\\
&=\mfb(\T_{w_\circ}\Ti\T_{w_\circ}^{-1})\\
&=\mfb(\mfa\mfb\mfa(\Ti))\\
&=(\mfb\mfa)^2(\Ti)
\end{align*}
and 
\begin{align*}
(\mfa\mfb)^2(\Tiii)&=\mfa(\mfb\mfa\mfb(\Tiii))\\
&= \mfa(\T_{w_\circ}\Tiii\T_{w_\circ}^{-1})\\
&=\T_{w_\circ}\Tiii\T_{w_\circ}^{-1}\\
&=\mfb\mfa\mfb(\Tiii)\\
&=(\mfb\mfa)^2(\Ti).
\end{align*}
The generators corresponding to the simple nodes are fixed by $\mfa$ and $\mfb$, but also by conjugation by $\T_{w_\circ}$ which is a central element in the subgroup they generate. 
\end{proof}
\begin{Thm}\label{thm: auto2main}
Assume that $X_n$ is double laced. The map sending $\mfu_1$, $\mfu_2^2$ to $\mfa$, $\mfb$, respectively, defines an injective group morphism  
$$\wGamma(2)\to \Aut(\B(\fH{X}_n); X_n).$$
The induced morphism
$$\Gamma_1(2)\to \Out(\B(\fH{X}_n); X_n)$$
is injective.
\end{Thm}
\begin{proof}
Straightforward from Theorem \ref{thm: auto2} and Theorem \ref{P1}.  The injectivity claim follows as in the proof of Theorem \ref{thm: auto1main}.
\end{proof}
\subsection{} Similar facts hold for the remaining double Coxeter diagram.

\begin{Thm}\label{P2}
For $\fH{G}_2$ we have $\mfa\mfb\mfa\mfb\mfa(\Ti)=\T^2_{w_\circ}\Ti\T_{w_\circ}^{-2}$ and $\mfb\mfa\mfb\mfa\mfb(\Tiii)=\T^2_{w_\circ}\Tiii\T_{w_\circ}^{-2}$. Furthermore, $(\mfa\mfb)^3=(\mfb\mfa)^3$ acts on $\B(\fH{G}_2)$ by conjugation by $\T^2_{w_\circ}$.
\end{Thm}
\begin{proof}
As in Proposition \ref{P1} it is enough to show that $\mfa\mfb\mfa\mfb\mfa(\Ti)=\T^2_{w_\circ}\Ti\T_{w_\circ}^{-2}$. We record a few equalities that will be used in the computation.
First, since $\mfa$ fixes $\C$ we obtain from \eqref{centralrel} that 
\begin{equation}\label{eq15}
\mfa(\Ti\Tiii\varPhi\Ti)=\Ti\Tiii\varPhi\Ti.
\end{equation}
Second, 
\begin{equation}\label{eq16}
\begin{aligned}
\Ti^{-1}\varTheta^{-1}\varPhi\Ti\varPhi^{-1}\varTheta\Ti&=\Ti^{-1}\T_{i_\ph}\T_{i_\th}^{-1}\Ti\T_{i_\th}\T_{i_\ph}^{-1}\Ti && \text{by Lemma \ref{lemma-wy-artin}i),vii)} \\
&= \T_{i_\ph}\Ti^{-1}\T_{i_\th}^{-1}\Ti\T_{i_\th}\Ti \T_{i_\ph}^{-1}  && \text{by braid relations}\\
&= \T_{i_\ph}\T_{i_\th} \T_{i_\ph}^{-1} && \text{by braid relations}
\end{aligned}
\end{equation}
Third, by using braid relations we obtain
\begin{equation}\label{eq17}
\begin{aligned}
 \T_{i_\ph}\T_{i_\th} \T_{i_\ph}^{-1}\Ti^{-1}\varTheta^{-1}\cdot  \T_{i_\ph}\T_{i_\th} \T_{i_\ph}^{-1}\cdot \varTheta\Ti  \T_{i_\ph}\T_{i_\th}^{-1} \T_{i_\ph}^{-1}
&= \T_{i_\ph}\T_{i_\th} \T_{i_\ph}^{-1}\Ti^{-1}    \T_{i_\ph}^2\T_{i_\th} \T_{i_\ph}^{-2}   \Ti  \T_{i_\ph}\T_{i_\th}^{-1} \T_{i_\ph}^{-1}\\
&= \T_{i_\ph}\T_{i_\th} \T_{i_\ph}\cdot \Ti^{-1}\T_{i_\th}\Ti\cdot \T_{i_\ph}^{-1}\T_{i_\th}^{-1} \T_{i_\ph}^{-1}  \\
&=  \T_{i_\ph}\T_{i_\th} \T_{i_\ph}\T_{i_\th}\Ti \T_{i_\th}^{-1} \T_{i_\ph}^{-1}\T_{i_\th}^{-1} \T_{i_\ph}^{-1}\\
&= \varPhi\Ti\varPhi^{-1}.
\end{aligned}
\end{equation}
Now, 
\begin{equation}\label{eq18}
\begin{aligned}
\mfb\mfa(\Ti)&= \varTheta\Ti\Tiii\cdot \Ti^{-1}\varTheta^{-1}\varPhi\Ti\varPhi^{-1}\varTheta\Ti\cdot \Tiii^{-1}\Ti^{-1}\varTheta^{-1}&&\\
&=\varTheta\Ti\Tiii\cdot \T_{i_\ph}\T_{i_\th} \T_{i_\ph}^{-1} \cdot \Tiii^{-1}\Ti^{-1}\varTheta^{-1}&& \text{by \eqref{eq16}}\\
\end{aligned}
\end{equation}
and 
\begin{equation}\label{eq19}
\mfb\mfa(\Tiii)= \varTheta\Ti\Tiii\Ti^{-1}\varTheta^{-1}.
\end{equation}
In consequence,
\begin{equation}\label{eq20}
\mfb\mfa(\Ti\Tiii)=\varTheta\Ti\Tiii\T_{i_\ph}\T_{i_\th} \T_{i_\ph}^{-1}\Ti^{-1}\varTheta^{-1}.
\end{equation}
From \eqref{eq18}, \eqref{eq20}, and \eqref{eq17} we obtain
\begin{align*}
\mfa\mfb\mfa\mfb\mfa(\Ti)&=\mfa(\varTheta^2\Ti\Tiii\cdot \varPhi\Ti\varPhi^{-1} \cdot \Tiii^{-1}\Ti^{-1}\varTheta^{-2}) && \\
&= \varTheta^2\Ti\Tiii\varPhi\Ti\varPhi^{-1}\Tiii^{-1}\cdot \Tiii\varPhi\Ti^{-1}\varPhi^{-1}\Tiii^{-1}\cdot \varTheta^{-2}  && \text{by \eqref{eq15}} \\
&= \varTheta^2\Ti\varTheta^{-2} &&  \\
&= \varTheta\T_{i_\ph}^2\Ti\T_{i_\ph}^{-2}\varTheta^{-1} &&  \text{by braid relations}  \\
&= \T_{w_\circ}^2\Ti\T_{w_\circ}^{-2} &&  \text{by \eqref{eq8}}
\end{align*}
Furthermore, we have
\begin{align*}
(\mfa\mfb)^3(\Ti)&=\mfa\mfb\mfa\mfb\mfa(\Ti)\\
&= \T^2_{w_\circ}\Ti\T_{w_\circ}^{-2}\\
&=\mfb(\T^2_{w_\circ}\Ti\T_{w_\circ}^{-2})\\
&=\mfb(\mfa\mfb\mfa\mfb\mfa(\Ti))\\
&=(\mfb\mfa)^3(\Ti)
\end{align*}
and 
\begin{align*}
(\mfa\mfb)^3(\Tiii)&=\mfa(\mfb\mfa\mfb\mfa\mfb(\Tiii))\\
&= \mfa(\T^2_{w_\circ}\Tiii\T_{w_\circ}^{-2})\\
&=\T^2_{w_\circ}\Tiii\T_{w_\circ}^{-2}\\
&=\mfb\mfa\mfb\mfa\mfb(\Tiii)\\
&=(\mfb\mfa)^3(\Ti).
\end{align*}
The generators corresponding to the simple nodes are fixed by $\mfa$ and $\mfb$, but also by conjugation by $\T_{w_\circ}$ which is a central element in the subgroup they generate. 
\end{proof}
\begin{Thm}\label{thm: auto3main}
The map sending $\mfu_1$, $\mfu_2^3$ to $\mfa$, $\mfb$, respectively, defines an injective group morphism  
$$\wGamma(3)\to \Aut(\B(\fH{G}_2); G_2).$$
The induced morphism
$$\Gamma_1(3)\to \Out(\B(\fH{G}_2); G_2)$$
is injective. 
\end{Thm}
\begin{proof}
Straightforward from Theorem \ref{thm: auto2} and Theorem \ref{P2}.  The injectivity claim follows as in the proof of Theorem \ref{thm: auto1main}.
\end{proof}
\subsection{} As an immediate consequence of Theorem \ref{thm: auto1}i), Theorem \ref{thm: auto1main}, Theorem \ref{thm: auto2n2}, Theorem \ref{thm: auto2}, Theorem \ref{thm: auto2main}, and Theorem \ref{thm: auto3main} we can state the following result.
\begin{Thm}\label{thm: autoall}
The group $\wXi(r)$ acts faithfully by morphisms or anti-morphisms on $\B(\DDDot{X}_n)$ (for $r=1$), $\B(\fH{X}_n)$ (for $r=2$, $X_n$ double laced), and $\B(\fH{G}_2)$ (for $r=3$).
Similarly, the group $\wXi(2)^\prime$ acts faithfully by morphisms or anti-morphisms on $\B(\DDot{C}_n)$ and $\B(\DDot{C}_n)^{c}$.
\end{Thm}
\begin{proof}
In each situation, extend the action of $\wGamma(r)$ by mapping $\mfv(r)$ to $\mfe$.
\end{proof}

\begin{Rem}
The action of the group $\wGamma(r)$ descends to the canonical action (see \cite{ISEALA}) of $\Gamma_1(r)$ on $\Cox(\DDDot{X}_n)$ (for $r=1$), $\Cox(\fH{X}_n)$ (for $r=2$, $X_n$ double laced), and $\Cox(\fH{G}_2)$ (for $r=3$). This can be directly verified by showing that the action of $\mfa$ and $\mfb$ coincides with the canonical action of $u_{12}$ and $u_{21}$ (for $r=1$), $u_{12}$ and $u_{21}^2$ (for $r=2$, $X_n$ double laced), and $u_{12}$ and $u_{21}^3$ (for $r=3$).  Similarly, the action of $\wGamma(2)^\prime$ descends to the canonical action of $\Gamma_1(2)^\prime$ on $\Cox(\DDot{C}_n)$.
\end{Rem}

\section{(Anti)Involutions}\label{sec: involutions}
\subsection{} Given the important role played by (anti)involutions in the theory of Hecke algebras (e.g. the Kazhdan-Lusztig involution) and their close relationship with the construction of invariant Hermitian forms in this context, it is perhaps useful to give an account of the (anti)involutions that originate from the constructions in \S\ref{sec: auto}.

We will be interested in involutions $\B(\DDDot{X}_n)\to \B(\DDDot{X}_n)$ and $\B(\fH{X}_n)\to \B(\fH{X}_n^\mfe)$ which map $\T_i$ , the generators corresponding to finite nodes, to $\T_{\mfe(i)}^{-1}$. By composing such an involution with the inverse map, we can equivalently describe the anti-involutions that map $\T_i$ to $\T_{\mfe(i)}$. The basic anti-involutions (defined below) are precisely the anti-involutions with this property.
\begin{Def}
Let $\gamma$ be an element of $\Aut(\B(\DDDot{X}_n);X_n)$ or $\Aut(\B(\fH{X}_n);X_n)$ with the property that
\begin{equation}\label{eq27}
\mfe\gamma\mfe=\gamma^{-1}.
\end{equation}
The map $\mfe\gamma: \B(\DDDot{X}_n)\to \B(\DDDot{X}_n)$ and respectively $\mfe\gamma: \B(\fH{X}_n)\to \B(\fH{X}_n^\mfe)$, which is a group anti-morphism whose inverse is 
$\gamma^{-1}\mfe: \B(\DDDot{X}_n)\to \B(\DDDot{X}_n)$ and respectively $\gamma^{-1}\mfe: \B(\fH{X}_n^\mfe)\to \B(\fH{X}_n)$,  will be called a basic anti-involution of $\B(\DDDot{X}_n)$ and respectively $\B(\fH{X}_n)$. The basic anti-involutions for $\B(\DDot{C}_n)$ and $\B(\DDot{C}_n)^{c}$ are defined in the same manner.
\end{Def}

\subsection{} For $1\leq r\leq 3$, let 
$$
\Upsilon_1(r)=\{A\in \Gamma_1(r)~|~e(r)Ae(r)=A^{-1}\}.
$$
A simple computation shows that 
$$
\Upsilon_1(r)=\left\{ \begin{bmatrix}a & b\\ -rb & d\end{bmatrix}~\Bigg\vert~a,b,d,\in \Z\right\}.
$$
We denote by $\wUpsilon(r)$  the inverse image of $\Upsilon_1(r)$  inside $\wGamma(r)$. Let  $\Upsilon_1(2)^\prime=\Upsilon_1(1)\cap \Gamma_1(2)^\prime$ and consider $\wUpsilon(2)^\prime=\wUpsilon(1)\cap\wGamma(2)^\prime$, its pre-image inside $\wGamma(2)^\prime$.
 \begin{Prop}\label{P3}
The set $\wUpsilon(r)$ is the subset of $\wGamma(r)$ consisting of elements $\widetilde{\gamma}$ with the property that
 \begin{equation}\label{eq28}
\mfv\widetilde{\gamma}\mfv=\widetilde{\gamma}^{-1},
\end{equation}
inside $\wXi(r)$. Consequently, $\wUpsilon(2)^\prime$ is the subset of $\wGamma(2)^\prime$  consisting of elements satisfying the same property inside $\wXi(2)^\prime$.
 \end{Prop}
 \begin{proof}
Note that the set of elements satisfying \eqref{eq28} is a subset of $\wUpsilon(r)$. We are left to showing the reverse inclusion.  Remark that any central element of $\wGamma(r)$  satisfies \eqref{eq28}, so it is enough to show that the pre-image of any
$$
\gamma=\begin{bmatrix}a & b\\ -rb & d\end{bmatrix}\in \Upsilon_1(r)
$$ 
contains one element that satisfies \eqref{eq28}. We show this by induction on $|b|\geq 0$.

For the initial check, if $b=0$ then $\mfc, \mfc^2$ are the desired elements of $\wGamma(r)$. If $r=1$ we also check the case when $|b|=1$. In this case, note that for all $N\in \Z$ we have $\mfu_1\mfu_2\mfu_1^N=\mfu_2^N\mfu_1\mfu_2$, $\mfu_1^N\mfu_2\mfu_1=\mfu_2\mfu_1\mfu_2^N$ and 
$$
\pi(\mfu_1\mfu_2\mfu_1^N)=\begin{bmatrix}0 & -1\\ 1 & -N+1\end{bmatrix} \quad \text{and}\quad \pi(\mfu_1^N\mfu_2\mfu_1)=\begin{bmatrix}-N+1 & -1\\ 1 & 0\end{bmatrix}.
$$
Therefore, the elements $\mfu_1\mfu_2\mfu_1^N$, $\mfu_1^N\mfu_2\mfu_1$ and their inverses satisfy \eqref{eq28} and at least one of them is in the pre-image of any possible matrix $\gamma$ with $|b|=1$.

For the induction step, let $|b|\geq 2$ if $r=1$  and $|b|\geq 1$ if $r=2,3$. Then, $a,d$ are necessarily non-zero. Remark that 
\begin{align*}
u_{12}\gamma u_{21}^r&=\begin{bmatrix}a+r(2b-d) & b-d\\ -r(b-d) & d\end{bmatrix},\quad & u_{21}^r\gamma u_{12}&=\begin{bmatrix}a & b-a\\ -r(b-a) & d+r(2b-a)\end{bmatrix},\\
u^{-1}_{12}\gamma u_{21}^{-r}&=\begin{bmatrix}a-r(2b+d) & b+d\\ -r(b+d) & d\end{bmatrix},\quad &u^{-r}_{21}\gamma u_{12}^{-1}&=\begin{bmatrix}a & b+a\\ -r(b+a) & d-r(2b+a)\end{bmatrix}.
\end{align*}
If $b>0$ then either $|b-|a||<|b|$ or $|b-|d||<|b|$ (or both). If $b<0$ then either $|b+|a||<|b|$ or $|b+|d||<|b|$ (or both). By the induction hypothesis, there exists $\widetilde{\eta}\in \wGamma(r)$ that satisfies property \eqref{eq28} such that at least one of the elements, all of which satisfy \eqref{eq28},
$$
\mfu_{1}\widetilde{\eta} \mfu_{2}^r, \quad \mfu_{2}^r\widetilde{\eta} \mfu_{1},\quad \mfu^{-1}_{1}\widetilde{\eta} \mfu_{2}^{-r},\quad \mfu^{-r}_{2}\widetilde{\eta} \mfu_{1}^{-1},
$$
are in the pre-image of $\gamma$. 
 \end{proof}
\subsection{} We are now ready to describe the basic anti-involutions 
 
 \begin{Thm}
 The elements of $ \mfe\wUpsilon(r)\subset\wXi(r)$ act as (distinct) basic anti-involutions of $\B(\DDDot{X}_n)$ (for $r=1$), $\B(\fH{X}_n)$ (for $r=2$, $X_n$ double laced), and $\B(\fH{G}_2)$ (for $r=3$).
Similarly, the elements of $ \mfe\wUpsilon(2)^\prime\subset\wXi(2)^\prime$ act as (distinct) basic anti-involutions of $\B(\DDot{C}_n)$ and $\B(\DDot{C}_n)^{c}$.
 \end{Thm}
\begin{proof}
Straightforward from Theorem \ref{thm: autoall} and Proposition \ref{P3}.
\end{proof}
\section{Other automorphisms}\label{sec: other}
\subsection{} In this section we give a brief account of some other automorphisms of $\B(\DDDot{X}_n)$ and $\B(\fH{X}_n)$ that are revealed by the Coxeter type presentation, being induced from  diagram automorphisms of the double affine Coxeter diagrams (not necessarily preserving the set of affine nodes).
\subsection{} 
Let us consider first the case of a Coxeter diagram of type $\DDDot{X}_n$. We define two maps on the set of generators of $\B(\DDDot{X}_n)$ by sending each generator corresponding a finite node to its inverse and by
 \begin{align*}
\mfe_{1}(\varTheta_{01})&=\varTheta_{02}^{-1}, & \mfe_1(\varTheta_{02})&=\varTheta_{01}^{-1}, &\mfe_1(\varTheta_{03})&=\T_{w_\circ}^{-1}\varTheta_{03}^{-1}\T_{w_\circ}, \\ 
\mfe_2(\varTheta_{01})&=\T_{w_\circ}\varTheta_{01}^{-1}\T_{w_\circ}^{-1}, & \mfe_2(\varTheta_{02})&=\varTheta_{03}^{-1}, & \mfe_2(\varTheta_{03})&=\varTheta_{02}^{-1}. 
\end{align*}
\begin{Prop}
The elements $\mfe_1,\mfe_2$ are order two automorphisms of $\B(\DDDot{X}_n)$. The group $\<\mfe_1,\mfe_2\>$  is isomorphic to the group freely generated by two order two elements. The automorphism $(\mfe_2\mfe_1)^3$ is inner and it generates a normal subgroup inside $\<\mfe_1,\mfe_2\>$. The quotient $\<\mfe_1,\mfe_2\>/\<(\mfe_2\mfe_1)^3\>$ is a subgroup of outer automorphisms of $\B(\DDDot{X}_n)$ isomorphic to the symmetric group $S_3$.
\end{Prop}
\begin{proof}
All verifications are straightforward.
\end{proof}
\subsection{} The double Coxeter diagram of type $\DDDot{A}_1$ has nontrivial automorphisms that do not preserve the set of affine nodes and these also induce automorphisms of $\B(\DDDot{A}_1)$. The full group of automorphisms of $\B(\DDDot{A}_1)$ is very large. It can be described as the subgroup of the group of automorphisms of the free group of rank $4$ (generated by $\varTheta_{01}$, $\varTheta_{02}$, $\varTheta_{03}$, $\varTheta$) that stabilizes $\{ \C, \C^{-1} \}$. It is easy to construct subgroups of automorphisms of $\B(\DDDot{A}_1)$ that are isomorphic to the cyclic group of order four, or to the Klein four-group.

\subsection{} The only nontrivial automorphism induced from automorphisms of double Coxeter diagrams of type $\DDot{C}_n$, $n\geq 2$, $\fH{F}_4$ and $\fH{G}_2$ is the anti-involution $\mfe$ composed with the map the inverts each generator. This automorphism corresponds to the unique automorphism of the double Coxeter diagram that swaps the two affine nodes.

The only nontrivial automorphism of the double Coxeter diagrams of type $\fH{B}_n/\fH{C}_n$, $n\geq 3$ fixes all nodes except the affine node of degree three and the finite node connected to the affine node of degree four, which are swapped. The induced map on generators, composed with the map that acts by inverting the generators is, as it can be directly verified, an automorphism of $\B(\fH{B}_n/\fH{C}_n)$.

The automorphism group of the double Coxeter diagram of type $\fH{B}_2/\fH{C}_2$ is $D_8$, the dihedral group of order eight.  Note that in this case $\varThetap$ and $\varPhip$ are precisely the generators associated to the finite nodes. More precisely, $\varThetap$ labels the finite node connected to $\Tiii$ and $\varPhip$ labels the finite node connected to $\Ti$. We define two maps on the set of generators of $\B(\fH{B}_2/\fH{C}_2)$ by
 \begin{align*}
\mfs(\Ti)&=\varPhip^{-1}, & \mfs(\Tiii)&=\varThetap^{-1}, &\mfs(\varThetap)&=\Tiii^{-1}, &\mfs(\varPhip)&=\Ti^{-1},\\ 
\mfr(\Ti)&=\varPhip, & \mfr(\Tiii)&=\Ti, &\mfr(\varThetap)&=\Tiii, &\mfr(\varPhip)&=\varThetap. 
\end{align*}
\begin{Prop}
The elements $\mfs,\mfr$ are  automorphisms of $\B(\fH{B}_2/\fH{C}_2)$ of order two and, respectively, four, and satisfy the relation $\mfs\mfr\mfs=\mfr^{-1}$. The group $\<\mfs,\mfr\>$  is isomorphic to $D_8$ and consists of non-inner automorphisms.
\end{Prop}
\begin{proof}
All verifications are straightforward.
\end{proof}

\subsection{} 

The automorphism group of the double Coxeter diagram of type $\DDot{A}_1$ is the permutation group $S_3$ (or $D_6$, the dihedral group of order six).  Note that in this case $\varTheta$ is precisely the generator associated to the finite node. We define two maps on the set of generators of $\B(\DDot{A}_1)$ by
 \begin{align*}
\mfs(\Ti)&=\varTheta^{-1}, & \mfs(\Tiii)&=\Tiii^{-1}, &\mfs(\varTheta)&=\Ti^{-1}, \\ 
\mfr(\Ti)&=\Tiii, & \mfr(\Tiii)&=\varTheta, &\mfr(\varTheta)&=\Ti. 
\end{align*}
\begin{Prop}
The elements $\mfs,\mfr$ are  automorphisms of $\B(\DDot{A}_1)$ of order two and, respectively, three, and satisfy the relation $\mfs\mfr\mfs=\mfr^{-1}$. The group $\<\mfs,\mfr\>$  is isomorphic to $S_3$ and consists of non-inner automorphisms.
\end{Prop}
\begin{proof}
All verifications are straightforward.
\end{proof}
The full group of automorphisms of $\B(\DDot{A}_1)$ is, again, very large. It can be described as the subgroup of the group of automorphisms of the free group of rank 3 (generatd by $\Ti$, $\Tiii$, and $\varTheta$) that stabilizes $\{ \C, \C^{-1} \}$.

 \section{Reductive group data}\label{sec: reductive-DABG}

\subsection{} One construction in the literature \cite{HaiChe} attaches slightly more general double affine Artin groups to data typically used to describe reductive groups. We summarize this construction (following \cite{HaiChe}, up to duality) emphasising the relationship with the double affine Artin groups.

\subsection{}
A reductive datum is a triple $(Y,R_0,R_0^\vee)$  such that $Y$ is a free $\Z$-module of $\rank(Y)\geq n$, $R_0\subset Y$ is an irreducible finite root system of rank $n$, and $R_0^\vee\subset Y^\vee=\Hom_\Z(Y,\Z)$ is the dual root system. It is understood that we are provided with a bijection between $R_0$ and $R_0^\vee$, so that $R_0^\vee$ can be considered as the set of coroots associated to $R_0$. If $R_0$ spans the vector space $Y\otimes_\Z \Rat$ we say that the root datum  $(Y,R_0,R_0^\vee)$ is semisimple.

Let $Q_0\subset Y$, $Q_0^\vee\subset Y^\vee$, and $W_0$ be the root lattice, coroot lattice, and the Weyl group of $R_0$, respectively. As before, we fix a basis $\{\a_i\}_{1\leq i\leq n}$ of $R_0$ and denote by $\{\a_i^\vee\}_{1\leq i\leq n}$ the corresponding coroots. Note that  $Y$ and $Y^\vee$ are stable under the usual action of $W_0$ by reflections. The affine Weyl group associated to $(Y,R_0,R_0^\vee)$ is defined as $W_0\ltimes Y^\vee$. This contains  $W_0\ltimes Q_0^\vee$ as a normal subgroup. If $\mathbf{\Pi}$ denotes their quotient (which is isomorphic to $Y^\vee/Q_0^\vee$) then
$$
W_0\ltimes Y^\vee\cong \mathbf{\Pi}\ltimes (W_0\ltimes Q_0^\vee).
$$

The Smith Normal Form for the inclusion $Q_0\subset Y$ of free $\Z$-modules produces a basis $$\{e_1,\dots,e_{\rank(Y)}\}$$ of $Y$ and integers $k_1,\dots,k_n$ such that $\{k_1e_1,\dots,k_ne_n\}$ is a basis of $Q_0$. Furthermore, since the pairing between $Q_0$ and $Q_0^\vee$ is non-degenerate it can be arranged that the elements of $Q_0^\vee$ vanish on $\{e_{n+1},\dots,e_{\rank(Y)}\}$. Therefore, if  $Y_s$ is the $\Z$-span of $\{e_1, \dots, e_n\}$ and $Y_c$ is the $\Z$-span of $\{e_{n+1}, \dots, e_{\rank(Y)}\}$, and $Y^\vee_s$, $Y^\vee_c$ the corresponding dual $  \Z$-modules,  we have the following splittings 
\begin{equation*}
Y= Y_s\oplus Y_c, \quad \text{ and }\quad Y^\vee= Y^\vee_s\oplus Y^\vee_c.
\end{equation*}
Remark that  $(Y_s,R_0,R_0^\vee)$ is a semisimple root datum. We write
$$
 (Y,R_0,R_0^\vee)= (Y_s,R_0,R_0^\vee)\oplus  (Y_c,\emptyset,\emptyset)
$$
to refer to these direct sum decompositions of $Y$ and $Y^\vee$. The notation $Y_s$ and $Y_c$ is suggestive of the fact that, when the reductive datum arises from a reductive group, they correspond to the characters of the semisimple and, respectively, central part of a maximal torus. This splitting induces a corresponding splitting 
$$
\mathbf{\Pi}=\mathbf{\Pi}_s\times Y^\vee_c,
$$
of $\mathbf{\Pi}$ into its torsion and free components. 

Since $W_0$ acts trivially on $Y^\vee_c$ the affine Weyl group associated to $(Y,R_0,R_0^\vee)$ is the direct product between the Weyl groups associated to $(Y_s,R_0,R_0^\vee)$ and $(Y_c, \emptyset, \emptyset)$. In other words,
$$W_0\ltimes Y^\vee\cong (W_0\ltimes Y^\vee_s)\times Y_c^\vee \cong(\mathbf{\Pi}_s\ltimes (W_0\ltimes Q_0^\vee))\times Y^\vee_c.$$
The lattice $Y^\vee_s$, being a $W_0$-stable lattice that takes integral values on $Q_0$ must lie between $Q_0^\vee$ and $P_0^\vee$, the coweight lattice of $R_0$. In consequence, 
$$
W_0\ltimes Q_0^\vee\leq W_0\ltimes Y^\vee_s\leq W_0\ltimes P_0^\vee.
$$

The affine Artin group $\A(Y, W_0)$ associated to $(Y,R_0,R_0^\vee)$ is defined as the group generated by the finite Artin group $\A(W_0)$ and the lattice $Y^\vee$  such that the following relations are satisfied for all $1\leq i\leq n$ and $\mu\in Y^\vee$
\begin{subequations} 
 \begin{alignat}{2}
& T_iY_\mu=Y_\mu T_i   \text{ if } \mu(\a_i)=0, \\
& T_iY_\mu T_i=s_i(Y_{\mu}) \text{ if } \mu(\a_i)=1.
\end{alignat}
\end{subequations}
Above, for $\mu\in Y^\vee$, the corresponding element of $\A(Y,W_0)$ is denoted by $Y_\mu$, and the group operation is denoted multiplicatively (i.e. $Y_{\mu}Y_{\nu}=Y_{\mu+\nu}$).

In this case also, the affine Artin group associated to $(Y,R_0,R_0^\vee)$ is the direct product between the Artin groups associated to $(Y_s,R_0,R_0^\vee)$ and $(Y_c, \emptyset, \emptyset)$. In other words,
$$\A(Y, W_0)\cong \A(Y_s,W_0)\times Y_c^\vee.$$

\subsection{} A double reductive datum consists of two reductive data $(X,R_1,R_1^\vee)$ and $(Y,R_2,R_2^\vee)$ together with
\begin{enumerate}[label={\alph*)}]
\item an isomorphism $(W_1, R_1)\cong (W_2,R_2)$ of Coxeter systems;
\item a $W_1$-invariant pairing $(X^\vee,Y^\vee)\to \Rat$;
\item an integer $m$ such that $(X^\vee,Y^\vee)\subseteq \frac{1}{m}\Z$.
\end{enumerate}
The double reductive datum is said to be semisimple if both $(X,R_1,R_1^\vee)$ and $(Y,R_2,R_2^\vee)$ are semisimple reductive data.

An isomorphism of Coxeter systems is always constructed from a bijection between fixed bases of $R_1$ and $R_2$. Therefore, bases of $R_1$ and $R_2$ are fixed and the corresponding simple roots are set in correspondence. We will use $W_0$ to refer to $W_1$ and $W_2$ identified through the isomorphism. The pairing between $X$ and $Y$, when restricted to a paring between $Q_1$ and $Q_2$,  is unique up to scaling and can be normalized so that for simple roots it matches (through the above mentioned bijection) the canonical pairing between $Q_1$ and $Q_1^\vee$. We assume that the pairing is normalized in this fashion and we also consider the induced pairing between $X^\vee$ and $Y^\vee$. Finally, note that, being $W_0$-invariant the pairing must satisfy
$$
(X^\vee_s,Y^\vee_c)=(X^\vee_c, Y^\vee_s)=0.
$$

Let $\widetilde{X}^\vee=X^\vee\oplus \frac{1}{m}\Z\d$, and consider the corresponding notation for $X_s$ and $X_c$. The action of $W_0$ on $X^\vee$ can be extended to $\widetilde{X}$ by letting $W_0$ act trivially on $\d$. We have an action of $Y^\vee$ on $\widetilde{X}^\vee$ defined by 
$$
Y_\mu(X_\l+a\d)=X_\l+a\d-(X_\l,Y_\mu)\d,
$$
for any $Y_\mu\in Y^\vee$ and $X_\l\in X^\vee$.

The double affine Weyl group associated to 
$\left((X,R_1,R_1^\vee)~ ; ~ (Y,R_2,R_2^\vee)\right)$  is defined as $(W_0\ltimes Y^\vee)\ltimes \widetilde{X}^\vee$. In particular, if $W_0\ltimes Y^\vee$ is presented as $\mathbf{\Pi}_2\ltimes (W_0\ltimes Q_2^\vee)$ we have an action of $\mathbf{\Pi}_2$ on $\widetilde{X}^\vee$ and the double affine Weyl group is isomorphic to 
$$
\mathbf{\Pi}_2\ltimes \left( (W_0\ltimes Q_2^\vee)\ltimes \widetilde{X}\right).
$$

The splittings on $X$ and $Y$ into their semisimple and central components induce an isomorphism between the double affine Weyl group and the almost-direct product of  $(W_0\ltimes Y_s^\vee)\ltimes \widetilde{X}_s^\vee$ and $Y_c\ltimes \widetilde{X}_c^\vee$  (i.e. direct product modulo the identification of the two copies of  $\frac{1}{m}\Z\d$ in $\widetilde{X}_s^\vee$ and $\widetilde{X}_c^\vee$). In other words, double affine Weyl group associated to double reductive datum $\left((X,R_1,R_1^\vee)~ ; ~ (Y,R_2,R_2^\vee)\right)$ is the almost-direct product of the double affine Weyl groups associated to  $\left((X_s,R_1,R_1^\vee)~ ; ~ (Y_s,R_2,R_2^\vee)\right)$ and $\left((X_c, 0, 0)~ ; ~ (Y_c,0, 0)\right)$.

The double affine Artin group $\A(X, Q_2, W_0)$ associated to $\left((X,R_1,R_1^\vee)~ ; ~ (Q_2,R_2,R_2^\vee)\right)$ is defined as the group generated by the affine Artin group $\A(Q_2, W_0)$ and the lattice $\widetilde{X}^\vee$ (with the group operation denoted now multiplicatively) such that the following relations are satisfied for all $0\leq i\leq n$ and $X_\b\in \widetilde{X}^\vee$
\begin{subequations} 
 \begin{alignat}{2} 
& T_iX_\b=X_\b T_i \text{ if }  (\b,\a_i)=0,\\ 
& T_iX_\b T_i=s_i(X_{\b})   \text{ if }  (\b,\a_i)=-1.
\end{alignat}
\end{subequations}
Above, $s_0$ refers to the element $(s_{\th}, -\th^\vee)\in W_0\ltimes Q_2^\vee$, where $\th\in R_2$ is the highest root, and $T_0$ to the corresponding element of $\A(Q_2, W_0)$. It is clear from the definition that $\A(X, Q_2, W_0)$ is isomorphic to the direct product $$\A(X_s, Q_2, W_0)\times X_c^\vee.$$

The double affine Artin group $\A(X, Y, W_0)$ associated to $\left((X,R_1,R_1^\vee)~ ; ~ (Y,R_2,R_2^\vee)\right)$ is defined as
$$
\mathbf{\Pi}_2\ltimes \A(X, Q_2, W_0)\cong (\mathbf{\Pi}_{2,s}\times Y^\vee_c)\ltimes (\A(X_s, Q_2, W_0)\times X^\vee_c).
$$

The  almost-direct product 
$$
\A(X_s, Y_s, W_0)\overline{\times} \A(X_c, Y_c, \{1\})
$$
is defined as the quotient of the direct product of two double affine Artin groups modulo the identification of  the two copies of $\frac{1}{m}\Z\d$ in $\widetilde{X}_s^\vee$ and $\widetilde{X}_c^\vee$. Remark that $\A(X_c, Y_c, \{1\})$ is isomorphic to  (possibly a trivial central extension of) the Heisenberg group associated to the pairing of $X_c^\vee$ and $Y_c^\vee$. 
\begin{Prop}\label{prop: reductive}
The double affine Artin group $\A(X, Y, W_0)$ is isomorphic to 
$$
\A(X_s, Y_s, W_0)\overline{\times} \A(X_c, Y_c, \{1\})
$$
\end{Prop}
\begin{proof}
Straightforward from the definition of $\A(X, Y, W_0), \A(X_s, Y_s, W_0)$ and $\A(X_c, Y_c, \{1\})$.
\end{proof}
Therefore, the structure of the double affine Artin groups associated to double reductive data is fully captured by the double affine Artin groups associated to underlying semisimple double data which will be studied in the following section.  

\begin{Rem}
For fixed root systems  $R_1$ and $R_2$, the double affine Artin groups associated to semisimple data we have those attached to the smallest possible $X$ and $Y$, that is for $X=Q_1$ and $Y=Q_2$, and those attached to  the largest possible $X$ and $Y$, that is to $P_1$ and $P_2$, the weight lattices of $R_1$ and $R_2$. We call the former type of data, adjoint data and the latter, simply connected semisimple data. By definition, we have
\begin{equation}\label{eq: intermediate}
\A(Q_1, Q_2, W_0)\leq \A(X_s, Y_s, W_0)\leq \A(P_1, P_2, W_0).
\end{equation}

The double affine Artin groups associated to adjoint semisimple data are precisely the double affine Artin groups discussed in \S\ref{sec: DAAGs}-\ref{sec: other}. Indeed, there are two possibilities for $R_1$ and $R_2$: they are either of the same Dynkin type $X_n$ or of dual (but not equal) Dynkin type, say $X_n$ and $X_n^\vee$, respectively. In both cases we can and will choose $m=1$, choosing a larger $m$  having the effect of producing a trivial central extension of the double affine Artin group for $m=1$. In the first situation, we obtain the double Artin group associated to the affine root system $X_n^{(1)}$; in the second situation we obtain  the double Artin group associated to the affine root system $D_{n+1}^{(2)}$, $A_{2n-1}^{(2)}$, $E_6^{(2)}$, $D_4^{(3)}$ if $X_n$ is of type $B_n$, $C_n$, $F_4$, $G_2$, respectively. 

The double affine Artin groups associated to simply connected semisimple data are known in the literature under the name  \emph{extended} double affine Artin groups. We discuss them in some detail in the following section. 

\end{Rem}

 \section{Extended double affine Artin groups}\label{sec: ext-DABG}
 
 \subsection{} Following the original work of Cherednik \cites{CheDou, CheMac, CheDou-2, MacAff} in this area, some  objects that are frequently considered in the literature are the so-called extended double affine Artin groups (and Hecke algebras). 
It might be of interest to discuss the relationship between double affine Artin groups and extended double affine Artin groups and  give a description of the latter that highlights this relationship. For the details of the constructions that follow we refer to \cites{CheDou, CheMac, CheDou-2, MacAff}. 
 

 \subsection{}\label{sec: extended-cherednik}  Let $A$ be an indecomposable affine Cartan matrix of rank $n$ and recall the notation set-up in \S\ref{sec: DAAGs}. In the case of $A_{2n}^{(2)}$, $n\geq 1$, the analysis that follows does not produce a new object, so we also adhere to the convention in \S\ref{bigdaddy}. In addition, let $N\subset \Hring^*_\Re$ be equal to $\nu(\Pring^\vee)$ if $r=1$ and equal to $\Pring$ if $r=2,3$. We will also need the following notation
 \begin{equation}
\P_Y:=\{Y_\mu;\mu\in N\}\quad \text{and} \quad\P_X:=\{X_\b;\b\in \Pring^\vee\}.
\end{equation}
The fundamental weigths and coweights with respect to the basis $\{\a_i\}_{1\leq i\leq n}$ are denoted by $\{\l_i\}$  and, respectively, $\{\l_i^\vee\}$. 
 
 Let $\mho$ denote the (finite) group of outer diagram automorphisms of the Dynkin diagram $D(A)$. An element of $\mho$ is uniquely determined by its action on the affine node. The action of $\mho$ on the nodes of $D(A)$ induces an action of $\mho$ on $\A(R)$ as group automorphisms as well as a linear action on $Q^\vee$ which fixes $\d$. The semidirect product $\mho\ltimes \A(R)$ is referred to as the extended affine Artin group. Unless $\mho$ is trivial, the extended affine Artin group is not a Coxeter braid group and therefore it does not admit a Coxeter presentation. However, it does admit the following Bernstein-type presentation (see e.g. \cite{MacAff}*{\S 3.1-3.3}).
 
 \begin{Prop}\label{bernstein-presentation-extended}
The extended affine Artin group $\mho\ltimes \A(R)$ is generated by the finite Artin group $\A(\Rring)$ and the lattice $\P_Y$ such that the following relations are satisfied for all $1\leq i\leq n$ and $\mu\in N$
\begin{subequations} 
 \begin{alignat}{2}
& T_iY_\mu=Y_\mu T_i   \text{ if } (\mu, A_i^\vee)=0, \\
& T_iY_\mu T_i=Y_{s_i(\mu)} \text{ if } (\mu,A_i^\vee)=1.
\end{alignat}
\end{subequations}
\end{Prop}


 \subsection{} \label{rem: daffine-auto}  

The actions of $\mho$ on $\A(R)$ and $Q^\vee$ are compatible with the relations \eqref{dc1}, \eqref{dc2}, \eqref{dc3}. Therefore, there is canonical action of $\mho$ on $\Atilde(R)$. The semidirect product $\Atilde(R)_Y=\mho\ltimes\Atilde(R)$ will be called the $Y$-extended double affine Artin group.


 \subsection{} 
Following Cherednik \cite{CheDou}, we define a group, which shall be referred to as the $X$-extended double affine Artin group, as the group defined as in Proposition \ref{defcherednik}  but allowing $\b\in \Pring^\vee$. 
\begin{Def}\label{defcherednik-X}
The $X$-extended double affine Artin group $\Atilde(R)_X$ is generated by the affine Artin group $\A(R)$, the lattice $\P_X$, and the element $X_{\d}$  such that the following relations are satisfied for all $0\leq i\leq n$ and $\b\in \Pring^\vee$
\begin{subequations} 
 \begin{alignat}{2} \label{dc1e}
& T_iX_\b=X_\b T_i \text{ if }  (\b,\a_i)=0,\\ \label{dc2e}
& T_iX_\b T_i=X_{s_i(\b)}   \text{ if }  (\b,\a_i)=-1,\\ \label{dc3e}
& X_{\d} \text{ is central}.
\end{alignat}
\end{subequations}
For a positive integer $m$, define  $\Atilde(R)_{X,\frac{1}{m}}$ as the trivial central extension of $\Atilde(R)_X$, with center generated by an element  $X_{\frac{1}{m}\d}$ such that
$$X_{\frac{1}{m}\d}^m=X_{\d}.$$
\end{Def}
For convenience, we denote by $\Atilde(R)_{\frac{1}{m}}$ the corresponding trivial central extension of $\Atilde(R)$ and we adopt a similar notation of the groups $\B(\DDDot{X}_n)$ and $\B(\fH{X}_n)$.

The linear action of $\mho$ on $Q^\vee$ extends to a linear action on $\Pring^\vee\oplus \frac{1}{m}\Z\d$, where $m$ is the smallest positive integer with the property that
$$
(N,\Pring^\vee)\subseteq \frac{1}{m}\Z.
$$
As a result, the action of $\mho$ extends to $\Atilde(R)_{X,\frac{1}{m}}$. We denote the corresponding semidirect product $$\Atilde(R)_{X,Y}=\mho\ltimes\Atilde(R)_{X,\frac{1}{m}}$$ is called the $X,Y$-extended double affine Artin group. By definition
\begin{equation}\label{eq: omega-X-rels}
\omega X_{\b}\omega^{-1}=X_{\omega(\b)}, \quad \text{ for all} \quad \omega\in \mho, \b\in\Pring^\vee.
\end{equation}

The $X,Y$-extended double affine Artin group is precisely $\Atilde(R)$ if $\mho$ is trivial, which is the case for $A=E_8^{(1)}, F_4^{(1)}, G_2^{(1)}, E_6^{(2)}, D_4^{(3)}$. For the other cases, we note that  $m$ is equal to $n+1$ for $A_n^{(1)}$, 4 for $D_{2n+1}^{(1)}$, 3 for $E_6^{(1)}$,  2 for $C_{2n+1}^{(1)}, D_{2n}^{(1)}, E_7^{(1)}, A_{2n-1}^{(2)}, D_{n+1}^{(2)}$, and 1 for $B_n^{(1)}, C_{2n}^{(1)}$. If $\mho$ is trivial or if $A=C_{4n}^{(1)}, D_{4n}^{(1)}, E_6^{(1)}, A_{4n-1}^{(2)}, D_{2n+1}^{(2)}$, the action of $\wGamma(r)$ on $\Atilde(R)$ extends to $\Atilde(R)_{X,Y}$. Otherwise, to be able to extended the action of $\wGamma(r)$ on $\Atilde(R)$ we need to consider the group
$$
\mho\ltimes\Atilde(R)_{X,\frac{1}{2m}}.
$$
  The group on which $\wGamma(r)$ acts, as indicated above, is denoted by $\Atilde(R)^e$ and called the (fully) extended double affine Artin group.

\subsection{} 
We will now give a succinct translation of the above constructions in terms of the Coxeter type presentation of the double affine Artin group. We note that the actions induced from the group of outer automorphisms of an affine Dynkin diagram, although explicit, are considerably less intuitive in the Coxeter type presentation than in the Bernstein type presentation.

To fix notation, we use the labeling set up in \S\ref{sec: labeling}. Denote by  $\dot{X}_n$  the affine sub-diagram of  $\DDDot{X}_n$ or $\fH{X}_n$ consisting of the finite nodes together with $\varTheta_{01}$ or, $\Ti$, respectively. We will describe the action of the automorphisms induced from the group $\mho$ of outer automorphisms of $\dot{X}_n$, which under the present conventions correspond precisely to the automorphisms described in \S\ref{rem: daffine-auto}.  Since the elements of $\mho$ are determined by the action on the affine node we denote by $\omega_i$ the element that sends the affine node to the finite node labelled by $i$. We denote by $i^*$ the (finite) node for which $\omega_{i^*}=\omega_i^{-1}$. 

We denote by $\dot{X}_n^\mfe$ the indicated affine sub-diagram of $\DDDot{X}_n^\mfe$ or $\fH{X}_n^\mfe$ and by $\mho(\dot{X}_n^\mfe)$ the corresponding group. Then, $\mho^\mfe=\mfe \mho(\dot{X}_n^\mfe)\mfe$ is the group of outer automorphisms of $\mfe(\dot{X}_n^\mfe)$ which is a affine sub-diagram of $\DDDot{X}_n$ or $\fH{X}_n$ that consists of the finite nodes together with an affine node. The elements of $\mho^\mfe$ are determined by the action on this affine node and we denote by $\varpi_j=\mfe\omega_j(\dot{X}_n^\mfe)\mfe$ the element that sends the affine node to the finite node labelled by $j$. As abstract groups, $\mho$ and $\mho^\mfe$ are always isomorphic. The non-trivial elements of $\mho$ and $\mho^\mfe$ are labelled by the same set of finite nodes for diagrams of type $\DDDot{X}_n$.

\subsection{}\label{sec: wi-notation} For a double Coxeter diagram of type $\DDDot{X}_n$ and $\omega_i\in\mho$ non-trivial, let $w_i\in C(X_n)$ be the minimal length element for which $\varTheta=\T_{w_i^{-1}}\T_i\T_{w_i}$.  The following properties can be directly verified
\begin{subequations}\label{eq29}
\begin{alignat}{3}
\varTheta\T_j\varTheta^{-1}&=\T_{w_i^{-1}} \T_{\omega_i(j)} \T_{w_i}, &\quad  & \text{for } 1\leq j\leq n, ~j\neq i,i^*,\\
\varTheta\T_i\varTheta^{-1} &= \T_{w_{i^*}^{-1}} \T_{w_{i^*}}^{-1} \T_{i^*} \T_{w_{i^*}} \T^{-1}_{w^{-1}_{i^*}}.& &
\end{alignat}
\end{subequations}
The induced automorphism of $\B(\DDDot{X}_n)$ is still denoted by $\omega_i$. We have $\omega_i(\T_{w^{-1}_{i^*}})=\T_{w_i}$. The automorphisms $\omega_i$ and $\varpi_i$  act on the affine generators and $\varTheta$ as follows
 \begin{align*}
\omega_i(\varTheta_{01})&=\T_i, & \omega_i(\varTheta_{02})&=\T_{w_i}\varTheta^{-1}\varTheta_{03}\varTheta\T_{w_i}^{-1}, \\ \omega_i(\varTheta_{03})&=\T_{w_i}\varTheta_{01}\varTheta_{02}\varTheta_{01}^{-1}\T_{w_i}^{-1}, & \omega_i(\varTheta)&=T_{w_i}\varTheta_{01}\T_{w_i^{-1}},\\
\varpi_{i}(\varTheta_{03})&=\T_i, & \varpi_{i}(\varTheta_{02})&=\T^{-1}_{w_i^{-1}}\varTheta\varTheta_{01}\varTheta^{-1}\T_{w_i^{-1}}, \\ \varpi_{i}(\varTheta_{01})&=\T^{-1}_{w^{-1}_i}\varTheta_{03}^{-1}\varTheta_{02}\varTheta_{03}\T_{w_i^{-1}}  & \varpi_i(\varTheta)&=T_{w_i}\varTheta_{03}\T_{w_i^{-1}}. 
\end{align*}

\subsection{} For a double Coxeter diagram of type $\fH{X}_n$, we have a nontrivial group $\mho$ only if $X_n=B_n, C_n$, in which case $\mho$ has order 2. In this situation $\mho$ contains a unique nontrivial element, denoted by $\omega_i$, and $\mho^\mfe$ contains a unique nontrivial element, denoted by $\varpi_{i^\mfe}$.

 Let $\gamma_1\in C(X_n)$ be the minimal length element for which $\varThetap=T_{\gamma_1}T_iT_{\gamma_1^{-1}}$ and let $\gamma_2\in C(X_n)$ be the element such that $T_{\gamma_2}=\omega_i(T_{\gamma_1})$.  Let $\eta_1\in C(X_n)$ be the minimal length element for which $\varPhip=T_{\eta_1}T_{i^\mfe}T_{\eta_1^{-1}}$ and let $\eta_2\in C(X_n)$ be the element such that $T_{\eta_2}=\varpi_{i^\mfe}(T_{\eta_1})$.  We note that $\gamma_1$ and $\gamma_2$ are both trivial if $X_n=C_n$ and that $\eta_1$ and $\eta_2$ are both trivial if $X_n=B_n$. The induced automorphisms of $\B(\fH{X}_n)$ are still denoted by $\omega_i$ and $\varpi_i^\mfe$. The automorphisms $\omega_i$ and $\varpi_{i^\mfe}$  act on the affine generators and $\varThetap$ as follows
 \begin{align*}
\omega_i(\Ti)&=\T_i, & \omega_i(\Tiii)&=\T_{\gamma_2\gamma_1^{-1}}\varThetap^{-1}\Ti\Tiii\Ti^{-1}\varThetap\T_{\gamma_2\gamma_1^{-1}}^{-1},  \\
\varpi_{i^\mfe}(\Tiii)&=\T_{i^\mfe}, & \varpi_{i^\mfe}(\Ti)&=\T_{\eta_1\eta_2^{-1}}^{-1}\varPhip\Tiii^{-1}\Ti\Tiii\varPhip^{-1}\T_{\eta_1\eta_2^{-1}}. 
\end{align*}
Furthermore, $\omega_i(\varThetap)=T_{\gamma_2}\Ti\T_{\gamma_2^{-1}}$ and $\varpi_{i^\mfe}(\varPhip)=T_{\eta_2}\Tiii\T_{\eta_2^{-1}}$.
\subsection{}\label{sec: mho-action} We note that the above formulas show that $\mho$ acts on the set of conjugacy classes of the generators associated to the nodes in a double affine Coxeter diagram. Therefore, we obtain a natural action of $\mho$ on the set of connected components of the diagram obtained from the corresponding double affine Coxeter diagram by erasing all multiple edges. This action will play a role in specifying the number possible independent parameters for the double affine Hecke algebras attached to arbitrary double reductive data.
\subsection{} The semidirect products $\mho\ltimes \B(\DDDot{X}_n)$ and  $\mho\ltimes \B(\fH{X}_n)$ are isomorphic to the corresponding $Y$-extended double affine Artin groups while the semidirect products $\mho^\mfe\ltimes \B(\DDDot{X}_n)$, $\mho^\mfe\ltimes \B(\fH{X}_n)$ and $\mho\ltimes(\mho^\mfe\ltimes \B(\DDDot{X}_n)_{\frac{1}{m}})$, $\mho\ltimes(\mho^\mfe\ltimes \B(\fH{X}_n)_{\frac{1}{m}})$ are isomorphic to the corresponding $X$-extended, and $X,Y$-extended double affine Artin groups, respectively. The groups $\mho$ and $\mho^\mfe$ generate inside the automorphism groups of $\B(\DDDot{X}_n)$ and $\B(\fH{X}_n)$ a copy of the group freely generated by $\mho$ and $\mho^\mfe$. 
\begin{Rem}\label{rem: omega-to-X}
We note that for $\varpi_i\in \mho^\mfe$, the elements $X_{\lambda_i^\vee}(T_{i^*}T_{w_{i^*}})^{-1}\in \Atilde(R)_X$ and $\varpi_i$ correspond through the isomorphism, while 
for $\omega_i\in \mho$, the elements $(T_{w_{i^*}^{-1}}T_{i^*})^{-1}Y_{-\lambda_i^\vee}\in \Atilde(R)_Y$ (if $r=1$) or $(T_{w_{i^*}^{-1}}T_{i^*})^{-1}Y_{-\lambda_i}\in \Atilde(R)_Y$ (if $r=2$) and $\omega_{i^*}$ are in correspondence.
\end{Rem}

For double affine Coxeter diagrams of type $\DDDot{X}_n$, we have   $\mfb\omega_i\mfb^{-1}=\omega_i$ and  $\mfa\omega_i\mfa^{-1}$ coincides with conjugation by $\varpi_i\T_{i^*}\T_{w_{i^*}}\omega_i$ inside $\mho\ltimes(\mho^\mfe\ltimes \B(\DDDot{X}_n)_{\frac{1}{m}})$. Recall that the group $\wXi(1)$ acts on $\B(\DDDot{X}_n)$, the action of its generators being given by  $\mfa$, $\mfb$, and $\mfe$. This action can be extended to an action (by conjugation) of $\wXi(1)$  on the group of inner automorphisms of $\mho\ltimes(\mho^\mfe\ltimes \B(\DDDot{X}_n)_{\frac{1}{m}})$, or equivalently on  $\mho\ltimes(\mho^\mfe\ltimes \B(\DDDot{X}_n)_{\frac{1}{m}})$ modulo its center, as follows
\begin{align*}
\mfe(\omega_i)&=\varpi_{i^*}, &\mfa(\omega_i)&=\varpi_i\T_{i^*}\T_{w_{i^*}}\omega_i, &\mfb(\omega_i)&=\omega_i,\\
\mfe(\varpi_i)&=\omega_{i^*},& \mfa(\varpi_i)&=\varpi_i,& \mfb(\varpi_i)&=\varpi_i\omega_{i^*}\T_i^{-1}\T_{w_i^{-1}}^{-1}.
\end{align*}

To describe the corresponding actions for the double affine Coxeter diagram of type $\fH{B}_n$ and $\fH{C}_n$ we will use the standard labelling of the nodes. For double affine Coxeter diagrams of type $\fH{B}_n$, $n\geq 2$ we have  $\mfb\omega_n\mfb^{-1}=\omega_n$ and $\mfa\omega_n\mfa^{-1}$ coincides with conjugation by $(\prod_{k=1}^n \T_{k-1}^{-1}\cdots\T_1^{-1}\varpi_1\varTheta\T_1^{-1}\cdots\T_{k-1}^{-1})\omega_n$ inside $\mho\ltimes(\mho^\mfe\ltimes \B(\fH{B}_n)_{\frac{1}{m}})$.  For double affine Coxeter diagrams of type $\fH{C}_n$, $n\geq 3$, we have  $\mfb\omega_1\mfb^{-1}=\omega_1$ and $\mfa\omega_1\mfa^{-1}$ coincides with conjugation by $\Tiii\varPhi\omega_1$ inside $\mho\ltimes(\mho^\mfe\ltimes \B(\fH{C}_n)_{\frac{1}{m}})$. 

In this case, recall that the group $\wGamma(2)$ acts on $\B(\fH{B}_n)$ and $\B(\fH{C}_n)$, the action of its generators being given by  $\mfa$ and $\mfb$. This action can be extended to an action (by conjugation) of $\wGamma(2)$  on the group of inner automorphisms of $\mho\ltimes(\mho^\mfe\ltimes \B(\fH{B}_n)_{\frac{1}{m}})$, or equivalently on  $\mho\ltimes(\mho^\mfe\ltimes \B(\fH{B}_n)_{\frac{1}{m}})$ modulo its center, as
\begin{align*}
\mfa(\omega_n)&=(\prod_{k=1}^n \T_{k-1}^{-1}\cdots\T_1^{-1}\varpi_1\varTheta\T_1^{-1}\cdots\T_{k-1}^{-1})\omega_n, &\mfb(\omega_n)&=\omega_n,\\
 \mfa(\varpi_1)&=\varpi_1,& \mfb(\varpi_1)&=\varTheta\Ti\varpi_1,
\end{align*}
and on the group of inner automorphisms of $\mho\ltimes(\mho^\mfe\ltimes \B(\fH{C}_n)_{\frac{1}{m}})$, or equivalently on  $\mho\ltimes(\mho^\mfe\ltimes \B(\fH{C}_n)_{\frac{1}{m}})$ modulo its center, as
\begin{align*}
 \mfa(\omega_1)&=\Tiii\varPhi\omega_1,& \mfb(\omega_1)&=\omega_1,\\
\mfa(\varpi_n)&=\varpi_n, &\mfb(\varpi_n)&=(\prod_{k=1}^n \T_{k-1}\cdots\T_1\varPhi\omega_1\T_1\cdots\T_{k-1})\varpi_n.
\end{align*}

Denote by $\overline{\mfa(\omega_i)}$ the expressions for $\mfa(\omega_i)$ but seen now as elements of $\mho\ltimes(\mho^\mfe\ltimes \B(\DDDot{X}_n)_{\frac{1}{m}})$ and $\mho\ltimes(\mho^\mfe\ltimes \B(\fH{X}_n)_{\frac{1}{m}})$ (which in the Bernstein type presentations are expressed as $X_{\l_i^\vee}\omega_i$). It can be computed that 
$$
\overline{\mfa(\omega_i)}^{\ord(\omega_i)}=\C^{{\ord(\omega_i)(\l_i^\vee,\l_i^\vee)}/{2}},
$$
where $\ord(\omega_i)$ is the order of $\omega_i$. Similarly facts are true for $\mfb(\varpi_i)$. 

Define $\B(\DDDot{X}_n)^e$ and $\B(\fH{X}_n)^e$ as the smallest trivial central extension of $\mho\ltimes(\mho^\mfe\ltimes \B(\DDDot{X}_n)_{\frac{1}{m}})$ and $\mho\ltimes(\mho^\mfe\ltimes \B(\fH{X}_n)_{\frac{1}{m}})$ that contains the elements $\C^{(\l_i^\vee,\l_i^\vee)/{2}}$ for all $\omega_i\in\mho$. These groups are isomorphic to the corresponding fully extended double affine Artin groups. We can now prove the following the analogue of Theorem \ref{thm: autoall} for $\B(\DDDot{X}_n)^e$ and $\B(\fH{X}_n)^e$.
\begin{Thm}\label{thm: autoall-ext}
The group $\wXi(r)$ acts faithfully by (anti)morphisms on $\B(\DDDot{X}_n)^e$ (for $r=1$), $\B(\fH{X}_n)^e$ (for $r=2$, $X_n$ double laced), and $\B(\fH{G}_2)^e$ (for $r=3$).
\end{Thm}
\begin{proof}
The above formulas for $\mfa(\omega_i)$ corrected by a factor of $\C^{-(\l_i^\vee,\l_i^\vee)/{2}}$ (and similarly for $\mfb(\varpi_i)$) give a well-defined action of $\wXi(r)$ as automorphisms of $\B(\DDDot{X}_n)^e$ and $\B(\fH{X}_n)^e$. 
\end{proof}
For the extended double affine Artin groups associated to untwisted affine Dynkin diagrams this result was first proved by Cherednik \cite{CheMac}*{Theorem 4.3}. For  the extended double affine Artin groups associated to twisted affine Dynkin diagrams (the only case not already part of Theorem  \ref{thm: autoall} being $\fH{X}_n=\fH{B}_n/\fH{C}_n$) the result is new.

\subsection{} By \eqref{eq: intermediate} the double affine Artin groups attached to semisimple data are precisely the intermediate subgroups between double affine Artin groups and the corresponding extended double affine Artin groups. We are know able to show that in fact, all double affine Artin groups attached to semisimple data have a rich group of automorphisms. 
\begin{Thm}\label{thm: autoall-ss}
There exist a group $\widetilde{K}(r)$ of finite index in $\wGamma(r)$ that stabilizes any intermediate subgroup $\B(\DDDot{X}_n)\leq \B(\DDDot{X}_n)^{s} \leq \B(\DDDot{X}_n)^e$ (for $r=1$), $\B(\fH{X}_n)\leq \B(\fH{X}_n)^{s} \leq \B(\fH{X}_n)^e$ (for $r=2$), and $\B(\fH{G}_2)\leq \B(\fH{G}_2)^{s} \leq \B(\fH{G}_2)^e$ (for $r=3$).  Furthermore, the action of $\widetilde{K}(r)$ descends to an outer action of a congruence subgroup $K(r)$ of $\Gamma_1(r)$ of level at most $|\mho|$.
\end{Thm}
\begin{proof}
The subgroup of $\B(\DDDot{X}_n)^e$ or $\B(\fH{X}_n)^e$ generated by its center and by $\B(\DDDot{X}_n)$ or $\B(\fH{X}_n)$ is normal. The corresponding quotient group is a group generated by $\mho$ and $\mho^\mfe$. To fully determine the group structure of the quotient let $\omega_i\in \mho$ and $\varpi_j\in \mho^\mfe$. From \eqref{eq: omega-X-rels} (for $\omega=\omega_i$ and  $\b=\lambda^j$) and Remark \ref{rem: omega-to-X} it follows that the commutator of $\omega_i$ and $\varpi_j$ is trivial in the quotient group. Therefore, the quotient group is a finite abelian group isomorphic to the direct product $\mho\times \mho^\mfe$. The action of $\wGamma(r)$ stabilizes the normal subgroup and therefore descends to an action by automorphism on the quotient group $\mho\times\mho^\mfe$.

For the Coxeter diagrams of type $\DDDot{X}_n$, the action of $\wGamma(1)$ on $\mho\times \mho^\mfe$ is described as follows 
\begin{align*}
\mfa(\omega_i)&=\varpi_i\omega_i, &\mfb(\omega_i)&=\omega_i,\\
\mfa(\varpi_i)&=\varpi_i,& \mfb(\varpi_i)&=\varpi_i\omega_{i^*},
\end{align*}
for $\omega_i\in \mho$, $\varpi_i\in\mho^\mfe$.
The element $(\mfa\mfb)^6$ acts as identity on $\mho\times \mho^\mfe$ and therefore the action of $\wGamma(1)$ factors through $$\Gamma_1(1)\cong\wGamma(1)/\<(\mfa\mfb)^6\>.$$ 
This action of $\Gamma_1(1)$ on  $\mho\times \mho^\mfe$ as group automorphisms is described as follows
\begin{equation}\label{eq: cong-1}
 \begin{bmatrix}a&b\\ c&d \end{bmatrix}(\omega_i\varpi_j)=\omega_i^d\omega_j^{-c}\varpi_i^{-b}\varpi_j^a.
\end{equation}
Denote by $K(1)$ the kernel of the $\Gamma_1(1)$ action which, since $\mho\times\mho^\mfe$ is a finite set, has finite index in $\Gamma_1(1)$. As it is clear from \eqref{eq: cong-1} the principal congruence subgroup $\Gamma(|\mho|)$ acts trivially on $\mho\times\mho^\mfe$ and therefore $K(1)$ is a congruence subgroup of $\SL(2,\Z)$. The kernel $\widetilde{K}(1)$ of the $\wGamma(1)$ action is, of course, the inverse image of $K(1)$. The group $\widetilde{K}(1)$ thus stabilizes each equivalence class modulo $\B(\DDDot{X}_n)$ and in consequence any intermediate subgroup $\B(\DDDot{X}_n)\leq \B(\DDDot{X}_n)^{s} \leq \B(\DDDot{X}_n)^e$.

For the  Coxeter diagram of type $\fH{B}_n$, the action of $\wGamma(2)$ on $\mho\times \mho^\mfe$ is described as follows 
\begin{align*}
\mfa(\omega_n)&=\varpi_1^n\omega_n, &\mfb(\omega_n)&=\omega_n,\\
 \mfa(\varpi_1)&=\varpi_1,& \mfb(\varpi_1)&=\varpi_1.
\end{align*}
The element $(\mfa\mfb)^2$ acts as identity on $\mho\times \mho^\mfe$ and therefore the action of $\wGamma(2)$ factors through $$\Gamma_1(2)\cong\wGamma(2)/\<(\mfa\mfb)^4\>.$$ The action of $\Gamma_1(2)$ on  $\mho\times \mho^\mfe$ as group automorphisms is described as follows
\begin{equation}\label{eq: cong-2}
 \begin{bmatrix}a&b\\ c&d \end{bmatrix}(\omega_n^i\varpi_1^j)=\omega_n^i\varpi_1^{-ibn}\varpi_1^j.
\end{equation}
Denote by $K(2)$ the kernel of the $\Gamma_1(2)$ action. This is a subgroup of finite index in $\Gamma_1(2)$. As it is clear from \eqref{eq: cong-2} the principal congruence subgroup $\Gamma(|\mho|)$ acts trivially on $\mho\times\mho^\mfe$ and therefore $K(2)$ is a congruence subgroup of $\SL(2,\Z)$. The kernel $\widetilde{K}(2)$ of the $\wGamma(2)$ action is, of course, the inverse image of $K(2)$. The group $\widetilde{K}(2)$ thus stabilizes each equivalence class modulo $\B(\fH{B}_n)$ and in consequence any intermediate subgroup $\B(\fH{B}_n)\leq \B(\fH{B}_n)^{s} \leq \B(\fH{B}_n)^e$. The treatment of the  Coxeter diagram of type $\fH{C}_n$ is completely analogous.
\end{proof}
\begin{Rem}
If for every $i$ the elements $\omega_i$ and $\varpi_i$ are either both in or both not in $\B(\DDDot{X}_n)^{s}$ then the formulas for the action of $\mfa$ and $\mfb$ show that the entire group $\wGamma(r)$ stabilizes $\B(\DDDot{X}_n)^{s}$. Also, for even $n$, the entire group $\wGamma(2)$ stabilizes $\B(\fH{X}_n)^{s}$. For each proper intermediate subgroup between the double affine Artin group and the extended affine Artin group,  it would be interesting to find the maximal subgroup of $\wGamma(r)$ that is stabilizing it. 
\end{Rem}

\subsection{} We illustrate the previous constructions for the case of the double affine Coxeter diagram of type $\DDDot{A}_1$. This should facilitate the comparison with the existing constructions of the corresponding fully extended double affine Artin group as in \cite{CheDou-1}*{\S0.4.3, \S2.5.3, \S2.7.1}, \cite{MacAff}*{\S6.1,\S6.4}.
\begin{Exp}
The group $\B(\DDDot{A}_1)$ is generated by four elements $\varTheta_{01}, \varTheta_{02}, \varTheta_{03}, \varTheta$ that are only required to satisfy the relation \eqref{centralrel-1}. The group $\mho$ has order two and we denote its non-trivial element by $\omega$. The actions of $\omega$ and $\varpi=\mfe\omega\mfe$ are as follows
 \begin{align*}
\omega(\varTheta)&=\varTheta_{01}, & \omega(\varTheta_{01})&=\varTheta, & \omega(\varTheta_{02})&=\varTheta^{-1}\varTheta_{03}\varTheta, & \omega(\varTheta_{03})&=\varTheta_{01}\varTheta_{02}\varTheta_{01}^{-1},\\ 
\varpi(\varTheta)&=\varTheta_{03}, & \varpi(\varTheta_{01})&=\varTheta_{03}^{-1}\varTheta_{02}\varTheta_{03}, & \varpi(\varTheta_{02})&=\varTheta\varTheta_{01}\varTheta^{-1}, & \varpi(\varTheta_{03})&=\varTheta.
\end{align*}
We denote $Y^{-2}=\varTheta\varTheta_{01}$ and $X^2=\varTheta_{03}\varTheta$. The group $\B(\DDDot{A}_1)$ is generated by $\varTheta, Y^{-2}, X^2$ and we have
\begin{equation*}
\omega X^2 \omega=\C X^{-2}\quad\text{and}\quad \varpi Y^{-2} \varpi=\C Y^{2}.
\end{equation*}
The elements $Y^{-1}=\varTheta\omega\in \mho\ltimes \B(\DDDot{A}_1)$ and $X=\varpi\varTheta\in \mfe\mho\mfe\ltimes \B(\DDDot{A}_1)$ square to $Y^{-2}$ and, respectively, $X^2$. Therefore, $\mho\ltimes \B(\DDDot{A}_1)$ (the $Y$-extended $\B(\DDDot{A}_1)$)  is generated by $\varTheta, Y^{-1}, X^2$, and $\mho^\mfe\ltimes \B(\DDDot{A}_1)$ (the $X$-extended $\B(\DDDot{A}_1)$)  is generated by $\varTheta, Y^{-2}, X$. Furthermore,
\begin{align*}
\varTheta Y\varTheta &=Y^{-1}, & \varTheta_{03} Y^{-1}\varTheta_{03} &=\C^{-1}Y,\\
\varTheta X^{-1}\varTheta &=X, & \varTheta_{01} X\varTheta_{01} &=\C^{-1}X^{-1}.
\end{align*}
The group $\mho\ltimes(\mho^\mfe\ltimes \B(\DDDot{A}_1)_{\frac{1}{2}})$ is generated by $\varTheta, Y^{-1}, X$ or, equivalently, by $\varTheta, \omega, \varpi$. Short computations show that
\begin{equation*}
\omega X\omega X=\C^{\frac{1}{2}}=\varpi Y^{-1}\varpi Y^{-1}.
\end{equation*}

The actions of $\mfa\omega\mfa^{-1}$  and $\mfb\omega\mfb^{-1}$ on $\B(\DDDot{A}_1)$ are as follows
 \begin{align*}
\mfa\omega\mfa^{-1}(\varTheta)&=\varTheta_{02}, & \mfa\omega\mfa^{-1}(\varTheta_{01})&=\varTheta_{03}, \\ \mfa\omega\mfa^{-1}(\varTheta_{02})&=\varTheta, & \mfa\omega\mfa^{-1}(\varTheta_{03})&=\varTheta_{01},\\
\mfb\omega\mfb^{-1}(\varTheta)&=\varTheta_{01}, & \mfb\omega\mfb^{-1}(\varTheta_{01})&=\varTheta, \\ \mfb\omega\mfb^{-1}(\varTheta_{02})&=\varTheta^{-1}\varTheta_{03}\varTheta, & \mfb\omega\mfb^{-1}(\varTheta_{03})&=\varTheta_{01}\varTheta_{02}\varTheta_{01}^{-1},
\end{align*}
and they can be also achieved by conjugation with $\varpi\varTheta\omega$ and respectively $\omega$ inside $\mho\ltimes(\mho^\mfe\ltimes \B(\DDDot{A}_1)_{\frac{1}{2}})$. Therefore, the group $\wXi(1)$ acts on  $\mho\ltimes(\mho^\mfe\ltimes \B(\DDDot{A}_1)_{\frac{1}{2}})$ modulo its center as follows
\begin{align*}
\mfe(\omega)&=\varpi, &\mfa(\omega)&=\varpi\varTheta\omega, &\mfb(\omega)&=\omega,\\
\mfe(\varpi)&=\omega,& \mfa(\varpi)&=\varpi,& \mfb(\varpi)&=\varpi\omega\varTheta^{-1}.
\end{align*}
Now, inside $\mho\ltimes(\mho^\mfe\ltimes \B(\DDDot{A}_1)_{\frac{1}{2}})$ we have $(\varpi\varTheta\omega)^2=\C^{\frac{1}{2}}$ and $(\varpi\omega\Theta^{-1})^2=\C^{-\frac{1}{2}}$. If we consider $\B(\DDDot{A}_1)^e$ defined as  $\mho\ltimes(\mho^\mfe\ltimes \B(\DDDot{A}_1)_{\frac{1}{4}})$ (the fully extended $\B(\DDDot{A}_1)$) we can set
\begin{align*}
\mfe(\omega)&=\varpi, &\mfa(\omega)&=\C^{-\frac{1}{4}}\varpi\varTheta\omega, &\mfb(\omega)&=\omega,\\
\mfe(\varpi)&=\omega,& \mfa(\varpi)&=\varpi,& \mfb(\varpi)&=\C^{\frac{1}{4}}\varpi\omega\varTheta^{-1}.
\end{align*}
which is an action of $\wXi(1)$ on $\B(\DDDot{A}_1)^e$.
\end{Exp}
\subsection{} We close this section by pointing out that, in fact, the groups $\B(\DDDot{X}_n)^e$ and $\B(\fH{X}_n)^e$ should allow a presentation by generators and relations that is akin to the Coxeter type presentation of $\B(\DDDot{X}_n)$ and $\B(\fH{X}_n)$. 

Of course, we only need to specify the presentations for the non-trivial extensions. Therefore, among the diagrams of type $\DDDot{X}_n$ we don't need to consider $\DDDot{E}_8$, $\DDDot{F}_4$ and $\DDDot{G}_2$, and among the diagram of type $\fH{X}_n$ we only need to consider the diagrams of type $\fH{B}_n$ and $\fH{C}_n$ (for which the groups will be isomorphic). In what follows we adhere to this convention without further warning.

Again, we will make reference to the double affine diagrams in Figure \ref{dddot-diagrams} and Figure \ref{ddot-diagrams}. The generators associated to the finite nodes are denoted by $\T_1,\dots,\T_n$, the finite nodes being labelled in the conventional fashion (see \S\ref{sec: cl}). For each affine subdiagram of a double affine Coxeter diagram we consider its group of outer automorphisms. Each non-identity element in the group of outer automorphisms will be labeled by the relevant affine node and the (finite) label of its image. 

More precisely, for $\DDDot{X}_n$, the non-identity elements of the outer automorphism groups will be denoted by $\omega_{01,i}$, $\omega_{02,i}$, $\omega_{03,i}$, for the affine nodes labelled by $\varTheta_{01}$, $\varTheta_{02}$, $\varTheta_{03}$, respectively.  The corresponding finite groups will be denoted by $\mho_{01}$, $\mho_{02}$, $\mho_{03}$, respectively.  In this case, the set of finite labels for the image of the affine nodes is the same. We will denote this set by $I$. Furthermore, we denote by $i^*$ the (finite) node for which $\omega_{01, i^*}=\omega_{01,i}^{-1}$. Also recall the definition and notation in \S\ref{sec: wi-notation} for the elements $\T_{w_i}$ of $B(X_n)$.

For $\fH{B}_n$ and $\fH{C}_n$, the unique non-trivial elements of the outer automorphism groups will be denoted by $\omega$ (corresponding to the affine node labeled by $\Ti$) and $\varpi$ (corresponding to the affine node labeled by $\Tiii$). For both these diagrams, $\mho$ denotes the group of outer automorphisms associated to the affine node labelled by $\Ti$ and  $\bar{\mho}$  denotes the group of outer automorphisms associated to affine node labelled by $\Tiii$.

We are now ready to indicate an alternative conjectural presentation of  the extended double affine Artin group associated to each double affine Coxeter diagram. We start with the double affine Coxeter diagrams of type $\DDDot{X}_n$.

\begin{Cnj}
The group $\B(\DDDot{X}_n)^e$ is the group generated by $B(X_n)$ and the finite groups $\mho_{01}$, $\mho_{02}$, $\mho_{03}$ subject to the the following relations:
\begin{enumerate}[label={\alph*)}]
\item For $1\leq j\leq 3$, $i\in I$, $1\leq k\leq n$, $k\neq i^*$ we have
$$\omega_{0j, i}\T_k\omega_{0j, i}^{-1}=\T_{\omega_{0j, i}(k)}.$$
\item For $i\in I$ the elements
$$
\mathcal{D}_i=\omega_{01,i}\omega_{02,i^*}\omega_{03,i}\T_{i^*}\T_{w_{i^*}}
$$
are central.
\end{enumerate}
\end{Cnj}
For the double affine Coxeter diagram of type $\fH{X}_n$  we use the standard labelling of the finite nodes.
\begin{Cnj}
The groups and  $\B(\fH{X}_n)^e$  is the group generated by $B(X_n)$ and the finite groups $\mho$, $\bar\mho$ subject to the the following relations:
\begin{enumerate}[label={\alph*)}]
\item For $1\leq i, j\leq n$ such that $\omega(i), \varpi(j)$ are finite nodes, we have
$$\omega\T_i\omega=\T_{\omega(i)}, \quad \varpi\T_j\varpi=\T_{\varpi(j)}.$$
\item The element
$$
\mathcal{D}=\varpi\T_1\cdots\T_n\omega\varpi\T_n\cdots\T_1\omega
$$
is central.
\end{enumerate}
\end{Cnj}

\section{Hecke algebras}\label{sec: hecke}

\subsection{} The Hecke algebra associated to a double affine Artin group $\Atilde(R)$, called double affine Hecke algebra, is defined in the literature as the quotient of the $\F$-group algebra of $\Atilde(R)$ (with respect to an appropriate field of parameters) by the quadratic relations 
$$
T_i - T_i^{-1} = t_i^{\frac{1}{2}} - t_i^{-\frac{1}{2}}
$$
corresponding to the generators of $\A(R)$. It is customary to consider the central element $X_\d$ as a parameter in $\F$; to emphasize this, we denote
$$
X_\d=q^{-1}.
$$

The generators are indexed by the orbits of $W$ on $R$: $t_i=t_j$ if $W(\a_i)=W(\a_j)$. In particular, the maximal number of possible parameters coincides with the number of $1$-connected components of the affine Dynkin diagram of $R$, which means that there are, depending on the case, 1, 2, or 3 possible independent Hecke parameters. If $R$ is allowed to be nonreduced then the number of independent parameters is specified in Table \ref{table: generic-parameters-nonreduced}.

\begin{table}[ht]
\caption{Maximal number of Hecke parameters associated to non-reduced affine root systems}\label{table: generic-parameters-nonreduced}
\begin{tabular}{l l l }
$(BC_n, C_n), n\geq 1$ && 4  \\
$(C_n^\vee, BC_n), n\geq 1$ && 4   \\
$(B_n, B_n^\vee), n\geq 3$ && 3  \\
$(C_n^\vee, C_n), n\geq 1$ && 5 \\
$(C_2, C_2^\vee)$ && 4 
\end{tabular}
\end{table}
We emphasize that the double affine Hecke algebra constructions in the literature do not directly refer to the affine root systems $A_{2n}^{(2)}$, $(BC_n, C_n)$, $(C_n^\vee, BC_n)$. This is probably a consequence of the fact that these root systems can be realized as root sub-systems of $(C_n^\vee, C_n)$ and their Macdonald theory  can be obtained by appropriately specializing the parameters in the Macdonald theory for $(C_n^\vee, C_n)$. As it is clear from Proposition \ref{comparison} and \S\ref{sec: nonreduced}, the associated double affine Artin groups, and consequently,  the associated double affine Hecke algebras bear a similar relationship.

The double affine Artin groups associated to $C_n^{(1)}$ and $(C_n^\vee, C_n)$ can be identified as explained in \S\ref{sec: nonreduced}. The double affine Hecke algebra associated to $(C_n^\vee, C_n)$, when defined correctly \cite{SahNon}, depends on five independent Hecke parameters (four if $n=1$) as opposed to three parameters for the double affine Hecke algebra associated to  $C_n^{(1)}$. Some quadratic relations in the double affine Hecke algebra of type $(C_n^\vee, C_n)$ are imposed on elements that are not in the group $\A(C_n^{(1)})$. Therefore, it is somewhat delicate to correctly define the double affine Hecke algebras associated to non-reduced root systems (or to $A_{2n}^{(2)}$) and to describe the relationship between them.


\subsection{}  We take advantage of our Coxeter type presentation for the double affine Artin groups to give a uniform construction of double affine Hecke algebras. We do this by associating to each double affine Coxeter diagram a Hecke algebra $\bH(\DDDot{X}_n)$, $\bH(\fH{X}_n)$, or  $\bH(\DDot{C}_n)$  that depends on the largest possible number of independent parameters.  We call these Hecke algebras generic Hecke algebras. The field $\F$ contains one Hecke parameter for each node in the double Coxeter diagram. Two Hecke parameters are equal if the corresponding nodes are in the same 1-connected component of the double affine Coxeter diagram and they are distinct otherwise. Recall that the number of independent Hecke parameters for each double affine Coxeter diagram is recorded in Table \ref{table: generic-parameters}. Also, we will consider the central element $\C$ as an element of $\F$ and to emphasize this we denote $\C=q^{-1}$, with $q$ a formal parameter. We assume that the field $\F$ contains the square roots of all these parameters.

\begin{Def}\label{def: hecke}
The generic Hecke algebra $\bH(\DDDot{X}_n)$, $\bH(\fH{X}_n)$, and  $\bH(\DDot{C}_n)$  is defined as the quotient of the $\F$-group algebra of $\B(\DDDot{X}_n)$, $\B(\fH{X}_n)$, and $\B(\DDot{C}_n)$  by the ideal generated by the quadratic relations for the generators associated to the nodes of $\DDDot{X}_n$, $\fH{X}_n$, and  $\DDot{C}_n$, respectively.  Equivalently, $\bH(\DDDot{X}_n)$, $\bH(\fH{X}_n)$, and $\bH(\DDot{C}_n)$  is defined, respectively, as the quotient of the Coxeter Hecke $\F$-algebra $H(\DDDot{X}_n)$, $H(\fH{X}_n)$, and  $H(\DDot{C}_n)$  by the ideal generated by the relations \eqref{centralrel-1}-\eqref{ellbraid} (for $\DDDot{X}_n$), \eqref{centralrel} (for $\fH{X}_n$), and \eqref{centralrelbis}-\eqref{ellbraidbis} (for $\DDot{C}_n$). 
\end{Def}

\begin{Rem}
For $\DDot{C}_n$  we can also consider the algebra  $\bH(\DDot{C}_n)^c$ defined as the quotient of the $\F$-group algebra of $\B(\DDot{C}_n)^c$ by the ideal generated by the quadratic relations for the generators associated to the nodes of $\DDot{C}_n$. Since we assume that the field $\F$ contains the square roots of the parameters,  $\bH(\DDot{C}_n)$ and $\bH(\DDot{C}_n)^c$ and are in fact isomorphic. Proposition \ref{comparisonbis} lets us compare these algebras with $\bH(\DDDot{C}_n)$. The second morphism in Proposition \ref{comparisonbis} is more convenient to work with. Using it we obtain the following result.
\end{Rem}
\begin{Thm}
The generic Hecke algebra $\bH(\DDot{C}_n)$ is isomorphic to the algebra obtained from $\bH(\DDDot{C}_n)$ by specializing the parameter corresponding to $\varTheta_{02}$ to 1.
\end{Thm}

\begin{Rem}\label{rem: conj03}
The generators of $\B(\DDDot{X}_n)$, $\B(\fH{X}_n)$, $\B(\DDot{C}_n)$  that correspond to nodes in the same 1-connected component of the double affine diagram are conjugate,  and the parameters associated to these nodes in the definition of $\bH(\DDDot{X}_n)$, $\bH(\fH{X}_n)$,  $\bH(\DDot{C}_n)$ are equal. In particular, we see that the quadratic relation corresponding to $\varTheta_{02}$, $\varTheta_{03}$ and, respectively,  $\Tiii$ is in the ideal generated by the quadratic relations for $T_i$, $1\leq i\leq n$, unless $\th_{02}$, $\th_{03}$, and respectively $\phi_0$, is different from all $t_i$, $1\leq i\leq n$. This happens precisely if the double affine Coxeter diagram is of type $\DDDot{C}_n$, $n\geq 1$, $\fH{B}_2$, $\fH{C}_n$, $n\geq 2$, or $\DDot{C}_n$, $n\geq 1$.
\end{Rem}

The Hecke algebras associated to  nonreduced irreducible affine root systems $R$ are defined as the algebras obtained from the appropriate double Coxeter diagram (see \S\ref{sec: nonreduced} and \eqref{DCvsDA}) by specializing the parameters according to their equivalence with respect to the action of the relevant affine Weyl group on $R$. We report in Table \ref{table: parameters} the specialization of parameters for all nonreduced irreducible affine root systems and for the reduced root systems for which  a specialization is necessary. The parameters are denoted  by the same symbols as the corresponding generators, but using lowercase letters, with the nodes of $X_n$ being labelled according to our convention in \S\ref{sec: cl}.

\begin{table}[ht]
\caption{Parameter specialization for double affine Hecke algebras}
\label{table: parameters}
\begin{tabular}{ l  l  l }
Hecke algebra &Parameter specialization & Affine root system   \\
  \hline
$\bH(\DDDot{A}_1)$ & $\th_{01}=\th_{02}=\th_{03}=t_1$ &   $A_1^{(1)}$ \\ 
 $\bH(\DDDot{C}_n)$ & none &       $(C_n^\vee, C_n), n\geq 1$ \\ 
\ \ " &$\th_{01}=\th_{02}$  &    $(BC_n, C_n), n\geq 1$  \\ 
\ \ " &$ \th_{02}=1$   &    $(C_n^\vee, BC_n), n\geq 1$ \\ 
\ \ " &$ \th_{02}=1, \th_{03}=t_n$  &   $A_{2n}^{(2)}$, $n\geq 1$  \\ 
\ \ " &$ \th_{03}=1, \th_{01}=\th_{02}$  &   $A_{2n}^{(2)\vee}$, $n\geq 1$  \\ 
\ \ "  & $\th_{01}=\th_{02}=\th_{03}$ &  $C_n^{(1)}, n\geq 2$  \\ \hline
$\bH(\fH{C}_2)$ & none  &    $(C_2, C_2^\vee)$  \\ 
\ \ " & $\th_0=t_1$  &   $A_{3}^{(2)}$ \\ 
$\bH(\fH{B}_n)$ & $\th_{0}=t_n$  &    $D_{n+1}^{(2)}$, $n\geq 2$ \\ 
$\bH(\fH{C}_n)$ & none &   $(B_n, B_n^\vee), n\geq 3$ \\ 
\ \ " & $\phi_{0}=t_n$  &  $A_{2n-1}^{(2)}$, $n\geq 3$  \\ \hline
$\bH(\DDot{C}_n)$ & none   &    $(C_n^\vee, BC_n), n\geq 1$ \\ 
\ \ " & $\phi_{0}=t_n$  &   $A_{2n}^{(2)}$, $n\geq 1$ \\ 
\end{tabular}
\end{table}
\subsection{} The action of the groups that appear in the statement of Theorem \ref{thm: autoall} descend to faithful actions as morphisms or anti-morphisms on the relevant algebra $\bH(\DDDot{X}_n)$, $\bH({\fH{X}_n})$, or $\bH(\DDot{C}_n)$. These actions are trivial on $\F$ except for $\bH(\DDDot{C}_n)$ and $\bH(\DDot{C}_n)$ where $\wXi(1)$ and, respectively, $\wXi(2)^\prime$ act by permuting the parameters. This affects the descent of these actions to the  Hecke algebras of type $A_{2n}^{(2)}$ and $(BC_n, C_n)$. For $(BC_n, C_n)$ only the subgroup of $\wXi(1)$ that acts on parameters by fixing $\th_{03}$ (which is isomorphic to $\wGamma(2)^\prime$) will act on the Hecke algebra and for  $A_{2n}^{(2)}$ only the subgroup of $\wXi(2)^\prime$ that fixes all parameters (the pure braid group $P_3$) will act on the Hecke algebra. 

\subsection{} The extended generic Hecke algebras $\bH(\DDDot{X}_n)^e$ and $\bH(\fH{X}_n)^e$ are defined precisely in the same fashion but as quotients of the extended groups $\B(\DDDot{X}_n)^e$ and $\B(\fH{X}_n)^e$. It is important to stress that for any non-trivial extension the number of possible independent parameters for extended double affine Hecke algebra decreases. More precisely, the number of possible independent parameters is not the number of connected components of the diagram obtained from the corresponding double affine Coxeter diagram by erasing all multiple edges but rather the number of orbits of the group generated by $\mho$ and $\mho^\mfe$ on the set of connected components (see \S\ref{sec: mho-action}). The same statement (with $\mho$ and $\mho^\mfe$ replaced by the appropriate subgroups) is true about the double affine Hecke algebras attached to semisimple data, which, by \eqref{eq: intermediate}, are precisely the intermediate subalgebras $\bH(\DDDot{X}_n)\leq \bH(\DDDot{X}_n)^{s} \leq \bH(\DDDot{X}_n)^e$ and $\bH(\fH{X}_n)\leq \bH(\fH{X}_n)^{s} \leq \bH(\fH{X}_n)^e$.  Table \ref{table: extended-generic-parameters} lists the maximal number of distinct parameters in (fully) extended Hecke algebras.

\begin{table}[ht]
\caption{Maximal number of parameters in extended Hecke algebras}\label{table: extended-generic-parameters}
\begin{tabular}{l l l l l l l}
$\DDDot{A}_n, n\geq 1$ && 1  && $\DDot{A}_1$   && 3\\
$\DDDot{B}_n, n\geq 3$ && 2  && $\DDot{C}_n, n\geq 2$   && 4\\
$\DDDot{C}_n, n\geq 2$ && 2  &&  &&\\
$\DDDot{D}_n, n\geq 4$ && 1  && $\fH{B}_n/\fH{C}_n, n\geq 2$   && 2\\
$\DDDot{E}_n, n=6,7,8$ && 1  && $\fH{F}_4$   && 2\\
$\DDDot{F}_4$  && 2 && $\fH{G}_2$  && 2 \\
$\DDDot{G}_2$ && 2 && &&
\end{tabular}
\end{table}

\subsection{}\label{sec: reductiveHecke} Let $\left((X,R_1,R_1^\vee)~ ; ~ (Y,R_2,R_2^\vee)\right)$ be a double reductive datum and consider the underlying semisimple and central data, denoted by $$\left((X_s,R_1,R_1^\vee)~ ; ~ (Y_s,R_2,R_2^\vee)\right)\quad \text{and}\quad \left((X_c,0,0)~ ; ~ (Y_c,0,0)\right),$$ respectively. In analogy with Proposition \ref{prop: reductive}, the generic double affine Hecke algebra $\bH(X, Y, W_0)$ that corresponds to the double reductive datum $\left((X,R_1,R_1^\vee)~ ; ~ (Y,R_2,R_2^\vee)\right)$ can be defined as 
\begin{equation}
\bH(X_s, Y_s, W_0)~\overline{\times} ~\bH(X_c, Y_c, \{1\}),
\end{equation}
where $\bH(X_c, Y_c, \{1\})$ is defined as the $\F$-group algebra of $\A(X_c, Y_c, \{1\})$.

It is perhaps useful to specify the relationship between the generic Hecke algebras associated to arbitrary reductive group data and the Hecke algebras defined in \cite{StoKoo-2}. The latter are by data of the type $D=(R_0, \Delta_0, \bullet, \Lambda, \Lambda^d)$, with $R_0$ an irreducible reduced finite root system, $\Delta_0$ an ordered basis, $\bullet\in\{u,t\}$ ($u$ standing for \emph{untwisted}, $t$ standing for \emph{twisted}), and lattices $\Lambda$, $\Lambda^d$ satisfying appropriate compatibility conditions. The double affine Hecke algebra associated to $D$ is denoted by $\bH(D)$. If $\bullet=t$ then $\bH(D)$ is precisely the generic Hecke algebra associated to the double datum $\left((\Lambda,R_0^\vee,R_0)~ ; ~ (\Lambda^d,R_0^\vee,R_0)\right)$. If $\bullet=u$ then $\bH(D)$ is precisely the generic Hecke algebra associated to the double datum $\left((\Lambda,R_0^\vee,R_0)~ ; ~ (\Lambda^d,R_0,R_0^\vee)\right)$. In other words, if $\bullet=t$ and $R_0$ is of type $X_n^\vee$ then $\bH(D)$ is an intermediate double affine Hecke algebra attached to the diagram of type $\DDDot{X}_n$, and  if $\bullet=u$ and $R_0\not\cong R_0^\vee$ is of type $X_n^\vee$ then $\bH(D)$ is an intermediate double affine Hecke algebra attached to the diagram of type $\fH{X}_n$.

\subsection{} 
By employing Theorem \ref{thm: autoall-ext} instead of Theorem \ref{thm: autoall} we obtain faithful actions of the groups in the statement of Theorem \ref{thm: autoall-ext}  as morphisms or anti-morphisms on the relevant extended Hecke algebra. Furthermore, from Theorem \ref{thm: autoall-ss} we obtain that any double affine Hecke algebra attached to double reductive group data has a rich group of automorphisms that descends to an outer action of a congruence subgroup as described in the statement of Theorem  \ref{thm: autoall-ss}.

\subsection{}

Let us indicate how Definition \ref{def: hecke} leads to the conventional definition of double affine Hecke algebras. Since we have explicit isomorphisms between the groups $\B(\DDDot{X}_n)$, $\B(\DDot{C}_n)$, $\B(\fH{X}_n)$ and the double affine braid groups we can directly investigate the effect of imposing the relevant quadratic relations on the $\F$-group algebra of a double affine Artin group $\Atilde(R)$. We call the resulting quotient  the generic (adjoint) double affine Hecke algebra associated to $R$, and we denote it by $\bH(R)$. Recall that we consider the central element $X_\d$ as a parameter in $\F$ and use the notation
$$
X_\d=q^{-1}.
$$
The (adjoint) double affine Hecke algebra associated to the irreducible affine root system (reduced or nonreduced), denoted by $\Htilde(R)$, is defined as the algebra obtained from the generic double affine Hecke algebra $\bH(R)$  by specializing the parameters according to their equivalence with respect to the action of the affine Weyl group on $R$. The precise specialization of parameters, whenever necessary, is recorded in Table \ref{table: parameters}. 

Similarly, we denote by $\mathcal{H}(R)$ and $\mathcal{H}(\Rring)$ the quotient of the $\F$-group algebra of $\A(R)$ and, respectively, $\A(\Rring)$ by the ideal generated by the quadratic relations for the generators corresponding to the nodes of the Dynkin diagram of $R$ and $\Rring$, respectively. The number of independent Hecke parameters for $\mathcal{H}(R)$ and $\mathcal{H}(\Rring)$ is the number of 1-connected components of he Dynkin diagram of $R$ and, respectively, $\Rring$.

\subsection{} 
Let $A$ be an indecomposable affine Cartan matrix of rank $n$ and let $R$ the corresponding affine root system.  From Proposition \ref{lemma1} we know that the affine generators $\varTheta_{01}, \varTheta_{02}, \varTheta_{03}$ of $\B(\DDDot{X}_n)$ correspond to the elements
\begin{equation}
T_0, ~T_0^{-1}X_{\a_0^\vee}, ~X_{\ph^\vee}\Phi^{-1} \in \Atilde(X_n^{(1)}).
\end{equation}
Recall that $\th=\ph$ for root systems of type $X_n^{(1)}$. 

Similarly, from Proposition \ref{prop1} and Theorem \ref{mainA2n2}  the affine generators $\Ti, \Tiii$ of $\B(\fH{X}_n)$ or $\B(\DDot{C}_n)$ correspond to the elements
\begin{equation}
T_0, ~X_{\ph^\vee}\Phi^{-1}
\end{equation}
in the double affine Artin group of the corresponding twisted affine root system. 

Therefore, aside from the quadratic relations imposed on the elements $T_i$, $1\leq i\leq n$, we impose the following quadratic relations
\begin{subequations}
\begin{alignat}{2}
 &T_0-T_0^{-1}=t_{01}^{\frac{1}{2}}-t_{01}^{-\frac{1}{2}},\\ \label{eq: q02}
& T_0^{-1}X_{\a_0^\vee}- X_{-\a_0^\vee}T_0=t_{02}^{\frac{1}{2}}-t_{02}^{-\frac{1}{2}},\\
&X_{\ph^\vee}\Phi^{-1} - \Phi X_{-\ph^\vee}=t_{03}^{\frac{1}{2}}-t_{03}^{-\frac{1}{2}}, \label{eq: q03}
\end{alignat}
\end{subequations}
for untwisted affine root systems, and
\begin{subequations}
\begin{alignat}{2}
&T_0-T_0^{-1}=\th_0^{\frac{1}{2}}-\th_0^{-\frac{1}{2}},\\
&X_{\ph^\vee}\Phi^{-1} - \Phi X_{-\ph^\vee}=\phi_{0}^{\frac{1}{2}}-\phi_{0}^{-\frac{1}{2}}, \label{eq: qph0}
\end{alignat}
\end{subequations}
for twisted affine root systems. 

It is important to note that, by Remark \ref{rem: conj03} and \eqref{DCvsDA}, the relation \eqref{eq: q03} is superfluous unless $R$ is of type $C_n^{(1)}$, $n\geq 1$, $A_{2n}^{(2)}$, $n\geq 1$, $A_{2n-1}^{(2)}$, $n\geq 3$, or $D_3^{(2)}$. In these cases, the underlying finite Dynkin diagram is of type $A_1$ or $C_n$, $n\geq2$. With the finite nodes labelled according to our convention in \S\ref{sec: cl}, it can be seen from Proposition \ref{first-presentation} that the elements $X_{\ph^\vee}\Phi^{-1}$ and $X_{\a_n^\vee}T_n^{-1}$ are conjugate inside $\Atilde(R)$ by an element of $\A(\Rring)$. Therefore, for $R$ of type $C_n^{(1)}$, $n\geq 1$ or $A_n^{(2)}$, $n\geq 2$, the relations \eqref{eq: q03} and \eqref{eq: qph0} can be, respectively,  replaced by
\begin{subequations}
\begin{alignat}{2}
&X_{\a_n^\vee}T_n^{-1} - T_nX_{-\a_n^\vee}=t_{03}^{\frac{1}{2}}-t_{03}^{-\frac{1}{2}},  \quad \text{for } C_n^{(1)}, ~n\geq 1\label{eq: q03bis} \\
&X_{\a_n^\vee}T_n^{-1} - T_nX_{-\a_n^\vee}=\phi_{0}^{\frac{1}{2}}-\phi_{0}^{-\frac{1}{2}}, \quad \text{for } A_{2n}^{(2)},~ n\geq 1,~ A_{2n-1}^{(2)},~ n\geq 3 \label{eq: qph0bis}\\
&X_{\a_1^\vee}T_1^{-1} - T_1X_{-\a_1^\vee}=\phi_{0}^{\frac{1}{2}}-\phi_{0}^{-\frac{1}{2}}, \quad \text{for } D_{3}^{(2)}.\label{eq: qph0d3}
\end{alignat}
\end{subequations}
Also by Remark \ref{rem: conj03}, the relation \eqref{eq: q02} is superfluous unless $R$ is of type $C_n^{(1)}$, $n\geq 1$.
\subsection{} We are now ready to provide a description of the generic double affine Hecke algebras. 
\begin{Thm}\label{thm: cherednik-presentation}
The generic double affine Hecke algebra $\bH(R)$ is the $\F$-algebra generated by the affine Hecke algebra $\mathcal{H}(R)$, the lattice $\Q_X$  such that the following relations, called Bernstein-Lusztig relations, are satisfied for all  $\b\in Q^\vee$:

For $R$  of type $X_n^{(1)}\neq C_n^{(1)}$, $n\geq 1$
\begin{subequations} 
 \begin{alignat}{2} 
& T_iX_\b-X_{s_i(\b)} T_i = (t_{i}^{\frac{1}{2}}-t_{i}^{-\frac{1}{2}} )\frac{X_\b-X_{s_i(\b)}}{1-X_{-\a_i^\vee}}, \quad 1\leq i\leq n, \\
& T_0X_\b-X_{s_0(\b)} T_0 = (t_{01}^{\frac{1}{2}}-t_{01}^{-\frac{1}{2}} )\frac{X_\b-X_{s_0(\b)}}{1-X_{-\a_0^\vee}};
\end{alignat}
\end{subequations}

For $R$  of type $C_n^{(1)}$, $n\geq 1$
\begin{subequations} 
 \begin{alignat}{2} 
& T_iX_\b-X_{s_i(\b)} T_i = (t_{i}^{\frac{1}{2}}-t_{i}^{-\frac{1}{2}} )\frac{X_\b-X_{s_i(\b)}}{1-X_{-\a_i^\vee}}, \quad 1\leq i\leq n-1,  \\
& T_nX_\b-X_{s_n(\b)} T_n = \left((t_{n}^{\frac{1}{2}}-t_{n}^{-\frac{1}{2}}) +(t_{03}^{\frac{1}{2}}-t_{03}^{-\frac{1}{2}})X_{-\a_n^\vee} \right)\frac{X_\b-X_{s_n(\b)}}{1-X_{-2\a_n^\vee}}, \\
& T_0X_\b-X_{s_0(\b)} T_0 = \left((t_{01}^{\frac{1}{2}}-t_{01}^{-\frac{1}{2}}) +(t_{02}^{\frac{1}{2}}-t_{02}^{-\frac{1}{2}})X_{-\a_0^\vee} \right)\frac{X_\b-X_{s_0(\b)}}{1-X_{-2\a_0^\vee}} ;
\end{alignat}
\end{subequations}

For $R$ of type $D_{n+1}^{(2)}$, $n\geq 3$, $E_6^{(2)}$, $D_4^{(3)}$
\begin{subequations} 
 \begin{alignat}{2} 
& T_iX_\b-X_{s_i(\b)} T_i = (t_{i}^{\frac{1}{2}}-t_{i}^{-\frac{1}{2}} )\frac{X_\b-X_{s_i(\b)}}{1-X_{-\a_i^\vee}}, \quad 1\leq i\leq n, \\
& T_0X_\b-X_{s_0(\b)} T_0 = (\th_{0}^{\frac{1}{2}}-\th_{0}^{-\frac{1}{2}} )\frac{X_\b-X_{s_0(\b)}}{1-X_{-\a_0^\vee}};
\end{alignat}
\end{subequations}

For $R$ of type $A_{n}^{(2)}$, $n\geq 2$
\begin{subequations} 
 \begin{alignat}{2} 
& T_iX_\b-X_{s_i(\b)} T_i = (t_{i}^{\frac{1}{2}}-t_{i}^{-\frac{1}{2}} )\frac{X_\b-X_{s_i(\b)}}{1-X_{-\a_i^\vee}}, \quad 1\leq i\leq n-1, \\
& T_nX_\b-X_{s_n(\b)} T_n = \left((t_{n}^{\frac{1}{2}}-t_{n}^{-\frac{1}{2}}) +(\phi_{0}^{\frac{1}{2}}-\phi_{0}^{-\frac{1}{2}})X_{-\a_n^\vee} \right)\frac{X_\b-X_{s_n(\b)}}{1-X_{-2\a_n^\vee}}, \\
& T_0X_\b-X_{s_0(\b)} T_0 = (\th_{0}^{\frac{1}{2}}-\th_{0}^{-\frac{1}{2}} )\frac{X_\b-X_{s_0(\b)}}{1-X_{-\a_0^\vee}};
\end{alignat}
\end{subequations}

For $R$ of type $D_3^{(2)}$

\begin{subequations} 
 \begin{alignat}{2} 
& T_1X_\b-X_{s_1(\b)} T_1 = \left((t_{1}^{\frac{1}{2}}-t_{1}^{-\frac{1}{2}}) +(\phi_{0}^{\frac{1}{2}}-\phi_{0}^{-\frac{1}{2}})X_{-\a_1^\vee} \right)\frac{X_\b-X_{s_1(\b)}}{1-X_{-2\a_1^\vee}}, \\ 
& T_2X_\b-X_{s_2(\b)} T_i = (t_{2}^{\frac{1}{2}}-t_{2}^{-\frac{1}{2}} )\frac{X_\b-X_{s_2(\b)}}{1-X_{-\a_2^\vee}}, \\
& T_0X_\b-X_{s_0(\b)} T_0 = (\th_{0}^{\frac{1}{2}}-\th_{0}^{-\frac{1}{2}} )\frac{X_\b-X_{s_0(\b)}}{1-X_{-\a_0^\vee}}.
\end{alignat}
\end{subequations}
\end{Thm}

\begin{proof}
It is straightforward to check that if any of the Bernstein-Lusztig relations holds for $\b$ in a set $S$, then it holds for $\b$ in the lattice spanned by $S\cup s_i(S)$. Recall the notation of Proposition \ref{first-presentation} and Remark \ref{rem: lattices}. Modulo the ideal generated by the quadratic relations, for a fixed $T_i$, any relation \eqref{eq1} and \eqref{eq2} for  $\mu_j$ is equivalent to the corresponding Bernstein-Lusztig relation for $\mu_j$. Therefore, modulo the ideal generated by the quadratic relations, for a fixed $T_i$ the relations \eqref{eq1} and \eqref{eq2} are equivalent to the corresponding Bernstein-Lusztig relation for $T_i$ and $\eala{M_i}$. For $i$ such that $\eala{M}_i=Q^\vee$ the statement is proved. 

To conclude our statement for the cases when $\eala{M}_i\subset Q^\vee$ it is sufficient to verify the Bernstein-Lusztig  relations for $T_i$ and $\a_i^\vee$. These cases are listed in Remark \ref{rem: lattices}, and they are precisely the cases for which the quadratic relations \eqref{eq: q02}, \eqref{eq: q03bis}, \eqref{eq: qph0} are not superfluous. In all these cases, it is clear that the quadratic relation \eqref{eq: q02}, \eqref{eq: q03bis}, \eqref{eq: qph0bis}, \eqref{eq: qph0d3} is equivalent to the corresponding Bernstein-Lusztig relation for  $T_i$ and $\a_i^\vee$.
\end{proof}

\appendix
\section{Combinatorial results}\label{combinatorics}

\subsection{Braid relations}
We record here some elementary computations involving braid relations.
\begin{Lm}[\cite{ISTri}*{Lemma 2.1}]\label{basic-braid-rels1}
Let $a,b,c$ be elements of a fixed group. Assume that $a$ and $b$ satisfy the $1$-braid relation and $a$ and $c$ satisfy the $1$-braid relation. The following are equivalent
\begin{enumerate}[label={\roman*)}]
\item $a$ and $c^{-1}bc$ satisfy the $1$-braid relation;  
\item $a$ and $bc$ satisfy the $2$-braid relation.
\end{enumerate}
Furthermore, if the equivalent conditions also hold $a$ and $(bc)^{-1}c(bc)$ satisfy the $1$-braid relation.
\end{Lm}
\begin{Lm}\label{basic-braid-rels2}
Let $a,b,c$ be elements of a fixed group.
\begin{enumerate}[label={\roman*)}]
\item If $a$ and $b$ satisfy the $1$-braid relation, then $a^p$ and $b$ satisfy the $p$-braid relation for $0\leq p\leq 4$.
\item If $a$ and $b$ satisfy the braid relations specified by the diagram
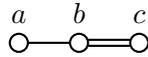
\begin{center}
  \begin{tikzpicture}[scale=.4]
    \draw (0,.3) node[anchor=south]  {$a$};
    \draw (2,.3) node[anchor=south]  {$b$};
    \draw (4,.3) node[anchor=south]  {$c$};
    \foreach \x in {0,...,2}
    \draw[xshift=\x cm,thick] (\x cm,0) circle [radius=3mm];
    \draw[xshift=0.15 cm,thick] (0.15 cm,0) -- +(1.4 cm,0);
    \draw[thick] (2.3 cm, .1 cm) -- +(1.4 cm,0);
    \draw[thick] (2.3 cm, -.1 cm) -- +(1.4 cm,0);
  \end{tikzpicture}
\end{center}
then the following pairs satisfy the $2$-braid relation: $a$ and $bcb$, $a$ and $bcb^{-1}$, $a$ and $b^{-1}cb$.
\item If $a$ and $b$ satisfy the braid relations specified by the diagram
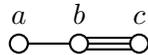
\begin{center}
  \begin{tikzpicture}[scale=.4]
    \draw (0,.3) node[anchor=south]  {$a$};
    \draw (2,.3) node[anchor=south]  {$b$};
    \draw (4,.3) node[anchor=south]  {$c$};
    \foreach \x in {0,...,2}
    \draw[xshift=\x cm,thick] (\x cm,0) circle [radius=3mm];
    \draw[xshift=2cm, thick] (40: 3mm) -- +(1.545 cm, 0);
    \draw[xshift=2cm, thick] (-40: 3 mm) -- +(1.545 cm, 0);
            \foreach \y in {0,1}
    \draw[thick,xshift=\y cm] (\y cm,0) ++(.3 cm, 0) -- +(14 mm,0);
  \end{tikzpicture}
\end{center}
then the following pairs satisfy the $3$-braid relation: $a$ and $bcb^{-1}$, $a$ and $b^{-1}cb$.
\end{enumerate}
\end{Lm}
\begin{proof}
Straightforward verification.
\end{proof}

\subsection{The length function}\label{length-combinatorics}
We first establish some notation. For each $w$ in $W$ let $\ell(w)$ be the length of a reduced expression of $w$ in terms of the simple reflections $s_i$, $0\leq i\leq n$.
We have
\begin{equation}\label{length}
\ell(w)=|\Pi(w)|
\end{equation}
where
\begin{equation}\label{pi}
\Pi(w)=\{\a\in R^+\ |\ w(\a)\in R^-\}\ .
\end{equation}
For the basic properties of the length function on Coxeter groups we refer to \cite{BouLie}*{Chapter 4}. We record below some well-known facts that will be used in what follows.

If $w=s_{j_p}\cdots s_{j_1}$ is a reduced expression, then
\begin{equation}\label{pi-description}
\Pi(w)=\{\a^{(k)}\ |\ 1\leq k\leq p\},
\end{equation}
with $\a^{(k)}=s_{j_1}\cdots s_{j_{k-1}}(\a_{j_k})$.
\begin{Lm} \label{pi-simple}
Let $w\in W$ and $0\leq i\leq n$. The following are equivalent
\begin{enumerate}[label={\roman*)}]
\item $w$ has a reduced expression ending in $s_i$;
\item $\ell(ws_i)=\ell(w)-1$;
\item $\a_i\in\Pi(w)$.
\end{enumerate}
\end{Lm}
We will need the following extension of the above result.
\begin{Lm}\label{pi-simple2}
Let $w,w_1,w_2\in W$ such that $w=w_1w_2$. The following are equivalent
\begin{enumerate}[label={\roman*)}]
\item $\ell(w_1w_2)=\ell(w_1)+\ell(w_2)$;
\item $\Pi(w_2)\subseteq\Pi(w)$.
\end{enumerate}
\end{Lm}
\begin{proof}  Let $w_2=s_{j_t}\cdots s_{j_1}$ and $w_1=s_{j_p}\cdots s_{j_{t+1}}$ be reduced expressions.

Let us assume that $\ell(w_1w_2)=\ell(w_1)+\ell(w_2)$.  Then, 
$w=s_{j_p}\cdots s_{j_1}$ is a reduced expression and the desired conclusion follows from \eqref{pi-description}.

We will show by induction on $\ell(w_2)$ that  if  $\Pi(w_2)\subseteq\Pi(w)$  then $\ell(w_1w_2)=\ell(w_1)+\ell(w_2)$. If $\ell(w_2)=1$ then the claim follows from Lemma \ref{pi-simple}. If $\ell(w_2)>1$ then, by  hypothesis, 
$$
\Pi(s_{j_{t-1}}\cdots s_{j_1})\subset \Pi(w_2)\subseteq \Pi(w)
$$
which in turn implies,  by the induction hypothesis, that 
$$
\ell(w)=\ell(w_1s_{j_t})+\ell(w_2)-1.
$$
From \eqref{pi-description} we obtain that $s_{j_1}\cdots s_{j_{t-1}}(\a_{j_t})\in \Pi(w)$ or, equivalently, $w_1(-\a_{j_t})\in R^-$. Therefore, $w_1(\a_{j_t})\in R^+$ and, by Lemma \ref{pi-simple},  $\ell(w_1s_{j_t})=\ell(w_1)+1$, from which our claim immediately follows.
\end{proof}
\begin{Lm}\label{lemma-simple-commuting}
Let $\gamma\in \Rring^+$ and $\a_i$ a simple root such that $(\a_i,\gamma)=0$. Then, 
$$\ell(s_\gamma s_i)=\ell(s_i s_\gamma)=\ell(s_\gamma)+1.$$
\end{Lm}
\begin{proof}
Straightforward from Lemma \ref{pi-simple}.
\end{proof}

\subsection{Non-simply laced finite root systems}\label{app-nsl}

In what follows we investigate the length of some elements of $\W$ in the case when $\Rring$ is a non-simply laced root system. We follow the notation set up in Section \ref{twisted} and denote by $\th$ the short dominant root and by $\ph$ the long dominant root. Also, we denote by ${i_\th}$ node corresponding to the unique simple root that is not orthogonal on $\th$ and by ${i_\ph}$ the node corresponding unique simple root that is not orthogonal on $\ph$. Recall that $\php^\vee=\th^\vee-\ph^\vee$ and $\thp=\ph-\th$. 

\begin{Lm}\label{lemma-tildevssimpleroots}
With the above notation we have
\begin{enumerate}[label={\roman*)}]
\item If $(\th, \a_{i_\th}^\vee)(\th^\vee,\a_{i_\th})=2$ then $\php=\a_{i_\th}$;
\item If $(\ph, \a_{i_\ph}^\vee)(\ph^\vee,\a_{i_\ph})=2$ then $\thp=\a_{i_\ph}$.
\end{enumerate}
\end{Lm}
\begin{proof}
Remark that 
\begin{align*}
(s_{i_\th}(\php^\vee), \th) &= (\ph^\vee, -s_\th s_{i_\th}(\th)) & &\\
&= (\ph^\vee, \th+(\th,\a_{i_\th}^\vee)\a_{i_\th}-(\th,\a_{i_\th}^\vee)(\th^\vee,\a_{i_\th})\th) & &   \\
&= -1. & &
\end{align*}
Since, $\th$ is dominant $s_{i_\th}(\php)$ must be a negative root and, keeping in mind that $\a_{i_\th}$ is a simple root, we must have $\php=\a_{i_\th}$. The second claim is proved in a similar fashion.
\end{proof}

\begin{Lm}\label{basis-alternative}
The set $\{\ph,\th\}\cup \{\a_i~|~1\leq i\leq n, ~i\neq i_\th,i_\ph\}$ is a basis for $\Hring^*$.
\end{Lm}
\begin{proof}
If $n=2$ then the claim follows from the fact that $\ph$ and $\th$ are linearly independent. For $n\geq 3$, remark that $(\php, \ph)=(\thp,\th)=0$. If
$$
c_\ph\ph+c_\th\th+\sum_{i\neq i_\th,i_\ph}c_i\a_i=0, 
$$
then, by taking the scalar product with $\thp$ and, respectively, $\php$ we obtain that $c_\ph=c_\th=0$ and, consequently  all $c_i$ are also zero. Therefore the set from the statement is linearly independent and, since it has $n$ elements, it is a basis for $\Hring^*$.
\end{proof}

\begin{Lm}\label{lemma-s}
We have
\begin{enumerate}[label={\roman*)}]
\item $\ell(s_\th)=\ell(s_\ph s_\th)+\ell(s_{\php})=\ell(s_{\php})+\ell(s_\th s_\ph)$;
\item $\ell(s_\ph)=\ell(s_\th s_\ph)+\ell(s_{\thp})=\ell(s_{\thp})+\ell(s_\ph s_\th )$.
\end{enumerate}
\end{Lm}
\begin{proof}  We will provide an argument for the first claim. The second claim follows from the first claim applied to the root system $\Rring^\vee$. Also note that  since $s_\ph s_\th$ is the inverse of $s_\th s_\ph$, we have $\ell(s_\ph s_\th)=\ell(s_\th s_\ph)$. Therefore we only need to show that $\ell(s_\th)=\ell(s_\th s_\ph)+\ell(s_{\php})$. Remark that $s_\th=s_{\php} (s_\th s_\ph)$ so, by Lemma \ref{pi-simple2}, it would be enough to show that
$$
\Pi(s_\th s_\ph)\subset \Pi(s_\th).
$$
Let $\a\in \Pi(s_\th s_\ph)$. Keeping in mind that $s_\th s_\ph(\a)\in \Rring^-$ and that $\th, \ph$ are dominant we obtain that
$$
(\a,\ph-\th)=(s_\th s_\ph(\a),\th)\leq 0.
$$
If $(\a,\th)=0$ this implies that $(\a,\ph)=0$ and therefore $s_\th s_\ph(\a)=\a$ which is contradicting the fact that $\a$ is a positive root. We infer that necessarily 
$(\a,\th)>0$, and $(s_\th(\a), \th)<0$ which shows that $\a\in \Pi(s_\th)$.
\end{proof}
\begin{Lm}\label{lemma-wy}
There exist  $x\in\stab_{\W}(\th)$ and $y\in\stab_{\W}(\ph)$ of order two such that the following hold
\begin{enumerate}[label={\roman*)}]
\item $s_\th x=s_\ph y$ and $\ell(s_\th x)=\ell(s_\ph y)=\ell(s_\th)+\ell(x)=\ell(s_\ph)+\ell(y)$;
\item $s_\ph s_\th=yx$ and $\ell(s_\ph s_\th)=\ell(y)+\ell(x)$;
\item $s_\th=y s_{\thp} y$ and $\ell(s_\th)=2\ell(y)+\ell(s_{\thp})$; 
\item $s_\ph=x s_{\php} x$ and $\ell(s_\ph)=2\ell(x)+\ell(s_{\php})$;
\item $\Pi(y)\subseteq \{\b\in \Rring^+~|~(\thp,\b^\vee)=-1\}$; 
\item $\Pi(x)\subseteq \{\b\in \Rring^+~|~(\php^\vee,\b)=-1\}$.
\end{enumerate}
\end{Lm}
\begin{proof} Let $v_\circ$ be the longest length element in the Coxeter group $\stab_{\W}(\th)\cap\stab_{\W}(\ph)$ and let $x,y\in \W$ such that
\begin{equation}\label{eq8}
v_\circ w_\circ=s_\th x= s_\ph y.
\end{equation}
The fact that $x,y$ have order two,  that  $x\in\stab_{\W}(\th)$ and $y\in\stab_{\W}(\ph)$, and that  $s_\ph s_\th=yx$, $s_\th s_\ph=xy$ are straightforward verifications. We show that
$\Pi(x)\subset \Pi(v_\circ w_\circ)$. Indeed, if $\a\in \Pi(x)$, then $\a$ and  $x(\a)$ are orthogonal on $\th$, hence $$v_\circ w_\circ (\a)=s_\th (x(\a))=x(\a)\in \Rring^-.$$
From Lemma \ref{pi-simple2} we obtain that $\ell(s_\th x)=\ell(s_\th)+\ell(x)$. Similarly, we obtain that $\ell(s_\ph y)=\ell(s_\ph)+\ell(y)$.

We will now show that $\ell(s_\ph s_\th)=\ell(y)+\ell(x)$. By Lemma \ref{pi-simple2} we have to argue that $\Pi(x)\subset \Pi(s_\ph s_\th)$. Indeed, if $\a\in \Pi(x)$, then $\a$ is orthogonal on $\th$ but it is not orthogonal on $\ph$ (otherwise $v_\circ w_\circ(\a)=v_\circ(-\a)\in \Rring^+$, in contradiction with $\a\in\Pi(v_\circ w_\circ)$). Therefore, $$s_\ph s_\th(\a)=s_\ph(\a)=\a-(\a,\ph^\vee)\ph\in \Rring^-,$$
which proves our claim.

For the remaining claims, remark that \eqref{eq8} implies that $y(\thp)=\th$ and $x(\php)=\ph$. Therefore, we have $s_\th=y s_{\thp}  y$ and $s_\ph=x s_{\php}  x$. Furthermore, $s_\th=(s_\ph s_\th)s_{\php}=(yx)s_{\php}$, hence $xs_{\php}=s_{\thp}y$ and 
\begin{align*}
\ell(y)+\ell(xs_{\php})&\geq \ell(s_\th)=\ell(s_\ph s_\th)+\ell(s_{\php}) & & \text{by Lemma \ref{lemma-s}i)}\\
&=\ell(y)+\ell(x)+\ell(s_{\thp}) & & \text{by Lemma \ref{lemma-wy}ii)}\\
&\geq \ell(y)+\ell(xs_{\php})
\end{align*}
We obtain that 
\begin{subequations} 
 \begin{alignat}{2} \label{eq9}
& \ell(s_\th)=\ell(y)+\ell(xs_{\php})=\ell(y)+\ell(s_{\thp}y) , \\ \label{eq10}
& \ell(s_{\php}x)=\ell(xs_{\php})=\ell(x)+\ell(s_{\php}).
\end{alignat}
\end{subequations}
In a similar fashion, we obtain that
\begin{subequations} 
 \begin{alignat}{2} \label{eq11}
& \ell(s_\ph)=\ell(x)+\ell(ys_{\thp})=\ell(x)+\ell(s_{\php}x) , \\ \label{eq12}
& \ell(s_{\thp}y)=\ell(ys_{\thp})=\ell(y)+\ell(s_{\thp}).
\end{alignat}
\end{subequations}
Now, our third claim follows from \eqref{eq9} and \eqref{eq12} and the fourth claim follows from \eqref{eq10} and \eqref{eq11}. We remark that iii) is equivalent to the fact that 
$y$ is the shortest element of $\W$ such that $y(\thp)=\th$ and iv) is equivalent to the fact that  $x$ is the shortest element of $\W$ such that $x(\php)=\ph$.

Let now $y=s_{j_p}\cdots s_{j_1}$ be a reduced expression and for $1\leq k\leq p$ let $$\a^{(k)}=s_{j_1}\cdots s_{j_{k-1}}(\a_{j_k}).$$ From \eqref{pi-description} we know that $\Pi(y)=\{\a^{(k)}~|~1\leq k\leq p\}$. Since $y$ is the shortest element of $\W$ such that $y(\thp)=\th$ we obtain that 
$$
(\thp,\a^{(k)\vee})=(s_{j_{k-1}}\cdots s_{j_1}(\thp),\a_{i_k}^\vee)\neq 0.
$$
On the other hand, 
$$
(\thp,\a^{(k)\vee})=(y(\th), \a^{(k)\vee})=(\th, y(\a^{(k)\vee}))\leq 0.
$$
Therefore, $(\thp,\a^{(k)\vee})<0$ and since $\th$ is short root we obtain that $(\thp,\a^{(k)\vee})=-1$ and this proves part v). The argument for part vi) is entirely similar.
\end{proof}
In fact, we can identify the elements $x$ and $y$ from Lemma \ref{lemma-wy}. 
\begin{Lm}\label{lemma-a6}
If $\Rring$ is doubly-laced then $x=s_{\thp}$ and $y=s_{\php}$. Furthermore, 
\begin{enumerate}[label={\roman*)}]
\item If $(\th, \a_{i_\th}^\vee)(\th^\vee,\a_{i_\th})=2$ then $y=s_{i_\th}$;
\item If $(\ph, \a_{i_\ph}^\vee)(\ph^\vee,\a_{i_\ph})=2$ then $x=s_{i_\ph}$.
\end{enumerate}
\end{Lm}
\begin{proof}
Let $v_\circ$ be the longest length element in the Coxeter group $\stab_{\W}(\th)\cap\stab_{\W}(\ph)$. From the proof of Lemma \ref{lemma-wy} it follows that 
it would be enough to show that $$v_\circ w_\circ=s_\th s_\ph s_\th s_\ph = s_\ph s_\th s_\ph s_\th.$$
In turn, this equality can be directly verified on the basis of $\Hring^*$ described in Lemma \ref{basis-alternative}. The remaining statements follow directly from Lemma \ref{lemma-tildevssimpleroots}.
\end{proof}

\begin{Lm}\label{lemma-a7}
If $\Rring$ is triply-laced then $x=s_{i_\ph}$ and $y=s_{i_\th}$.
\end{Lm}
\begin{proof}
In this case $\Rring$ is of rank two and therefore $\a_{i_\th}, \a_{i_\ph}$ are the only simple roots. Our claims follow from the fact that $\stab_{\W}(\th)$ is the group generated by $s_{i_\ph}$ and $\stab_{\W}(\ph)$ is the group generated by $s_{i_\th}$.
\end{proof}
\subsection{Relations in finite Artin groups}\label{app: finite}
We record some useful relations inside the Artin group $\A_{\W}$. 

\begin{Lm}\label{lemma-simple-commuting-artin}
Let $\gamma\in \Rring^+$ and $\a_i$ a simple root such that $(\a_i,\gamma)=0$. Then, 
$$T_{s_\gamma}T_i=T_i T_{s_\gamma}$$
\end{Lm}
\begin{proof}
Straightforward from Lemma \ref{lemma-simple-commuting}.
\end{proof}
In all the statements that follow $\Rring$ is necessarily non-simply laced and we continue to use the notation in Section \ref{app-nsl}.

\begin{Lm}\label{lemma-extrabraid-1}
Let $\gamma\in \Rring^+$ and $\a_i$ a simple root such that $(\a_i,\gamma^\vee)(\a_i^\vee, \gamma)=2$ and $(\a_i,\gamma^\vee)>0$. Then, 
$$T_{s_\gamma}=T_{i}T_{s_i s_\gamma s_i}T_{i}.$$
\end{Lm}
\begin{proof}
Remark that the hypothesis implies that $$s_\gamma(\a_i)=\a_i-(\a_i,\gamma^\vee)\gamma\in \Rring^-,$$ which means that $\a_i\in \Pi(s_{\gamma})$. From Lemma \ref{pi-simple} we obtain that $\ell(s_i s_{\gamma})=\ell(s_{\gamma}s_i)=\ell(s_{\gamma})-1$. Remark also that 
$$(s_\gamma(\a_i),\a_i^\vee)=2-(\a_i,\gamma^\vee)(\gamma,\a_i^\vee)=0$$
and consequently $(s_i s_{\gamma})(\a_i)=s_{\gamma}(\a_i)\in \Rring^-$. Therefore, again by Lemma \ref{pi-simple}, we obtain that 
$$
\ell(s_i s_{\gamma} s_i)=\ell(s_i s_{\gamma})-1=\ell(s_{\gamma})-2, 
$$
which in turn implies that 
$$
T_{s_i} T_{s_i s_{\gamma} s_i} T_{s_i}=T_{s_\gamma},
$$
which is our claim.
\end{proof}
\begin{Lm}\label{lemma-extrabraid-2}
If $\Rring$ is a double-laced root system, then
\begin{enumerate}[label={\roman*)}]
\item If $(\th, \a_{i_\th}^\vee)(\th^\vee,\a_{i_\th})=2$ then $T_{i_\th}$ and $T_{s_{i_\th} s_\th s_{i_\th}}$ satisfy the $2$-braid relation;
\item If $(\ph, \a_{i_\ph}^\vee)(\ph^\vee,\a_{i_\ph})=2$ then $T_{i_\ph}$ and $T_{s_{i_\ph} s_\ph s_{i_\ph}}$ satisfy the $2$-braid relation.
\end{enumerate}
\begin{proof}
The second statement is identical to the first one for the dual root system. We prove the first statement by induction on $n\geq 2$. 

Remark first that $(s_{i_\th}(\th), \th^\vee)=2-(\th, \a_{i_\th}^\vee)(\th^\vee,\a_{i_\th})=0$ and therefore $s_{i_\th}(\th)$ is a root in the standard parabolic sub-root system of $\Rring$ consisting of roots orthogonal on $\th$, more precisely the parabolic sub-root system obtained by removing $\a_{i_\th}$. In fact it is a dominant root in this parabolic subsystem: if $\a_i\neq \a_{i_\th}$ is a simple root then 
$$
(s_{i_\th}(\th), \a_i^\vee)=-(\th, \a_{i_\th}^\vee)(\a_i^\vee,\a_{i_\th})\geq 0.
$$ 
For $n=2$, $\a_{i_\th}$ and $s_{i_\th}(\th)$ are the two simple roots and claim is satisfied by the definition of the Coxeter braid group. For $n>2$, $s_{i_\th}(\th)$ is the short dominant root in the parabolic sub-root system obtained by removing $\a_{i_\th}$, which is a double-laced root system. In particular, $s_{i_\th}(\th)$ is not a simple root. Let $\a_i$ be the unique simple root in $\Rring$ for which $(\a_{i_\th},\a_i)\neq 0$. Then,
$$
2\geq (s_{i_\th}(\th), \a_i^\vee)(s_{i_\th}(\th)^\vee, \a_i)=(\th, \a_{i_\th}^\vee)(\th^\vee,\a_{i_\th})(\a_i, \a_{i_\th}^\vee)(\a_i^\vee,\a_{i_\th})=2(\a_i, \a_{i_\th}^\vee)(\a_i^\vee,\a_{i_\th}),
$$
and this forces $(s_{i_\th}(\th), \a_i^\vee)(s_{i_\th}(\th)^\vee, \a_i)=2$ and $(\a_i, \a_{i_\th}^\vee)(\a_i^\vee,\a_{i_\th})=1$. Applying the induction hypothesis for $\a_i$ and $s_{i_\th}(\th)$ we obtain that $T_i$ and $T_{s_i s_{i_\th} s_\th s_{i_\th} s_i}$ satisfy the $2$-braid relation and that $T_i$ and $T_{i_\th}$ satisfy the $1$-braid relation. Moreover, as before $s_i s_{i_\th}(\th)$ is a root in the standard parabolic sub-root system of $\Rring$ obtained by removing $\{\a_{i_\th},\a_i\}$ and all the simple roots in this parabolic sub-system are orthogonal on $\a_{i_\th}$. Therefore $T_{i_\th}$ and $T_{s_i s_{i_\th} s_\th s_{i_\th} s_i}$ commute.

We can apply Lemma \ref{basic-braid-rels2}ii) to infer that $T_{i_\th}$ and $T_i T_{s_i s_{s_{i_\th}(\th)} s_i} T_i$ satisfy the $2$-braid relations. But, by Lemma \ref{lemma-extrabraid-1}, $T_i T_{s_i s_{s_{i_\th}(\th)} s_i} T_i=T_{s_{s_{i_\th}(\th)}}$, which concludes the proof.
\end{proof}

\end{Lm}
\begin{Lm} \label{Psi}
We have,
\begin{enumerate}[label={\roman*)}]
\item $\Psi=\Phip\Theta^{-1}=\Phi^{-1}\Thetap$;
\item $\isP=\Theta^{-1}\Phip=\Thetap\Phi^{-1}$;
\item $\Thetap\Theta=\Phi\Phip$ and $\Phip\Phi=\Theta\Thetap$.
\end{enumerate}
\end{Lm}
\begin{proof}
From Lemma \ref{lemma-s} we obtain that 
$$
\Theta=T_{s_\ph s_\th} \Phip=\Phip T_{s_\th s_\ph}\quad\text{and}\quad \Phi=T_{s_\th s_\ph}\Thetap=\Thetap T_{s_\ph s_\th},
$$
from which our claims immediately follow.
\end{proof}
\begin{Lm}\label{lemma-wy-artin}
There exist  $x\in\stab_{\W}(\th)$ and $y\in\stab_{\W}(\ph)$ of order two such that
\begin{enumerate}[label={\roman*)}]
\item $\Phi^{-1}\Theta=T_yT_x^{-1}$ and $\Phi\Theta^{-1}=T_y^{-1}T_x$;
\item $\Psi=T_x^{-1}T_y^{-1}$ and $\isP=T_y^{-1}T_x^{-1}$;
\item $\Theta=T_y \Thetap T_y$ and $\Phi=T_x \Phip T_x$.
\end{enumerate}
Furthermore, 
\begin{enumerate}[resume, label={\roman*)}]
\item If $\Rring$ is doubly-laced then $T_x=\Thetap$ and $T_y=\Phip$;
\item If $\Rring$ is doubly-laced and $(\th, \a_{i_\th}^\vee)(\th^\vee,\a_{i_\th})=2$ then $T_y=T_{i_\th}$;
\item If $\Rring$ is doubly-laced and $(\ph, \a_{i_\ph}^\vee)(\ph^\vee,\a_{i_\ph})=2$ then $T_x=T_{i_\ph}$;
\item If $\Rring$ is triply-laced then $T_x=T_{i_\ph}$ and $T_y=T_{i_\th}$.
\end{enumerate}
\end{Lm}
\begin{proof}
Straightforward from Lemma \ref{lemma-wy}, Lemma \ref{lemma-a6},  and Lemma \ref{lemma-a7}.
\end{proof}
\begin{Lm}\label{lemma-explicit-2}
If $\Rring$ is doubly-laced then the following hold
\begin{enumerate}[resume, label={\roman*)}]
\item $\Theta=\Phip\Thetap\Phip$;
\item $\Phi=\Thetap\Phip\Thetap$;
\item $\Theta\Thetap=\Thetap\Theta=\Phi\Phip=\Phip\Phi=\Thetap\Phip\Thetap\Phip=\Phip\Thetap\Phip\Thetap$; In particular,  $\Thetap$ and $\Phip$ satisfy the $2$-braid relation;
\item $\Psi=\Thetap^{-1}\Phip^{-1}$ and $\isP=\Phip^{-1}\Thetap^{-1}$.
\end{enumerate}
\end{Lm}
\begin{proof}
The first two claims follow directly from Lemma \ref{lemma-wy-artin}iii-iv). The third claim follows from the first two claims and \ref{lemma-wy}i). The fourth claim is a consequence of Lemma \ref{lemma-wy-artin}ii).
\end{proof}
\subsection{Relations in affine Artin groups}\label{affine-rels}
In this section $R$ is an irreducible twisted affine root system not of type $A_{2n}^{(2)}$. In particular, $\Rring$ is non-simply laced and we follow the notation in Section \ref{app-nsl}. All statements refer to the affine Artin group $\A(R)$.

\begin{Lm}\label{lemma-affine-commutation}
With the notation above, we have 
\begin{enumerate}[label={\roman*)}]
\item $Y_{-\thp}= \Phi\Theta^{-1}Y_{-\th}\Psi =\isP Y_{-\th}\Theta^{-1}\Phi$;
\item $T_{i_\th}$ and $\Theta T_0 \Phi T_0 \Psi=\Phi T_0 \Psi\Theta T_0$ commute.
\end{enumerate}
\end{Lm}
\begin{proof} Let $y,w\in \W$ as in Lemma \ref{lemma-wy}. Then,
$$
\Phi\Theta^{-1}Y_{-\th}\Psi =\isP Y_{-\th}\Theta^{-1}\Phi=T_y^{-1}Y_{-\th}T_y^{-1}.
$$
Let $y=s_{j_p}\cdots s_{j_1}$ be a reduced expression. Then, from Lemma \ref{lemma-wy}v) we know that for any $1\leq k\leq p$ we have $$(s_{j_{k-1}}\cdots s_{j_1}(-\thp),\a_{i_k}^\vee)=(-\thp, s_{j_1}\cdots s_{j_{k-1}}(\a_{j_k}^\vee))=1.$$
Therefore, $T_y Y_{-\thp} T_y=Y_{-\th}$, which proves our first claim.

For the second claim, remark that $$\Theta T_0 \Phi T_0 \Psi=Y_{-\th}Y_{-\thp}=Y_{-\ph}=Y_{-\thp}Y_{-\th}=\Phi T_0 \Psi\Theta T_0$$ and that $T_{i_\th}$ and $Y_{-\ph}$ commute by Proposition \ref{bernstein-presentation}.
\end{proof}
\section{Non-Coxeter groups} \label{sec: nonCoxeter}

\subsection{} One of the common perceptions about double affine Weyl groups is that they fall outside the framework of Coxeter groups. We provide here a proof of this fact as well as some related results.

Let us start by recalling the following well known result \cite{BouLie}*{Ch. V, \S4, Ex. 3b)}.
\begin{Lm}\label{lm: center}
Let $G$ be an infinite irreducible Coxeter group of finite rank. Then $G$ has trivial center.
\end{Lm}

\begin{Prop}\label{prop: nonCoxeter}
A double affine Weyl group is not a Coxeter group.
\end{Prop}
\begin{proof}
For a contradiction, assume that the double affine Weyl group $\Wdaff$ is a Coxeter group. Being finitely generated, $\Wdaff$ must be a finite rank Coxeter group. By Lemma \ref{lm: center} the infinite irreducible components of $\Wdaff$ have trivial center. Therefore, the center of $\Wdaff$ is precisely the direct product of the centers of its finite irreducible components and, consequently, the center of $\Wdaff$ is a finite group. This contradicts the fact that the center of $\Wdaff$ is isomorphic to $\Z$.
\end{proof}

This argument raises the question of whether one can remove this basic obstruction by either taking a quotient of $\Wdaff$ by a nontrivial subgroup of its center, or by resolving the center. We show below that the former option is not viable (for somewhat more subtle reasons), leaving only the latter. The latter option does indeed materialize, leading precisely to our Coxeter-type presentation (see Remark \ref{rem: Cox} and Corollary \ref{cor: Cox}).

\subsection{}

Recall that a group is said to be indecomposable if it is not a direct product of proper subgroups. We record a mild variation of \cite{CHDec}*{Proposition 1} which will be useful for our purposes.

\begin{Prop}\label{prop: idec}
Let $G$ be group such that $G$ has a proper normal abelian subgroup $T$ with the property that $G/T$ acts faithfully by conjugation on any nontrivial $G/T$-invariant subgroup of $T$.
Then, 
\begin{enumerate}[label={\roman*)}]
\item $T$ is a maximal normal abelian subgroup of $G$;
\item $G$ is indecomposable.
\end{enumerate}
\end{Prop}
\begin{proof}
Let $N$ be a normal abelian subgroup of $G$. The commutator $[N,T]$ is a $G/T$-invariant subgroup of $T$ but also a subgroup of $N$. Since $N$ is abelian it acts trivially by conjugation on $[N,T]$. If $[N,T]$ is nontrivial then $N\subseteq T$. If $[N,T]$ is trivial then $N$ acts trivially by conjugation on $T$ and therefore $N\subseteq T$. Either way, $N\subseteq T$, showing that $T$ is a maximal normal abelian subgroup of $G$.

Assume that $G=G_1\times G_2$ and denote by $\pi_i$ the canonical projection $G\to G_i$, $i=1,2$. Let $T_i=\pi_i(T)$,  $i=1,2$. Remark that each $T_i$ is a normal abelian subgroup of $G_i$. Therefore, $T_1\times T_2$ is a normal abelian subgroup of $G$ that contains $T$. Hence, by part i) we must have $T=T_1\times T_2$. Without loss of generality we may assume that $T_1$ is nontrivial. Since $G_2$ acts trivially by conjugation on $T_1$ we deduce that $G_2\subseteq T$ and hence $G_2=T_2$. If $T_2$ is nontrivial then, as above, $G_1=T_1$ implying that $G=T$ which contradicts the fact that $T$ is a proper subgroup. Therefore, $T_2$ is trivial implying that $G_2$ is trivial.  This shows that $G$ is indecomposable.
\end{proof}
We now apply this criterion to double affine Weyl groups.
\begin{Prop}\label{prop: elliptic-indecomposable}
Let $\Wdaff$ be a double affine Weyl group. Then, $\Wdaff/Z(\Wdaff)$ is indecomposable.
\end{Prop}
\begin{proof}
With the notation of Proposition \ref{wdef},  let us denote by $H$ the group generated by $\{\l_\mu\}_{\mu\in M}$, $\{\tau_\b\}_{\b\in \Qring^\vee}$, and $\tau_\d$. Let $G$ denote $\Wdaff/Z(\Wdaff)$ and let $T$ denote $H/Z(\Wdaff)$. The group $T$ is a proper normal abelian subgroup of $G$ and the action of $G/T$ by conjugation on $T$ is isomorphic to the diagonal action of $\W$ on $\Qring^\vee\oplus \Qring^\vee$. The $\Re$-span inside $\Hring^*\oplus\Hring^*$ of any proper $\W$-invariant subgroup $\Qring^\vee\oplus \Qring^\vee$ contains a copy of the reflection representation of $\W$ which is faithful. Therefore, the action of $\W$ on any proper $\W$-invariant subgroup $\Qring^\vee\oplus \Qring^\vee$ is faithful. The claim now follows from Proposition \ref{prop: idec}.
\end{proof}

\subsection{} We are now ready to show that $\Wdaff/Z(\Wdaff)$ is not a Coxeter group. We will make use of the following result which is a direct consequence of \cite{HarGro}*{Corrollaire}.

\begin{Prop}\label{prop: growth}
Let $G$ be an irreducible finite rank Coxeter group. Then $G$ has a free, finitely generated abelian subgroup of finite index  if and only if $G$ is an affine Coxeter group.
\end{Prop}

\begin{Thm}
Let $\Wdaff$ be a double affine Weyl group and $C$ a subgroup of its center. Then, $\Wdaff/C$ is not a Coxeter group.
\end{Thm}
\begin{proof}
Assume that $G=\Wdaff/C$ is a Coxeter group. Since $G$ is finitely generated, it is a finite rank Coxeter group. If $G$ is irreducible and infinite, by Lemma \ref{lm: center}, it has trivial center and therefore $C=Z(\Wdaff)$. In this case, by Proposition \ref{prop: growth} $G$ is an irreducible affine Coxeter group. But irreducible affine Coxeter groups are uniquely determined by the rank of their maximal normal abelian subgroup and the corresponding quotient and it can be readily verified that $G$ cannot be an affine Coxeter group.

If $G$ is reducible as a Coxeter group then its decomposition into irreducible components induces a decomposition of $\Wdaff/Z(\Wdaff)$. Since $\Wdaff/Z(\Wdaff)$ is indecomposable, all but one irreducible component are subgroups of $Z(\Wdaff)/C$ and the remaining irreducible component is a finite rank infinite Coxeter group that either has nontrivial center (contradicting Lemma \ref{lm: center}), or it is isomorphic to $\Wdaff/Z(\Wdaff)$ (which we already showed not to be a Coxeter group). In conclusion, $G$ cannot be a Coxeter group. 
\end{proof}


\end{document}